\documentclass[11pt,a4paper]{article}

\usepackage[a4paper,top=1in,bottom=1in,left=1in,right=1in]{geometry} 
\usepackage{setspace}
\usepackage[authoryear, round]{natbib}

\usepackage{amsmath,amssymb,amsfonts,amsthm,bm,subcaption,mathrsfs,mathtools}
\usepackage{graphicx,color,xcolor}
\usepackage{multirow,booktabs,array}
\usepackage{float,rotating}
\usepackage{algorithm,algorithmicx,algpseudocode}
\usepackage{url,hyperref}
\usepackage{lipsum}
\usepackage{indentfirst}
\usepackage{authblk}
\usepackage{textcomp}
\usepackage{threeparttable}
\usepackage{pdflscape}
\usepackage{enumerate}

\hypersetup{
	colorlinks=true,
	linkcolor=blue,
	anchorcolor=blue,
	citecolor=blue,
	urlcolor=blue,
	CJKbookmarks=true
}

\newtheorem{theorem}{Theorem}

\newtheorem{lemma}{Lemma}
\newtheorem{remark}{Remark}
\newtheorem{corollary}{Corollary}

\newtheorem{example}{Example}
\newtheorem{definition}{Definition}

\newcommand{\Z}{{\bf Z}}
\newcommand{\T}{{\bf T}}
\newcommand{\X}{{\bf X}}
\newcommand{\bfx}{{\bf x}}
\newcommand{\bfy}{{\bf y}}
\newcommand{\V}{{\bf V}}

\newcommand{\bfD}{{\bf D}}
\newcommand{\z}{{\bf z}}
\newcommand{\U}{{\bf U}}
\newcommand{\bfv}{{\bf v}}
\newcommand{\x}{{\bf x}}
\newcommand{\bmmu}{{\bm\mu}}
\newcommand{\bmrho}{{\bm\rho}}
\newcommand{\bmchi}{{\bm\chi}}

\makeatletter \@addtoreset{equation}{section} \makeatother

\author{Yashi Wei}
\author{Jiang Hu}
\author{Zhidong Bai}

\affil{KLASMOE and School of Mathematics and Statistics, Northeast Normal University, China.}

\date{\today}  

\title{A general partial Cram\'{e}r's condition for Edgeworth expansion of a function of sample means with applications}

\begin{document}
	
	\maketitle
	
	\begin{abstract}
			A large class of statistics can be formulated as smooth functions of sample means of random vectors.  In this paper, we propose a general partial Cram\'{e}r's condition (GPCC) and apply it to establish the validity of the Edgeworth expansion for the distribution function of these functions of sample means. Additionally, we apply the proposed theorems to several specific statistics. In particular, by verifying the GPCC, we demonstrate for the first time the validity of the formal Edgeworth expansion of Pearson's correlation coefficient between random variables with absolutely continuous and discrete components. Furthermore, we conduct a series of simulation studies that show the Edgeworth expansion has higher accuracy.
	\end{abstract}
	
	
	\tableofcontents
	
	\section{Introduction}\label{sec:intro}
	
In classical multivariate analysis, most fundamental statistics can be expressed as functions of sample means. These include the sample mean, sample variance, sample covariance, Pearson's correlation coefficient,  and the empirical distribution function. It is straightforward to demonstrate that the normalized functions of sample means are asymptotically normal under mild conditions, as shown by the delta method. 

However, since the sample size is finite in statistical inference, a more precise evaluation of the asymptotic distributions is necessary. The asymptotic expansion of sample means can be traced back to \cite{cramer1928composition}. Subsequently, \cite{hsu1945approximate} obtained an asymptotic expansion of the sample variance under the assumption of non-singularity of the population. \cite{gotze1978asymptotic} derived an asymptotic expansion of the expectation of a smooth function of the sample mean under the moment condition. For more general cases, \cite{bhattacharya1978validity} established the validity of Edgeworth expansions for functions of sample means under Cram\'{e}r's condition. 

Specifically, consider that $\Z_1, \dots, \Z_n$ are independent and identically distributed (i.i.d.) random vectors in $R^k$.  A random vector $\Z$ in $R^k$ is said to satisfy Cram\'{e}r's condition if its characteristic function $v_\Z({\bf t})=\mathbb{E}(e^{i{\bf t}\Z^\top})$ adheres to the following condition:
\begin{equation}\label{CCC}
	\limsup_{\|{\bf t}\|\to\infty}|v_\Z({\bf t})|<1,
\end{equation}
where ${\bf t}=(t_1,t_2,\dots,t_k)$ and $\|{\bf t}\|$ denotes the standard Euclidean norm of ${\bf t}$.
It is important to note that any distribution with an absolutely continuous component satisfies Cram\'{e}r's condition by the Riemann-Lebesgue lemma, while any purely discrete distribution does not satisfy Cram\'{e}r's condition. Let $H$ be a real-valued Borel measurable function defined on $R^k$. Consider the statistic 
\begin{align}\label{Wn}
	W_n=n^{1/2}(H(\bar \Z)-H({\bm\mu})),
\end{align}
where $\bar \Z=\frac 1n \sum_{j=1}^n\Z_j=(\bar Z_1,\bar Z_2,\dots,\bar Z_k)$ and $\bmmu=\mathbb{E}\Z_1=(\mu_1,\mu_2,\dots,\mu_k) $. 
If the distribution of $\Z_1$ satisfies Cram\'{e}r's condition and has sufficiently many finite moments, under certain smoothness conditions on $H$, \cite{bhattacharya1978validity} proved that 
\begin{align}\label{EWE1}
	\sup_{x}\lvert \mathbb{P}(W_n\leq x)-\Psi_{s,n}(x)\rvert=o(n^{-(s-2)/2}),
\end{align}
where 
\begin{equation}\label{EWE2}
	\Psi_{s,n}(x)=\Phi_{\sigma^2}(x)+\sum_{j=1}^{s-2}n^{-j/2}p_j(x)\phi_{\sigma^2}(x),
\end{equation}
$\sigma^2$ is the limiting variance of $W_n$,
$p_j$'s are polynomials whose coefficients do not depend on $n$, and $\Phi_{\sigma^2}$  and $\phi_{\sigma^2}$ are the cumulative distribution function and the probability density function of a normal distribution with mean zero and variance $\sigma^2$, respectively. 

The Edgeworth expansion from the standardization of sums of i.i.d.  random variables has emerged as a powerful tool in statistics. As shown in Equation \eqref{EWE2}, the Edgeworth expansion corrects the normal term in the central limit theorem (CLT). Higher-order correction terms can be obtained if information on the third or higher moments of the underlying distribution is available. Therefore, using an Edgeworth expansion can increase the convergence rate of statistics or improve the accuracy of statistical inference. There is extensive literature on constructing classical Edgeworth expansion theories, including works by \cite{chung1946approximate}, \cite{feller1991introduction}, \cite{petrov2000classical}, \cite{hall1987edgeworth}, \cite{bhattacharya1988moment}, \cite{bai1992note}, \cite{baiedgeworth}, \cite{hall2013bootstrap}, \cite{decrouez2013normal}, \cite{angst2017weak}, among others. Their pioneering work has been instrumental in guiding later researchers in making high-accuracy statistical inferences. 

Nowadays, the Edgeworth expansion method is widely used as a common asymptotic method in various fields. For instance, \cite{kabluchko2017general} and \cite{podolskij2016edgeworth} applied Edgeworth expansion technology to the profiles of random trees and functions of diffusion processes, respectively. For Bayesian estimation, \cite{kolassa2020validity} rigorously established the validity of formal Edgeworth expansions of posterior densities, demonstrating that their results outperform other existing Edgeworth-type expansions. For a finite sample, \cite{zhilova2022new} studied the accuracy of the Edgeworth expansions in finite sample multivariate settings, establishing approximating bounds with explicit dependence on dimension and sample size.  \cite{zhang2022edgeworth} obtained a high-order approximation of the sample distribution of a given studentized network moment and applied their results to network inference. \cite{he2024higher} obtained higher-order coverage errors for batching methods by building Edgeworth-type expansions on t-statistics.

However, the validity of an Edgeworth expansion is not universally guaranteed, as it relies on specific assumptions regarding the underlying distribution. For instance, the classical Edgeworth expansion for the sample mean function $W_n$ requires the distribution to satisfy Cram\'{e}r's condition. This condition requires that the characteristic function of the random vector $\Z$ decays to zero sufficiently fast. Distributions with discrete components fail to satisfy this criterion, as their characteristic functions are periodic and do not decay. \cite{bai1991edgeworth} proposed a new smoothness condition that relaxes some of the stringent requirements of Cram\'{e}r's condition. This innovation broadens the class of distributions for which valid Edgeworth expansions can be obtained, thereby allowing for more accurate asymptotic approximations in a wide range of practical applications.
This smoothness condition requires that the conditional characteristic function of at least one component of the random vector, given the remaining components, decays sufficiently rapidly. A random vector with random components that are functions of a single underlying random variable fails this condition because the conditional distribution of any component is degenerate and its characteristic function, being periodic, does not decay to zero.

To address these limitations, we propose a new smoothness condition called the general partial Cram\'{e}r's condition (GPCC).  The GPCC framework unifies and generalizes both the classical Cram\'{e}r's condition and the partial Cram\'{e}r's condition. Weakening these conditions is beneficial not only from a theoretical perspective but also for practical statistical applications. For example, under GPCC, we can now derive the Edgeworth expansion of the distribution function of Pearson's correlation coefficient between a continuous and a discrete variable, which was not previously possible. The Pearson correlation coefficient between a continuous variable and a discrete variable (such as a mixture of a Chi-square distribution and a Poisson distribution) can be written in the form of a function of sample means, $H(\bar \Z)$, where $\Z=(X,Y,X^2,Y^2,XY)$, with $X$ following a Chi-square distribution and $Y$ following a Poisson distribution. Due to the discrete components of $\Z$, this statistic does not satisfy the classical Cram\'{e}r's condition. Furthermore, due to the structure of the correlation coefficient when written as a function of a $5$-dimensional vector, it also cannot satisfy partial Cram\'{e}r's condition. Specifically,  given any four components, the remaining component follows a discrete distribution and thus does not satisfy the partial Cram\'{e}r's condition. In addition, we can apply GPCC to other areas, such as the ratio of samples in survival analysis and the Z-score test statistic.

In this paper, we propose the general partial Cram\'{e}r's condition, which guarantees the validity of the Edgeworth expansion for functions of sample means. The GPCC is weaker than both the classical Cram\'{e}r's condition and the partial Cram\'{e}r's condition. Moreover, we prove that the Edgeworth expansion for functions of sample means remains valid if the sum of a number of i.i.d. copies of the basic vector satisfies the GPCC. Furthermore, we demonstrate that the GPCC is applicable to various practical statistical applications, such as the sample correlation coefficient, the ratio of samples in survival analysis, and the Z-score test statistic.

The remaining sections are organized as follows. In Section \ref{mrs}, we introduce the general partial Cram\'{e}r's conditions and state our main theorems regarding the validity of the Edgeworth expansion for functions of sample means.  In Section \ref{sec5}, we apply our theoretical results to several specific statistics and demonstrate the simulation results of our expansion in comparison with normal approximation results. The core proof of Theorem \ref{distance} stated in Section \ref{mrs} is provided in Section \ref{sec6}. The proof details are presented in Sections \ref{appdx}, \ref{appdx2}, \ref{appdx3}.

	\section{Main results}\label{mrs}
	\subsection{General partial Cram\'{e}r's condition}\label{sec2}
	In this subsection, we introduce a general partial Cram\'{e}r's condition to complement the classical Cram\'{e}r's condition and the later partial Cram\'{e}r's condition. The Cram\'{e}r's condition has consistently been the most utilized smoothness condition within the Edgeworth expansion method. For instance, \cite{calonico2022coverage} obtained the Edgeworth expansions of local polynomials based on Wald-type t statistics under Cram\'{e}r's conditions. \cite{bobkov2013rate} established an Edgeworth-type expansion for the entropy distance to the class of normal distributions of sums of i.i.d random variables under Cram\'{e}r's conditions. Additionally, \cite{chatterjee2018edgeworth} derived the validity of the Edgeworth expansion under Cram\'{e}r's conditions when the time series is a linear process driven by a series of i.i.d. random vectors. There are also numerous studies on Edgeworth expansion for U-statistics conducted under Cram\'{e}r's condition, such as \cite{jing2010unified}, \cite{bloznelis2022edgeworth} and recently \cite{jiang2023saddlepoint}.  Clearly,  the Cram\'{e}r's condition plays a pivotal role in the field of Edgeworth expansions.
	
	However, in certain applications, Cram\'{e}r's condition may be a strong assumption as it requires all components of the random vector to be non-lattice. A random variable $X_1$ that takes values in a set of the form $\{a+bk; k\in\mathbb{Z}\}$ for some $a,b\in R$ (where $b\neq 0$) is called a lattice distribution. Generally, discrete distributions can be considered ``approximate lattices'' because if they are not already lattice distributions, they can be viewed as periodically decreasing lattice distributions. If a distribution contains an absolutely continuous component, it is essentially a non-lattice distribution. In many studies, the non-lattice condition is often replaced by the stronger Cram\'{e}r's condition.
	
	For the Edgeworth expansion of the distribution function of $W_n$, when one component of the basic random vector is discrete, the Cram\'{e}r's condition is not applicable. In such cases, \cite{bai1991edgeworth} established the validity of the Edgeworth expansion of functions of sample means under the so-called partial Cram\'{e}r's condition.  A random vector $\Z=(Z_1,\dots, Z_k)$ with values in $R^k$ is said to satisfy the partial Cram\'{e}r's condition if its conditional characteristic function
	\[v_1(t)=\mathbb{E}[\exp(itZ_{1})|Z_{2},\dots,Z_{k}]\]
	is such that
	\begin{equation}\label{PCC}
		\limsup_{|t|\to\infty}\mathbb{E}|v_1(t)|<1.
	\end{equation}
	Note that any random vector with one component being independent of the other components and absolutely continuous satisfies the partial Cram\'{e}r's condition. This pioneering condition complements the classical Cram\'{e}r's condition and enables Edgeworth expansion of more statistics.
	
	Although the partial Cram\'{e}r's condition is generally weaker than Cram\'{e}r's condition, it has a significant limitation. If the basic vector contains two components that are functions of a common variable, say $X$ and $X^2$, then the latter is deterministic given the former. Conversely, given the latter, the former only takes on two different values, $\pm x$. Consequently, the partial Cram\'er condition is not satisfied. This means that many statistics, including fundamental ones such as sample variance, sample covariance, and sample correlation coefficient,  cannot satisfy the partial Cram\'{e}r's condition. Therefore, we propose a new smoothness condition that encompasses both the classical Cram\'{e}r's condition and the partial Cram\'{e}r's condition, which we refer to as the general partial Cram\'{e}r's condition (GPCC).

	\begin{definition}[GPCC]\label{DefGPCC}
		A random vector $\Z$ with values in $R^k$ is said to satisfy the  general partial Cram\'{e}r's condition  (GPCC) if there exists an integer $1\leq a\leq k$, the conditional characteristic function of $\Z$,
		\[v_a(\mathbf{t}_a)=\mathbb{E}[\exp(it_1Z_{1}+\dots+it_a Z_{a})|Z_{(a+1)},\dots,Z_{k}],\]
		is such that
		\begin{equation}\label{Order 1}
			\limsup_{\| \mathbf{t}_a\|\to\infty}\mathbb{E}|v_a( \mathbf{t}_a)|<1.
		\end{equation}
		where $\mathbf{t}_a=(t_1,\dots,t_a)$ and if $a=k$, $v_a(\mathbf{t}_a)=v_\Z(\mathbf{t})$ is the characteristic function of $\Z$.
	\end{definition}
	\begin{remark}
		It is clear that when $a=k$ and $a=1$, the GPCC reduces to the Cram\'{e}r's condition \eqref{CCC} and the partial Cram\'{e}r's condition \eqref{PCC}, respectively. 
	\end{remark}

	The GPCC essentially requires that the joint distribution of the first $a$ components, conditional on the remaining $k - a$ components, possesses a sufficiently strong ``non-lattice'' or ``continuous'' nature to ensure the decay of its conditional characteristic function. In the following, we present some examples that do not satisfy the Cram\'{e}r's condition \eqref{CCC} and the partial Cram\'{e}r's condition \eqref{PCC}, but satisfy the GPCC.

	\begin{example}\label{3-dimension}
		Suppose that $(X,X^2,Y)$ is a random vector where $X$ has an absolutely continuous component, and $Y$ is a discrete random variable independent of $X$. We can then find that :
		\[\limsup _{\|\mathbf{ t}_2\|\to\infty}\mathbb{E}\Big|\mathbb{E}\Big(\exp[i (t_{1}X+t_{2}X^2)]\Big| Y\Big)\Big|<1,\] 
		i.e., the GPCC is satisfied. Additionally, it does not satisfy the partial Cram\'{e}r's condition.  For $X$ and $X^2$, the latter is deterministic given the former, while given the latter, the former only takes on two different values, $\pm x$. The Cram\'{e}r's condition is not met because one of the components, $Y$, is a discrete random variable.
	\end{example}
	
	\begin{example}\label{5-dimension}
		Consider the scenario where $W=(X, Y,  X^2, Y^2, XY)$ is a random vector, and $X$ has an absolutely continuous component. $Y$ is a discrete random variable that is independent of $X$ and is not a constant. We can then find that:
		\[\limsup _{\|\mathbf{ t}_3\|\to\infty}\mathbb{E}\Big|\mathbb{E}\Big(\exp[i (t_{1}X+t_{2}X^2+t_{3}XY)]\Big| Y,Y^2\Big)\Big|<1,\]
		i.e., the GPCC is satisfied. Additionally, it does not satisfy the partial Cram\'{e}r's condition. Given any four components, the remaining component follows a discrete distribution, hence it does not satisfy the partial Cram\'{e}r's condition. The Cram\'{e}r's condition is not met because one of the components, $Y$, is a discrete random variable.
	\end{example}

	\subsection{Statement of the general Edgeworth expansion}\label{section2.2}
	In this section, we review the theoretical results from \cite{bhattacharya2010normal}, which focus on the Edgeworth expansion of the distribution of $k$-dimensional random vectors. Let $\Z$ be a random vector in $R^k$ and $G$ be the probability measure corresponding to the random vector $\Z$. Assume $\hat G$ is the characteristic function of $\Z$. Let $\Phi_{0,\U}$ be the normal distribution in $R^k$ with zero mean and covariance matrix $\U$, and denote its probability density function by $\phi_{0,\U}$. Let $\bm\chi_\bfv$ be the $\bfv$-th  cumulant of  random vector $\Z$, which is defined as,
	\begin{equation}\label{logtaylor}
		\log\hat G(t)=\sum_{|\bfv|\le s}\bm\chi_{\bfv}\frac{(it)^{\bfv}}{\bfv!}+o(\|t\|^s),\qquad (t\to 0).\end{equation}
	Assume $$\beta_s(\textbf{z})=s!\sum_{|\bfv|=s}\frac{\bm\chi_\bfv}{\bfv!}\z^\bfv,$$  where $\z=(z_1,\dots, z_k)$, $\z^\bfv=\prod_{i=1}^k z_i^{v_i}$, and  is a nonnegative integral vector in $R^k$. Besides, $|\bfv|=|v_1|+\dots+|v_k|$  and  $\bfv!=\prod_{i=1}^kv_i!$. Let $u$ is a real variable in $R$. Since
	\[\frac{d^s}{du^s}\log \hat G(ut)\Big|_{u=0}=s!\sum_{|\bfv|=s}\bm\chi_{\bfv}\frac{(it)^\bfv}{\bfv!},\]
	we can interpret $\beta_s(\textbf{z})$ as the $s$-th cumulant of a probability measure on $R$. Define the formal polynomials $p_s(\textbf{z}:\{\bm\chi_\bfv\})$ through the following identity between two formal power series:
	\[1+\sum_{s=1}^\infty p_s(\textbf{z}:\{\bm\chi_\bfv\})u^s=\exp\bigg(\sum_{s=1}^\infty \frac{\beta_{s+2}(\textbf{z})}{(s+2)!} u^s\bigg).\]
	For some integer $s\ge 3$, according to equation \eqref{logtaylor}, we can obtain:
	\begin{align*}
		\log\hat G^n\Big(\frac{t}{n^{1/2}}\Big)&=n\log\hat G\Big(\frac{t}{n^{1/2}}\Big)=-\frac 12\langle t, \U t\rangle+\sum_{r=1}^{s-2}\frac{\beta_{r+2}(it)}{(r+2)!}n^{-r/2}+n\times o\Big(\Big\|\frac{t}{n^{1/2}}\Big\|^s\Big).
	\end{align*}
	Thus, for any fixed $t\in R^k$,
	\begin{align*}
		\hat G^n\Big(\frac{t}{n^{1/2}}\Big)&=\exp\Big(-\frac 12\langle t, \U t\rangle\Big)\times\exp\bigg(\sum_{r=1}^{s-2}\frac{\beta_{r+2}(it)}{(r+2)!}n^{-r/2}+o(n^{-(s-2)/2})\bigg)\\
		&=\exp\Big(-\frac 12\langle t, \U t\rangle\Big)\bigg[1+\sum_{s=1}^\infty n^{-r/2}p_s(\textbf{z}:\{\bm\chi_\bfv\})\bigg]\big(1+o(n^{-(s-2)/2})\big),
	\end{align*}
	which is the asymptotic expansion for the characteristic distribution of independent sums of $k$-dimensional random vectors.
	We denote  \[P_r(-\phi_{0,\U}:\{\bm\chi_\bfv\})=p_r(-\bfD:\{\bm\chi_\bfv\})\phi_{0,\U},\]
	where $\bfD=(D_1,\dots,D_k)$ is a vector consisting of differential operators, and $-\bfD=(-D_1,\dots,-D_k)$. 
	
	\begin{remark}
		It is worth noting that the Fourier transform of $P_r(-\phi_{0,\U}:\{\bm\chi_\bfv\})$ is the coefficient of $n^{-r/2}$ in the asymptotic expansion of the sum of independent random vectors. Let $P_r(-\Phi_{0,\U}:\{\bm\chi_\bfv\})$ be a finite signed measure on $R^k$ with a probability density function $P_r(-\phi_{0,\U}:\{\bm\chi_\bfv\})$. Thus, the distribution function of $P_r(-\Phi_{0,\U}:\{\bm\chi_\bfv\})$ is obtained by using the operator $P_r(-\bfD:\{\bm\chi_\bfv\})$ on the normal distribution function $\Phi_{0,\U}$, i.e.,
		\[P_r(-\Phi_{0,\U}:\{\bm\chi_\bfv\})=p_r(-\bfD:\{\bm\chi_\bfv\})\Phi_{0,\U}.\]
	\end{remark}
	
	\begin{remark}If $k=1$, let $\mu_3$ be the 3-th moment of $\Z$ and $\chi_3$ be the 3-th cumulant of $\Z$, then
		\[P_1(-\phi_{0,\mathbf{I}}:\{\bm\chi_{\mathbf{v}}\})=\frac 16\chi_3(x^3-3x)\phi(x).\]
		Additionally, if the probability measure $G$ has zero mean, then
		\[P_1(-\phi_{0,\mathbf{I}}:\{\bm\chi_{\mathbf{v}}\})=\frac 16\mu_3(x^3-3x)\phi(x).\]
	\end{remark}

	\subsection{Edgeworth expansion under GPCC}\label{sec3}
	In this section, we establish the validity of the formal Edgeworth expansion of a function of sample means under the GPCC. Let $f$ be a real-valued and Borel-measurable function on $R^k$. We define a function $M_{s'}(f)$ as follows:
	\begin{equation*}
		M_{s'}(f)=\left\{
		\begin{aligned}
			\sup_{\bfx\in R^k}(1+\Vert \bfx\Vert^{s'})^{-1}\vert f(\bfx)\vert&,&   s'>0,\\[3mm]
			\sup_{\bfx,\bfy\in R^k}\vert f(\bfx)-f(\bfy)\vert&,&  s'=0.
		\end{aligned}
		\right.
	\end{equation*}
	Next, we define a translate $f_y$ of $f(x)$ by $y \in R$ as $f_{y}(x)=f(x+y)$. Finally, we consider the modulus of continuity and its Gaussian average:
	\[\omega_f(x:\epsilon):=\sup_{y\in B(x,\epsilon)}f(y)-\inf_{y\in B(x,\epsilon)}f(y),\quad \bar\omega_f(\epsilon:\Phi) :=\int\omega_f(x:\epsilon)\,d\Phi(x),\]
	where $B(x:\epsilon)$ denotes an open ball with center $x$ and radius $\epsilon$, and $\Phi(x)$ is the distribution function of the standard normal random variable.
	
	Consider a sequence of i.i.d.  random vectors  $\{\textbf{Z}_i, i=1,\dots,n\}$ with values in $R^k$, having zero means and a nonsingular covariance matrix $\bf V$. Write $\textbf{Z}_j=(Z_{j1},\dots, Z_{jk})$. Let $C_n\overset{\Delta}{=}\{Z_{j(a+1)},\dots, Z_{jk}, j=1,\dots,n\}$, where $1\leq a\leq k$ is an integer. Let $Q_n^*$ be the conditional distribution of $n^{1/2}\bar \Z=n^{-1/2}\sum_{i=1}^n\textbf{Z}_i$ given $C_n$.

	\begin{theorem}\label{distance}
		Assume that the distribution function $G_1$ of $\textbf{Z}_1$ has a finite $s$-th absolute moment for some integer $s \ge 3$. Additionally, assume the conditional distribution $G^*_1$ given $C_n$ satisfies the GPCC \eqref{Order 1}. Let $\U$  and  $\chi_\bfv$ be the covariance matrix and $\bfv$-th cumulant of $G_1$ respectively ($3\le \vert \bfv\vert\le s$). Then, for every real-valued, Borel-measurable function $f$ on $R^k$ satisfying
		\[M_{s'}(f)<\infty\]
		for some $s'$, $0\le s' \le s$, we have that 
		\begin{align}
			\bigg\vert \mathbb{E}\int f d\bigg( Q^*_n-\sum_{r=0}^{s-2}n^{-r/2}&P_r(-\Phi_{0,\U}:\{\bm\chi_\bfv\})\bigg)\bigg\vert\\\nonumber
			&\le M_{s'}(f)\delta_1(n)+c(s,k)\bar\omega_f(2e^{-dn}:\Phi_{0,\U}),
		\end{align}
		where $d$ is a suitable positive constant, $c(s,k)$ and $C(s,k)$ depend only on $s$ and $k$, and 
		\[\delta_1(n)=o(n^{-(s-2)/2}),\qquad (n\to\infty).\]
		Moreover,   the quantities $d$, $\delta_1(n)$ do not depend on $f$.
	\end{theorem}

	Theorem \ref{distance} is a generalized version of the result in \cite{bhattacharya2010normal}. This theorem is particularly useful for proving higher order asymptotic results on $Q_n$.

	\begin{remark}
		Theorem \ref{distance} indicates that our Edgeworth expansion expression may not be the same as that of \cite{bhattacharya2010normal}. However, the difference between them is minimal, with the discrepancy not exceeding $o(n^{-(s-2)/2})$. 
	\end{remark}

	\begin{remark}
		It should be noted that the conditional probability of $\Z$ given $k-a$ variables is still a $k$-dimensional function. Taking the binary case as an example, assume that $X_1$ and $X_2$ are coordinate random variables on a probability space $(R^2,\mathscr B^2, P)$ with an absolutely continuous density function $f(x_1,x_2)$. For $B\in\mathscr B^2$ and $\textbf{x}=(x_1,x_2)\in R^2$, define
		\begin{align*}
			f_1(x_1|x_2)=\left\{\begin{aligned}
				\frac{f(x_1,x_2)}{f_2(x_2)} & &\mathrm{if}  f_2(x_2)>0\\
				f_1(x_1) & &\mathrm{if} f_2(x_2)=0,
			\end{aligned}
			\right. \quad  P(B,\textbf{x})=\int_{\{s:(s,x_2)\in B\}}f_1(s|x_2)ds.
		\end{align*}
		Then, $P(B,\textbf{x})$ is a regular conditional probability measure on $\mathscr B^2$ given $\sigma(X_2)$.
	\end{remark}

	The proof for Theorem \ref{distance} is deferred to Section \ref{sec6}. The following corollary is immediate. Taking $f$ as the indicator of a special Borel set yields:
	\begin{corollary}\label{co1.1}
		Under the assumptions of Theorem \ref{distance}, we have that 
		\begin{equation}
			\sup_{B\in\mathscr{B}^k}\Big|\mathbb{E}Q_n^*(B)-\sum_{r=0}^{s-2}n^{-r/2}P_r(-\Phi_{0,\U}:\{\bm\chi_\bfv\})(B)\Big|=o(n^{-(s-2)/2}),
		\end{equation}
		for every class $\mathscr{B}$ of Borel sets satisfying
		\begin{equation}\label{boundary}
				\sup_{B\in\mathscr{B}^k}\int_{(\partial B)^\epsilon}\phi_{0,\U}(\x)d\x=O(\epsilon).
		\end{equation}
	Here $\partial B$ is the boundary of $B$, $(\partial B)^{\epsilon}$ is the $\epsilon$-neighborhood of $B$ and $\epsilon\to 0$. 
	\end{corollary}
	
	We are now in a position to consider the Edgeworth expansion result of a function of sample means $H(\bar \Z) $ under the GPCC. In some situations, the mean or higher-order moments of $H(\bar \Z)$ may not exist. To overcome this limitation, the Taylor expansion of $H(\bar \Z)$ has been employed to obtain the Edgeworth expansion of the distribution function of $W_n$. This approach eliminates the need to assume the existence of moments of $H(\bar \Z)$, requiring only the existence of moments of $Z_1$ and the existence of derivatives of $H$ at $\bm{\mu}$. By employing this method, the estimation of the distribution of $W_n$ becomes achievable. Denote the partial derivatives of $H$ at $\bm\mu$ by 
	\[l_{i_1,\dots,i_p}=(D_{i_1}D_{i_2}\cdots D_{i_p}H)(\bm\mu),\quad 1\le i_1,\dots,i_p\le k.\] 
	If all the derivatives of $H$ of order $s$ and less are continuous in a neighborhood of $\bm\mu$,   then the Taylor expansion of $W_n$ in \eqref{Wn}  yields the statistic
	\begin{align}\label{Taylor expansion}
		W_n'=n^{1/2}\Big(\sum_{i=1}^k&l_i(\bar Z_i-\mu_i)+\frac12\sum_{i,j}l_{i,j}(\bar Z_i-\mu_i)(\bar Z_j-\mu_j)+\cdots\\\nonumber
		&+\frac{1}{(s-1)!}\sum_{i_1,\dots,i_{s-1}}l_{i_1,\dots,i_{s-1}}(\bar Z_{i_1}-\mu_{i_1})\cdots(\bar Z_{i_{s-1}}-\mu_{i_{s-1}})\Big),
	\end{align}
	and  $W_n=W'_n+o_p(n^{-(s-2)/2})$. As a result, the asymptotic expansion of the distribution of $W_n'$ coincides with that of $W_n$. Moreover, recall that $\bm{\Sigma}=(\sigma_{ij})$ is the covariance matrix of $\Z_1$. Let $\sigma^2=\sum_{i,j=1}^k \sigma_{ij}l_il_j$ and $\kappa_{j,n}$ be the $j$-th cumulant of $W_n'$. Then, from \eqref{Taylor expansion}, we can obtain 
	\[\kappa_{j,n}=\tilde \kappa_{j,n}+o(n^{-(s-2)/2}),\]
	where $\tilde \kappa_{j,n}=\sigma^2+\sum_{i=1}^{s-2}n^{-i/2}b_{2,i}$ when $ j=2$, while $\tilde \kappa_{j,n}=\sum_{i=1}^{s-2}n^{-i/2}b_{j,i}$ when $ j\neq 2$. Here $b_{j,i}$ depend only on appropriate moments of $\Z_1$ and derivatives of $H$ at $\bmmu$ of orders $s-1$ and less. Then the expression
	\begin{equation}\label{expansion}
		\exp\Big(it\tilde\kappa_{1,n}+\frac{(it)^2}{2}(\tilde\kappa_{2,n}-\sigma^2)+\sum_{j=3}^{s}\frac{(it)^j}{j!}\tilde\kappa_{j,n}\Big)\exp(-\sigma^2t^2/2)
	\end{equation}
	is an approximation of the characteristic function of $W_n'$. Namely, we can obtain an approximation of the characteristic function of $W_n$ by appropriate moments of $\Z_1$ and derivatives of $H$ at $\bmmu$ of orders $s-1$ and less. 
	Thus, we can rewrite \eqref{expansion} as
	\begin{equation}\label{fourierexpansion}
		\exp(-\sigma^2t^2/2)\Big[1+\sum_{r=1}^{s-2}n^{-r/2}\pi_r(it)\Big]+o(n^{-(s-2)/2})=\hat\psi_{s,n}(t)+o(n^{-(s-2)/2}),
	\end{equation}
	where $\pi_r(\cdot)$ $(1\le r\le s-2)$ are polynomials that depend only on the moments of orders $s$ and less of $\Z_1$,
	\[\psi_{s,n}(x)=\Big[1+\sum_{r=1}^{s-2}n^{-r/2}\pi_r\Big(-\frac{d}{dx}\Big)\Big]\phi_{\sigma^2}(x), \quad \Psi_{s,n}(u)=\int_{-\infty}^{u}\psi_{s,n}(x)dx,\]
	and $\hat\psi_{s,n}$ is the Fourier-Stieltjes transform of $\Psi_{s,n}$.
	In addition, let $\mathcal{Q}_n$ be the distribution function of $W_n$. Then  we have the following theorem of the validity of the Edgeworth expansion of $\mathcal{Q}_n$. 
	\begin{theorem}\label{main1}
		Suppose that $\{\Z_j\}$ is a sequence of i.i.d.  random k-vectors. 
		Assume that:  (A1)  all the derivatives of $H$ of order $s$ and less are continuous in a neighborhood of $\bm\mu$, where $s\ge 3$;
		(A2) $\Z_{1}$ has finite $s$-th absolute moment, where $s\ge 3$ is a known integer and  (A3) $\Z_1$ satisfies the  GPCC, then we have that 
		\begin{align}\label{maineq}
			\sup_{x}\lvert \mathcal{Q}_n(x)-\Psi_{s,n}(x)\rvert=o(n^{-(s-2)/2}).
		\end{align}
	\end{theorem}

	\begin{remark}	
		\cite{gotze1978asymptotic} focused on the case where $f$ is a smooth function, whereas our results demonstrate that $f$ can be a real-valued Borel-measurable function. Consequently, our research addresses a much broader class of functions. For instance, indicator functions on measurable sets fall within our function class but not the one considered in the work of \cite{gotze1978asymptotic}.
	\end{remark}

	\begin{remark}	
		A key challenge in the proof of Theorem \ref{main1} is to demonstrate that the difference between two distribution functions is sufficiently small by controlling the difference between their corresponding characteristic functions. This is achieved through a three-part argument. For small values of $t$, a Taylor expansion is used to bound the difference. For large values of $t$, the exponential decay of the characteristic function ensures that the difference becomes negligible. For intermediate values of $t$, the GPCC is utilized to guarantee the necessary decay.
	\end{remark}

	\begin{remark}\label{second-term}
		Note that $\Psi_{s,n}$ can be written that
		\[\Psi_{s,n}(x)=\Phi_{\sigma^2}(x)+\sum_{j=1}^{s-2}n^{-j/2}p_j(x)\phi_{\sigma^2}(x),\]
		where $p_j$ is a polynomial of degree not exceeding $3j-1$ whose coefficients do not depend on $n$. In fact, the coefficients are determined by the moments of $\Z_1$ of orders not greater than $j+2$ and the partial derivatives of $H$ at $\bmmu$.
		In particular, define $\mu_{i_1\dots i_j}=\mathbb{E}(Z_{1,i_1}-\mu_{i_1})\cdots(Z_{1,i_j}-{\mu}_{i_j})$ for $j\ge 1$. We can obtain the specific form of $\Psi_{s,n}(x)$. For illustration, we calculate the formula for the coefficients in the polynomials $p_1$ and $p_2$, which are
		\begin{align*}
			&p_1(x)=-\Big(A_1\sigma^{-1}+\frac16 A_2\sigma^{-3}(x^2-1)\Big),\\
			&p_2(x)=-x\bigg(\frac 12\big[B_2/\sigma^2+(B_1/\sigma)^2\big]+\frac{1}{24}\big[B_4/\sigma^4+4(B_1/\sigma) (B_3/\sigma^3)\big](x^2-3)+\frac{1}{72}\\
			&\qquad\qquad\qquad\times(B_3/\sigma^3)^2(x^4-10x^2+15)\bigg),
		\end{align*}
		where the expressions of $A_1, A_2, B_1, B_2, B_3$, and $B_4$ can be found in Appendix \ref{B.2}.
		In fact, $A_1, A_2, B_1, B_2, B_3$, and $B_4$ are expressed in terms of the higher-order derivatives of $H$ and the higher-order moments of $\Z$. 
	\end{remark}

	\begin{remark}\label{CPC}
		In Theorem \ref{main1},  we assume that $\Z_j$ are i.i.d. and  $\Z_1$  satisfy the GPCC. However, there are cases where $\Z_1$ does not satisfy the GPCC, but $\sum_{j=1}^b\Z_j$ does; or where $\Z_j$, for $j=1,\dots,n$, are not i.i.d., but the aggregated variable $\tilde\Z_i=\frac 1b (\Z_{b(i-1)+1}+\dots+\Z_{bi})$, for $i=1,\dots, n/b$ are i.i.d..  Here $b>1$ is an integer. 
		
		For example, suppose $w_i$ is a sequence of i.i.d.  Bernoulli distributed random variables, with each random variable taking the values $0$ and $1$ with equal probability $1/2$ . Then we can express $Z=\sum_{i=1}^\infty w_i/2^{2i-1}$, which is a singular continuous random variable and does not satisfy the Cram\'{e}r's condition. Assuming $Z'$  and $Z$ are i.i.d., we find that $Z+Z'/2$ is an absolutely continuous random variable following the uniform distribution $U(0,1)$. Therefore, it is straightforward that \eqref{maineq} holds for $\Z_1$ replaced by $\tilde \Z_1$ and $n$ replaced by $n/b$ in the assumptions of Theorem \ref{main1}.
	\end{remark}

	\subsection{A special case}
	In this subsection, we consider a random vector characterized by a special structure that is useful in statistics and present an easier way to verify the GPCC. We assume that 
	\[Z_{j1},\dots, Z_{jk}\]
	are generated from the same random variable $w_j$, so that 
	\[\Z_j=(w_j, K_1(w_j),\dots, K_{k-1}(w_j)),\]
	where $K_i(x)$, for $i=1,\dots,k-1$, are first-order differentiable functions. Additionally, we assume that $w_j$, for $j=1, \dots,n$ are i.i.d. and absolutely continuous.

	\begin{theorem}\label{main3}
		Suppose the assumptions $(A1)$ and $(A2)$ in Theorem \ref{main1} hold. If  
		\begin{equation*}J=\left | \begin{matrix}
				1  \quad&1\quad  & \dots & \quad1\\
				K'_1(w_1) \quad&K'_1(w_2)\quad & \dots &\quad K'_1(w_k) \\
				\vdots \quad& \vdots \quad& \cdots & \quad\vdots \\
				K'_{k-1}(w_1) \quad&K'_{k-1}(w_2)\quad & \dots &\quad K'_{k-1}(w_k) 
			\end{matrix} \right |\neq 0 
		\end{equation*}
		almost surely, then we have that
		\[\sup_{x}\lvert \mathcal{Q}_n(x)-\Psi_{s,n}(x)\rvert=o(n^{-(s-2)/2}).\]
		
	\end{theorem}
	
	In the case where $k=2$, the following corollary provides a clearer condition, thereby simplifying the verification of the primary conclusion of Theorem \ref{main3}.  
	
	\begin{corollary}\label{co1}
		Suppose the assumptions $(A1)$ and $(A2)$ in Theorem \ref{main1} hold.
		If $k=2$ and $K_1(x)$ is a nonlinear first-order differentiable function, then we have:
		\[\sup_{x}\lvert \mathcal{Q}_n(x)-\Psi_{s,n}(x)\rvert=o(n^{-(s-2)/2}).\]
	\end{corollary}
	
	\begin{remark}
		It is important to note that $|J|\neq 0$ (almost surely) is only a sufficient condition, not a necessary one. For instance, there exists a $K_j(x)=ax+b$ with $a\neq 0$, for which the Jacobian determinant $ J$ is equal to zero. Consider a two-dimensional random vector $(X, K_1(X))$. Although it dose not satisfy the conditions of Corollary \ref{co1}, it actually satisfies the GPCC.
	\end{remark}

	The proof of Theorems \ref{main3} is deferred to Section \ref{sec6}. In the remainder of this subsection, we present several examples for illustration, which provide valuable insights into the practical implications of Theorem \ref{main3}. 
	\begin{example}\label{square}
		It is well known that the sample variance can be expressed as a function of $(w, w^2)$. Regardless of whether $w$ or $w^2$ is given, the conditional characteristic function does not satisfy Cram\'{e}r's condition. Our theorem provides an alternative validity condition for the Edgeworth expansion. Specifically, if $w$ has an absolutely continuous component, then the Edgeworth expansion of sample variance is valid.
	\end{example}
	
	\begin{example}\label{Ex1}
		Consider the case of $k=2$ where the random vector $\Z$ has a special structure. Suppose that $\Z=(w, \log w)$ and that $w$ has an absolutely continuous component. Under these circumstances, our Theorem \ref{main3} applies, yielding a valid Edgeworth expansion for the distribution of $H(\bar \Z)$.
	\end{example}
	
	\begin{example}\label{logdata}
		We showcase a practical application of Example \ref{Ex1} through point estimation. For the mean of the log-normal distribution, the maximum likelihood estimate, given by $\frac 1n\sum_{i=1}^n\log w_i$, can be expressed as $H(\bar \Z)$. Here, $\Z_i=(w_i,\log w_i)$, where $w_i$ represents an i.i.d.  random variable following a log-normal distribution.
	\end{example}

	\section{Applications and numerical examples}\label{sec5}
	In this section, we apply the theoretical results from Section \ref{mrs}, focusing on the expansions of the sample correlation coefficient, the ratio of samples, and the Z-score test statistic.

	\subsection{A valid Edgeworth expansion of Pearson's correlation coefficient}
	In this section, we present the validity of the formal Edgeworth expansion of Pearson's correlation coefficient between two random variables under the GPCC, with particular attention to the case where one variable is continuous and the other is discrete. Previous research includes \cite{babu1989edgeworth}, which provided first-order Edgeworth expansion results for the correlation coefficient of two-dimensional random variables $(X, Y)$, where $X$ is continuous and $Y$ is lattice. Additionally, \cite{ogasawara2006asymptotic} derived the second-order expansion of the sample correlation coefficient under Cram\'{e}r's condition and used simulations to confirm the accuracy of the second-order expansion. 
	
	Consider a sequence of i.i.d.  random two-dimensional vectors ${\bf Y}_n=(Y_{n1},Y_{n2}), {n\ge 1}$. Let $f_1,\dots, f_5$ be real-valued Borel measurable functions on $R^2$. Assume
	\[\Z_i=(f_1({\bf Y}_i), f_2({\bf Y}_i),\dots, f_5({\bf Y}_i)),\]
	with
	\[\bar \Z=\frac 1n \sum_{i=1}^n \Z_i=\bigg(\frac 1n \sum_{i=1}^n f_1({\bf Y}_i),\frac 1n \sum_{i=1}^n f_2({\bf Y}_i),\dots, \frac 1n \sum_{i=1}^n f_5({\bf Y}_i)\bigg),\]
	where
	\[f_1({\bf Y}_n)=Y_{n1}, f_2({\bf Y}_n)=Y_{n2}, f_3({\bf Y}_n)=Y_{n1}^2, f_4({\bf Y}_n)=Y_{n2}^2, f_5({\bf Y}_n)=Y_{n1}Y_{n2}.\] 
	Let \[\bmmu=(\mathbb E Y_{11},\mathbb E Y_{12}, \mathbb E Y_{11}^2, \mathbb E Y_{12}^2,\mathbb E Y_{11}Y_{12}),\]
	and define
	\[H(\textbf{z})=(z_5-z_{1}z_{2})(z_{3}-z_{1}^2)^{-1/2}(z_{4}-z_{2}^2)^{-1/2},\quad \text{for}~ \textbf{z}=(z_{1},\dots,z_{5}).\]
	Then the Pearson's population correlation coefficient of $Y_{11}$ and $Y_{12}$ can be expressed as $\rho=H(\bmmu)$. Pearson's sample correlation coefficient is:
	\begin{align*}
		H(\bar \Z)&=\frac{\frac 1n \sum_{i=1}^n  f_5({\bf Y}_i)-(\frac 1n \sum_{i=1}^n  f_1({\bf Y}_i))(\frac 1n \sum_{i=1}^n  f_2({\bf Y}_i))}{[\frac 1n \sum_{i=1}^n  f_3({\bf Y}_i)-(\frac 1n \sum_{i=1}^n  f_1({\bf Y}_i))^2]^{\frac 12}[\frac 1n \sum_{i=1}^n  f_4({\bf Y}_i)-(\frac 1n \sum_{i=1}^n  f_2({\bf Y}_i))^2]^{\frac 12}},\\[3mm]
		&=\frac{\frac 1n \sum_{i=1}^n Y_{i1}Y_{i2}-(\frac 1n \sum_{i=1}^n Y_{i1})(\frac 1n\sum_{i=1}^nY_{i2})}{\big[\frac 1n\sum_{i=1}^n Y_{i1}^2-(\frac 1n \sum_{i=1}^n Y_{i1})^2\big]^{\frac 12}\big[\frac 1n\sum_{i=1}^n Y_{i2}^2-(\frac 1n \sum_{i=1}^n Y_{i2})^2\big]^{\frac 12}},\\[3mm]
		&\overset{\Delta}{=}\hat\rho.
	\end{align*}
	
	By Theorem \ref{main1}, we have the following theorem, which establishes the validity of the formal Edgeworth expansion of the sample correlation coefficient under GPCC.
	
	\begin{theorem}
		Assume the following:
		\begin{itemize}
			\item[(A1)]
			
			$H$ is $s$ times continuously differentiable in a neighborhood of $\bmmu$,   where $s\ge 3$ is an integer.
			\item[(A2)]
			${\bf Y}_{1}$ has finite $s$-th absolute moments.
			\item[(A3)] $\Z_1$ satisfies the GPCC.
		\end{itemize}
		Then we have that 
		\[ P(n^{1/2}(\hat\rho-\rho)\le x)=\Phi_{\sigma^2}(x)+\sum_{j=1}^{s-2}n^{-j/2}p_j(x/\sigma)\phi_{\sigma^2}(x)+o(n^{-(s-2)/2}),\]
		where $p_j$ is a polynomial of degree not exceeding $3j-1$ whose coefficients do not depend on $n$. In fact, the coefficients are determined by the cumulants of $\Z_1$ of orders not greater than $j+2$ and the partial derivatives of $H$ at $\bm{\mu}$. 
	\end{theorem}
	\begin{remark}
		Our results confirm the validity of the formal Edgeworth expansion for sample correlation coefficients, not only for two continuous random variables, but also for correlation coefficients involving a continuous and a discrete random variable. For the expansion of the correlation coefficient between two discrete random variables, we hypothesize that additional Edgeworth expansion formulas may be necessary. 
	\end{remark}
	\begin{corollary}\label{correlation} Adopting the above theorem, the first-order Edgeworth expansion of $n^{1/2}(\hat\rho-\rho)$ is given by
		\[ P\big(n^{\frac 12}(\hat\rho-\rho)\le x\big)=\Phi_{\sigma^2}(x)-n^{-\frac 12}\Big(A_3\sigma^{-1}+\frac16 A_4\sigma^{-3}(x^2\sigma^{-2}-1)\Big)\phi_{\sigma^2}(x)+o(n^{-\frac 12}),\]
		valid uniformly in $x$. The specific expressions of $A_3$ and $A_4$ can be found in Appendix \ref{B.2}.
	\end{corollary}

	\subsection{A valid Edgeworth expansion for the ratio of sample means}
	An important example of a function of sample means with a counting component is the ratio estimator used in survival analysis, such as the ratio of the proportion of individuals dying in a given period to the average lifetime. \cite{babu1989edgeworth} presents first-order Edgeworth expansion results for the single ratio case, while \cite{bai1992note} extends these results, providing further insights into the statistical properties of such estimators. 
	
	In practice, outcomes can be influenced by multiple factors. For instance, the number of plants that die can depend on drug dosage, environmental conditions, and genetic variability. To account for multiple influences, the Edgeworth expansion of the statistic for the ratio of multiple sample means can be utilized. In this subsection, we apply our theorem to the multivariate ratio case. Suppose \[\{(X_{1i}, Y_{1i}), (X_{2i},Y_{2i}), \dots, (X_{ki},Y_{ki}), i=1,2,\dots n\}\] is a sequence of i.i.d random vectors with finite $s$-th moment ($s\ge 3$). Define
	\[R_j=\frac{\sum_{i=1}^nX_{ji}}{\sum_{i=1}^nY_{ji}},\quad (j=1,\dots, k),\quad W_n=R_1^2+\dots+R_k^2.\]
	Assume $\Z_i=(X_{1i}, \dots,X_{ki}, Y_{1i},\dots, Y_{ki})$, then we can rewrite $W_n$ as
	\[W_n=H(\bar Z), \quad H(x_1,\dots, x_{2k})=\Big(\frac{x_1}{x_{k+1}}\Big)^2+\dots+\Big(\frac{x_k}{x_{2k}}\Big)^2.\]
	Besides, denote the partial derivatives of $H$ at $\bm\mu$ by 
	\[l_{i_1,\dots,i_p}=(D_{i_1}D_{i_2}\cdots D_{i_p}H)(\bm\mu),~ 1\le i_1,\dots,i_p\le k.\]
	Assume  $\bmmu=\mathbb{E}\Z_1=(\mu_1,\mu_2,\dots,\mu_{2k}) $ and  $\sigma^2=\sum_{i,j=1}^k \sigma_{ij}l_il_j$. By Theorem \ref{main1}, we have the following theorem:
	
	\begin{theorem}
		Assume the following:
		\begin{itemize}
			\item[(A1)]			
			$H$ is $s$ times continuously differentiable in a neighborhood of $\bmmu$,   where $s\ge 3$ is an integer.
			\item[(A2)]
			${\bf Y}_{1}$ has finite $s$-th absolute moments.
			\item[(A3)] $\Z_1$ satisfies the GPCC.
		\end{itemize}
		Then we have that 
		\[ P\Big(n^{1/2}(H(\bar Z)-H(\mu))\le x\Big)=\Phi_{\sigma^2}(x)+\sum_{j=1}^{s-2}n^{-j/2}p_j(x/\sigma)\phi_{\sigma^2}(x)+o(n^{-(s-2)/2}),\]
		where $p_j$ is a polynomial of degree not exceeding $3j-1$ whose coefficients do not depend on $n$. In fact, the coefficients are determined by the cumulants of $\Z_1$ of orders not greater than $j+2$ and the partial derivatives of $H$ at $\bm{\mu}$. 
	\end{theorem}

	\subsection{A valid Edgeworth expansion of the Z-score test statistic}
	The log-normal distribution is widely observed in various fields, including finance, medicine, and environmental science. The Z-score test statistic, proposed by \cite{zhou1997methods}, is designed to compare the means of two log-normal outcomes using log-transformed data. In this subsection, we present the statistical application of Corollary \ref{co1}, focusing on the Edgeworth expansion for the distribution function of the Z-score test statistic. Assume that
	\[\log X_{i}\sim N(\mu_1,\sigma_1^2), \quad \log Y_{i}\sim N(\mu_2,\sigma_2^2).\] 
	The null  hypothesis is
	\[H_0:M_1=M_2,\]
	where $M_1$ and $M_2$ are $X_i$ and $Y_i$ corresponding means respectively. Define
	\begin{align*}
		&\hat \mu_1=\frac{1}{n}\sum_{i=1}^n \log X_{i}, \quad S_1^2=\frac{1}{n_1-1}\sum_{i=1}^{n_1}(\log X_i-\hat \mu_1)^2,\\
		&\hat \mu_2=\frac{1}{n}\sum_{i=1}^n \log Y_{i},\quad S_2^2=\frac{1}{n_2-1}\sum_{i=1}^{n_2}(\log Y_i-\hat \mu_2)^2.
	\end{align*}
	The test statistic proposed is
	\[W_n=\frac{\hat\mu_2-\hat\mu_1+(1/2)(S_2^2-S_1^2)}{\sqrt{\frac{S_1^2}{n_1}+\frac{S_2^2}{n_2}+(1/2)\Big(\frac{S_1^4}{n_1-1}+\frac{S_2^4}{n_2-1}\Big)}}.\]
	Let $a=n_1/n_2$ and set
	\[\Z_i=(\log(X_{i}), \log^2(X_{i}), \log(Y_{i}), \log^2(Y_{i})),\]
	so that we can express $W_n$ in the form 
	\[W_n=\sqrt n_2 H(\bar Z),\]
	where
	\begin{align}
		\label{H2}&H(x_1,x_2,x_3,x_4)=\frac{x_3-x_1-\frac{1}{2}(x_2-x_1^2)+\frac{1}{2}(x_4-x_3^2)}{\sqrt{a(x_2-x_1^2)+(x_4-x_3^2)+\frac{1}{2}a(x_2-x_1^2)^2+\frac{1}{2}(x_4-x_3^2)^2}}.
	\end{align}
	In the case where $n_1\neq n_2$, $W_n$ can be represented in the form of \eqref{H2}. Besides, denote the partial derivatives of $H$ at $\bm\mu$ by 
	\[l_{i_1,\dots,i_p}=(D_{i_1}D_{i_2}\cdots D_{i_p}H)(\bm\mu),~ 1\le i_1,\dots,i_p\le 4.\]
	Assume  $\bmmu=\mathbb{E}\Z_1=(\mu_1,\mu_2,\mu_3,\mu_{4}) $ and  $\sigma^2=\sum_{i,j=1}^4 \sigma_{ij}l_il_j$. By Theorem \ref{main1}, we have the following theorem:
	
	\begin{theorem}
		Assume the following:
		\begin{itemize}
			\item[(A1)]	
			$H$ is $s$ times continuously differentiable in a neighborhood of $\bmmu$,   where $s\ge 3$ is an integer.
			\item[(A2)]
			${\bf Y}_{1}$ has finite $s$-th absolute moments.
			\item[(A3)] $\Z_1$ satisfies the GPCC.
		\end{itemize}
		Then we have that 
		\[ P\Big(n^{1/2}(H(\bar Z)-H(\mu))\le x\Big)=\Phi_{\sigma^2}(x)+\sum_{j=1}^{s-2}n^{-j/2}p_j(x/\sigma)\phi_{\sigma^2}(x)+o(n^{-(s-2)/2}),\]
		where $p_j$ is a polynomial of degree not exceeding $3j-1$ whose coefficients do not depend on $n$. In fact, the coefficients are determined by the cumulants of $\Z_1$ of orders not greater than $j+2$ and the partial derivatives of $H$ at $\bm{\mu}$. 
	\end{theorem}
	
	\subsection{Simulation experiments for correlation}\label{simulation}
	In this section, we use numerical experiments to evaluate the performance of the Edgeworth expansion of the sample correlation. We present the results of the first-order and second-order Edgeworth expansions. For comparison, we also present the results of the normal approximation. 
	
	\textbf{Experiment 1 (Continuous and continuous random variables )} In this experiment, we generate two independent and identically distributed continuous random variables $X$ and $Y$, each following $\chi^2(1)$ distribution. Define
	\[\Z=(X, Y,X^2,Y^2,XY),\]
	and let $\Z_i=(X_{i},Y_{i},X_{i}^2, Y_{i}^2,X_{i}Y_{i})$ for $i=1,\dots, n$. Suppose that 
	\begin{equation*}
		\setlength{\arraycolsep}{4pt}
		\mu=(1,1,3,3,1),\quad\Sigma=\begin{pmatrix}
			2&0& 12& 0& 2\\
			0&2& 0& 12& 2\\
			12& 0& 96& 0& 12\\
			0&12& 0& 96& 12\\
			2&2& 12& 12& 8
		\end{pmatrix}.
	\end{equation*}
	Based on our previous theorem, we can find that $Z_i$ satisfy the GPCC. Specifically, 
	\[\limsup _{\|\mathbf{ t}\|\to\infty}\mathbb{E}\Big|\mathbb{E}\Big(\exp[i (t_{1}Y_{i}+t_{2}Y_{i}^2+t_{3}X_{i}Y_{i})]\Big| X_{i},X_{i}^2\Big)\Big|<1.\]
	Therefore, the sample correlation of $X$ and $Y$ can be expanded using Corollary \ref{correlation}.

	\begin{figure}
		\begin{subfigure}{0.5\textwidth}
			\centering
			\includegraphics[width=\linewidth]{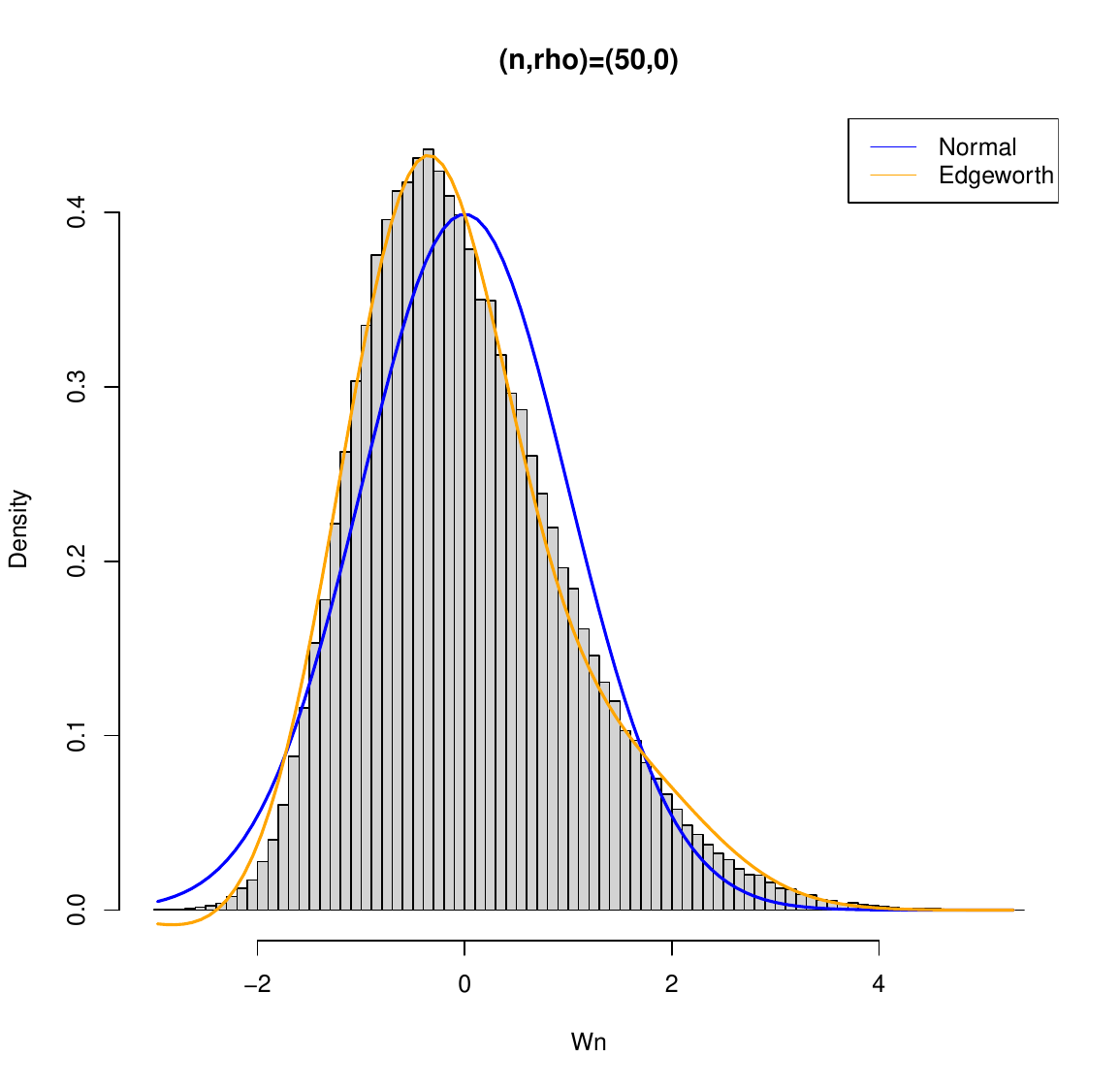}
			\caption{Chisq \& Chisq at $n=50$}
			\label{fig:image1}
		\end{subfigure}%
		\begin{subfigure}{0.5\textwidth}
			\centering
			\includegraphics[width=\linewidth]{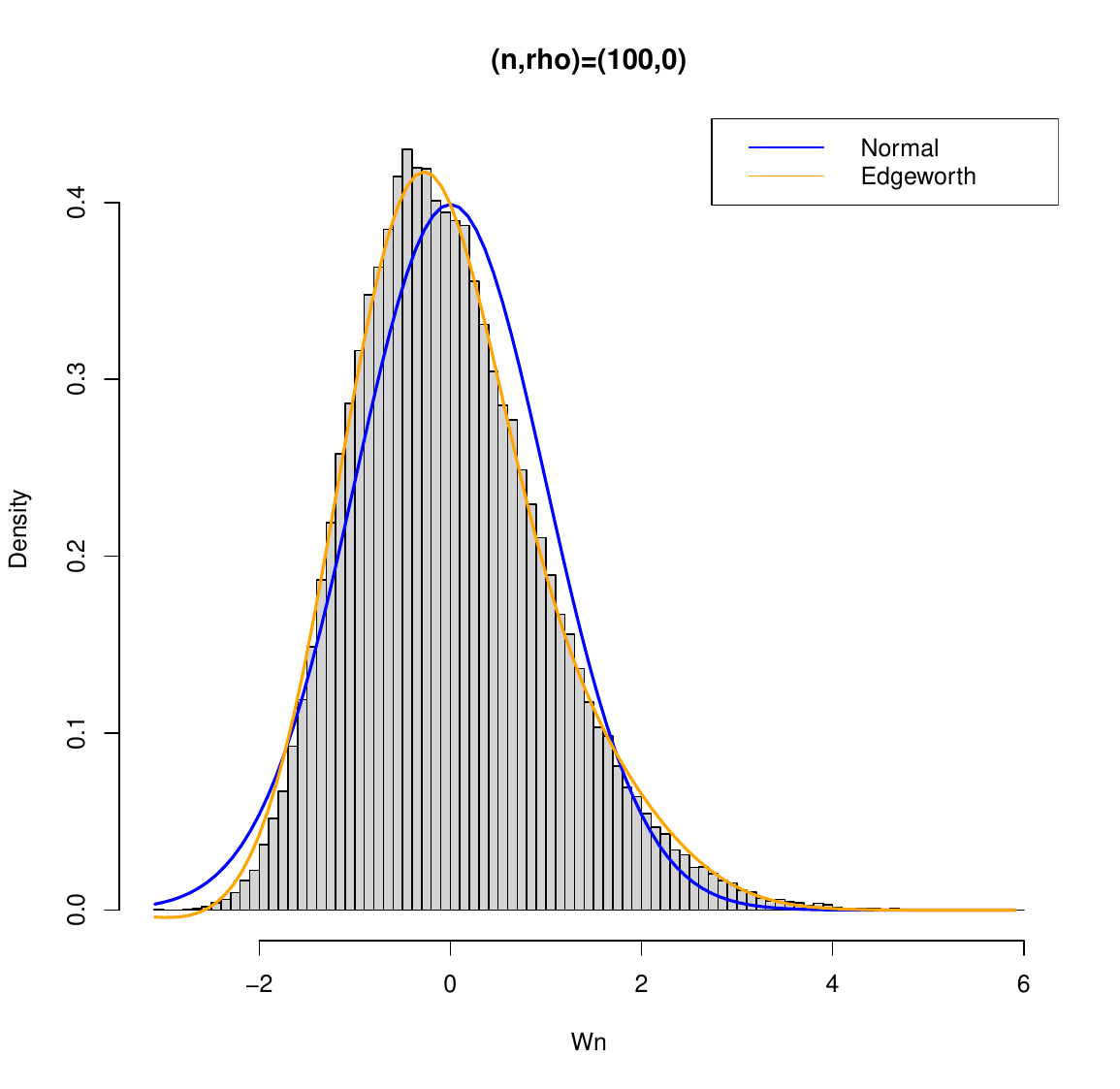}
			\caption{Chisq \& Chisq at $n=100$}
			\label{fig:image2}
		\end{subfigure}
		
		\begin{subfigure}{0.5\textwidth}
			\centering
			\includegraphics[width=\linewidth]{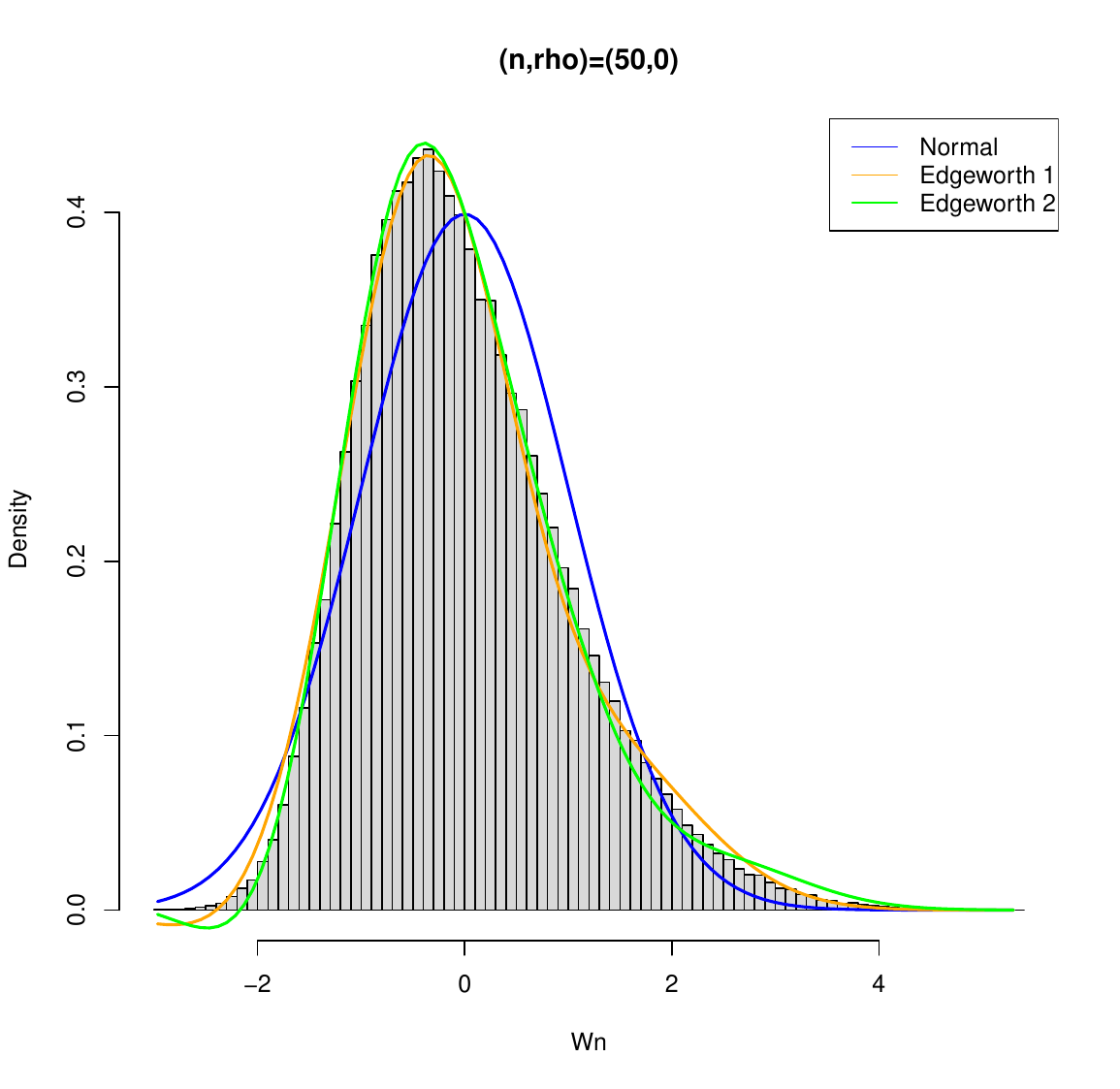}
			\caption{Chisq \& Chisq at $n=50$}
			\label{fig:image3}
		\end{subfigure}%
		\begin{subfigure}{0.5\textwidth}
			\centering
			\includegraphics[width=\linewidth]{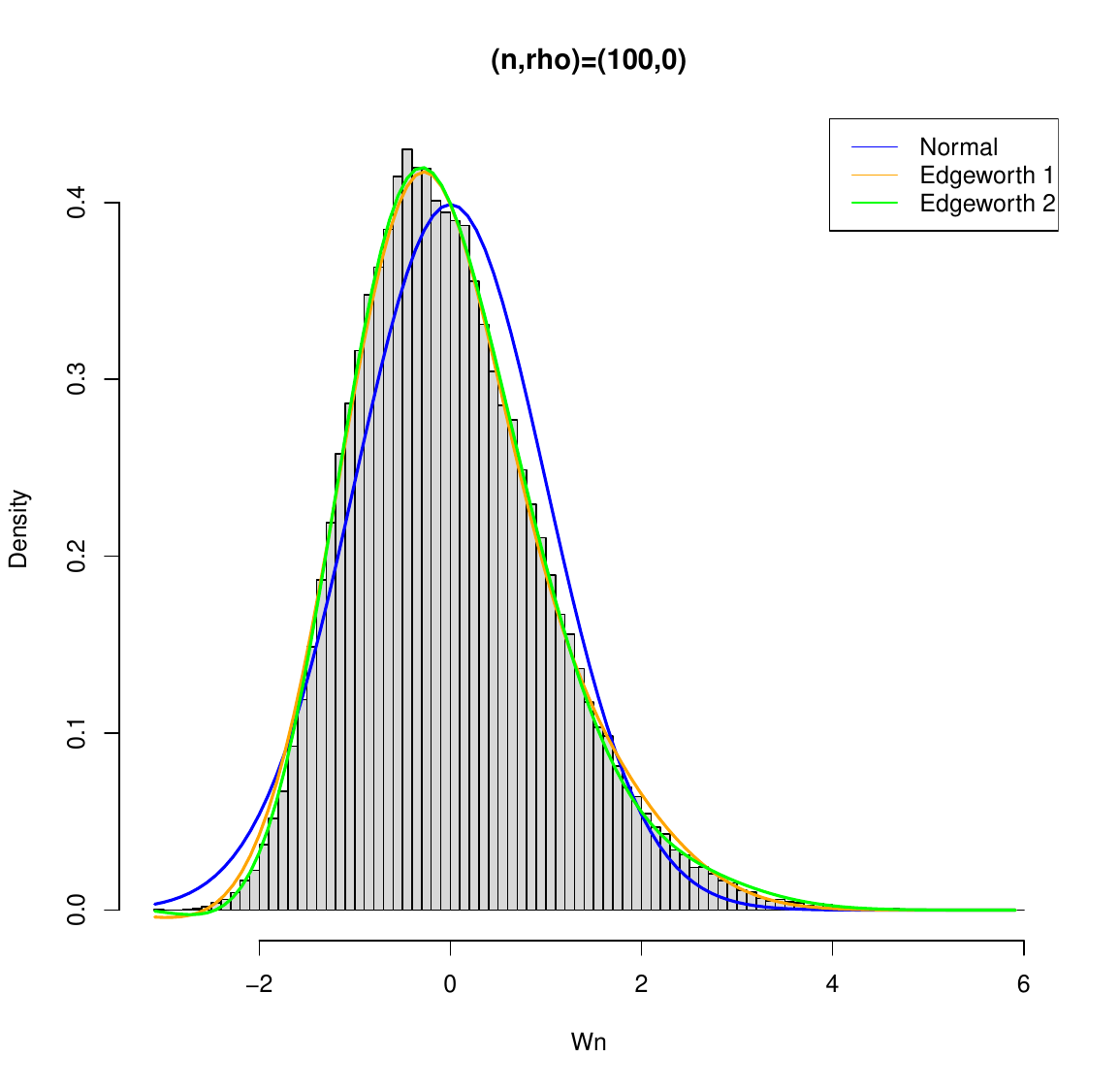}
			\caption{Chisq \& Chisq at $n=100$}
			\label{fig:image4}
		\end{subfigure}
		\caption{The Edgeworth expansion for continuous-continuous case at $n=50, 100$}
		\label{fig:both_images2}
	\end{figure}
	
	\textbf{Experiment 2 (Continuous and discrete random variables )} In this experiment, we generate two independent and identically distributed random variables, one discrete and the other continuous. Consider two specific random variables, one following $\chi^2(1)$ and the other following $Poisson(1)$. Define
	\[\Z=(X,Y,X^2,Y^2,XY),\]
	and $\Z_i=(X_{i},Y_{i},X_{i}^2, Y_{i}^2,X_{i}Y_{i})$. Suppose
	\begin{equation*}
		\setlength{\arraycolsep}{4pt}
		\mu=(1,1,2,3,1),\quad\Sigma=\begin{pmatrix}
			1&0&3&0&1\\
			0&2&0&12&2\\
			3&0&11&0&3\\
			0&12&0&96&12\\
			1&2&3&12&5
		\end{pmatrix}.
	\end{equation*}
	Based on our previous theorem, we can find that $\Z_i$ satisfy GPCC.
	
	\begin{figure}
		\begin{subfigure}{0.5\textwidth}
			\centering
			\includegraphics[width=\linewidth]{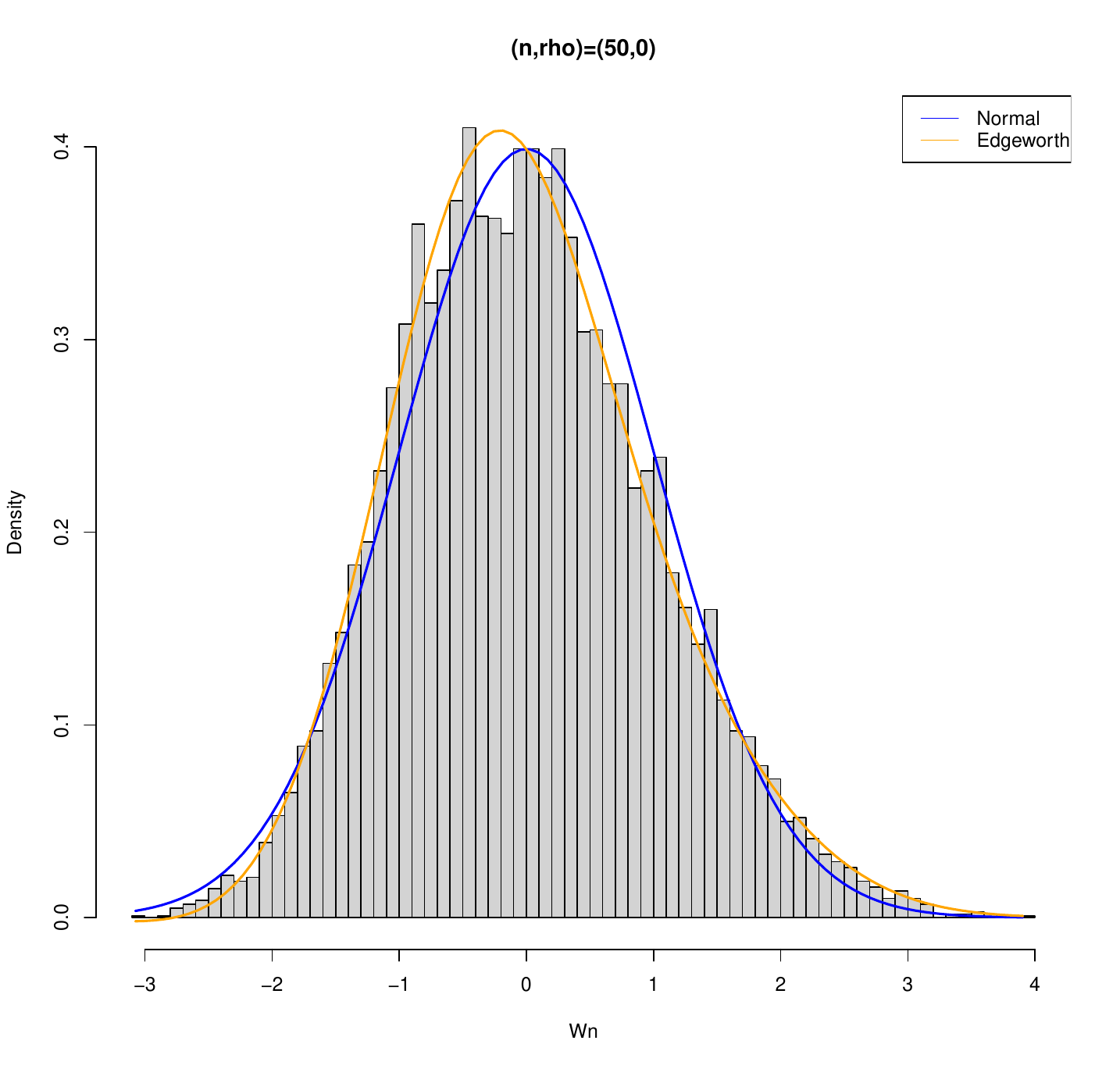}
			\caption{Poisson \& Chisq case at $n=50$}
			\label{fig:image5}
		\end{subfigure}%
		\begin{subfigure}{0.5\textwidth}
			\centering
			\includegraphics[width=\linewidth]{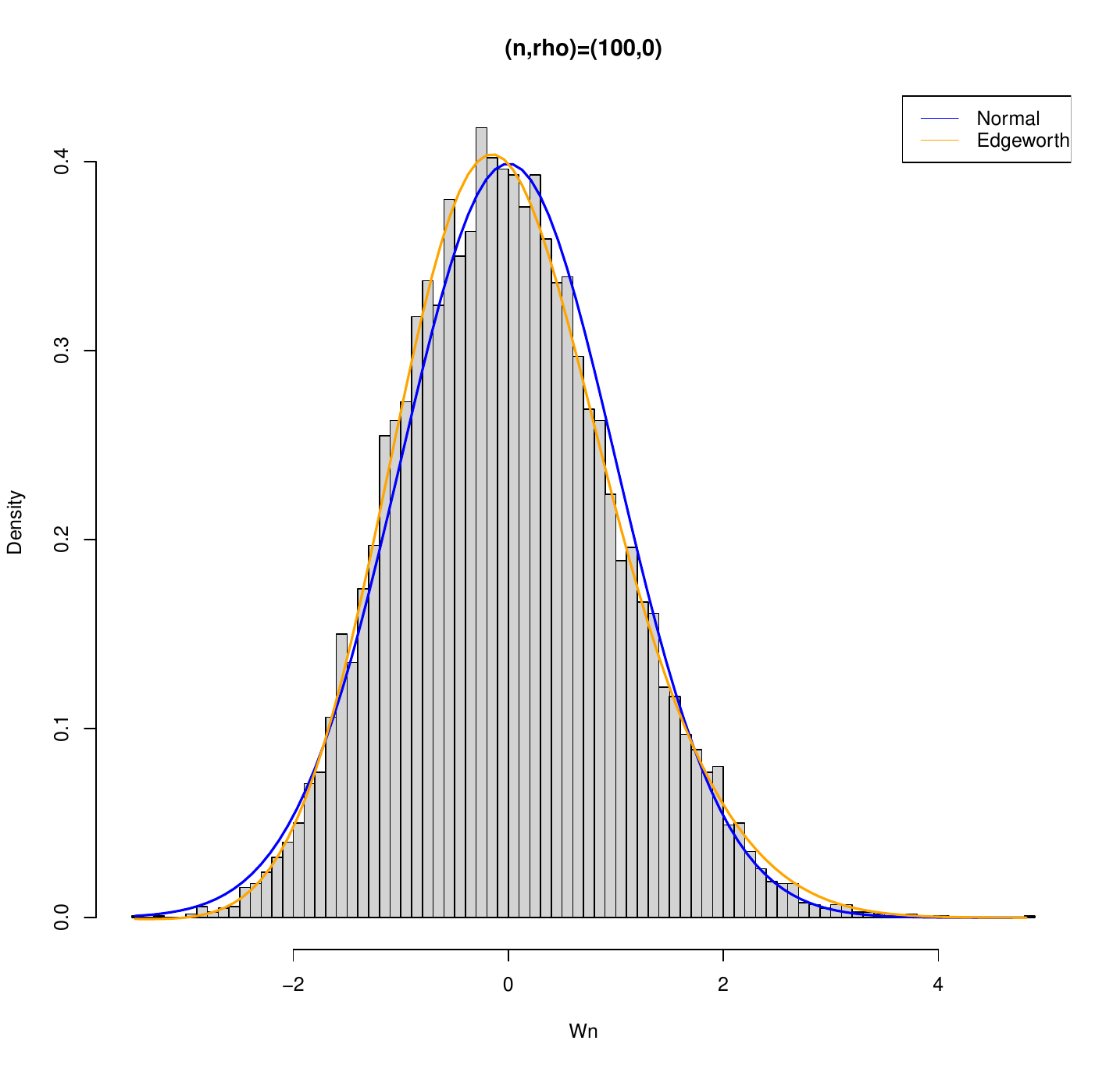}
			\caption{Poisson \& Chisq case at $n=100$}
			\label{fig:image6}
		\end{subfigure}
		
		\begin{subfigure}{0.5\textwidth}
			\centering
			\includegraphics[width=\linewidth]{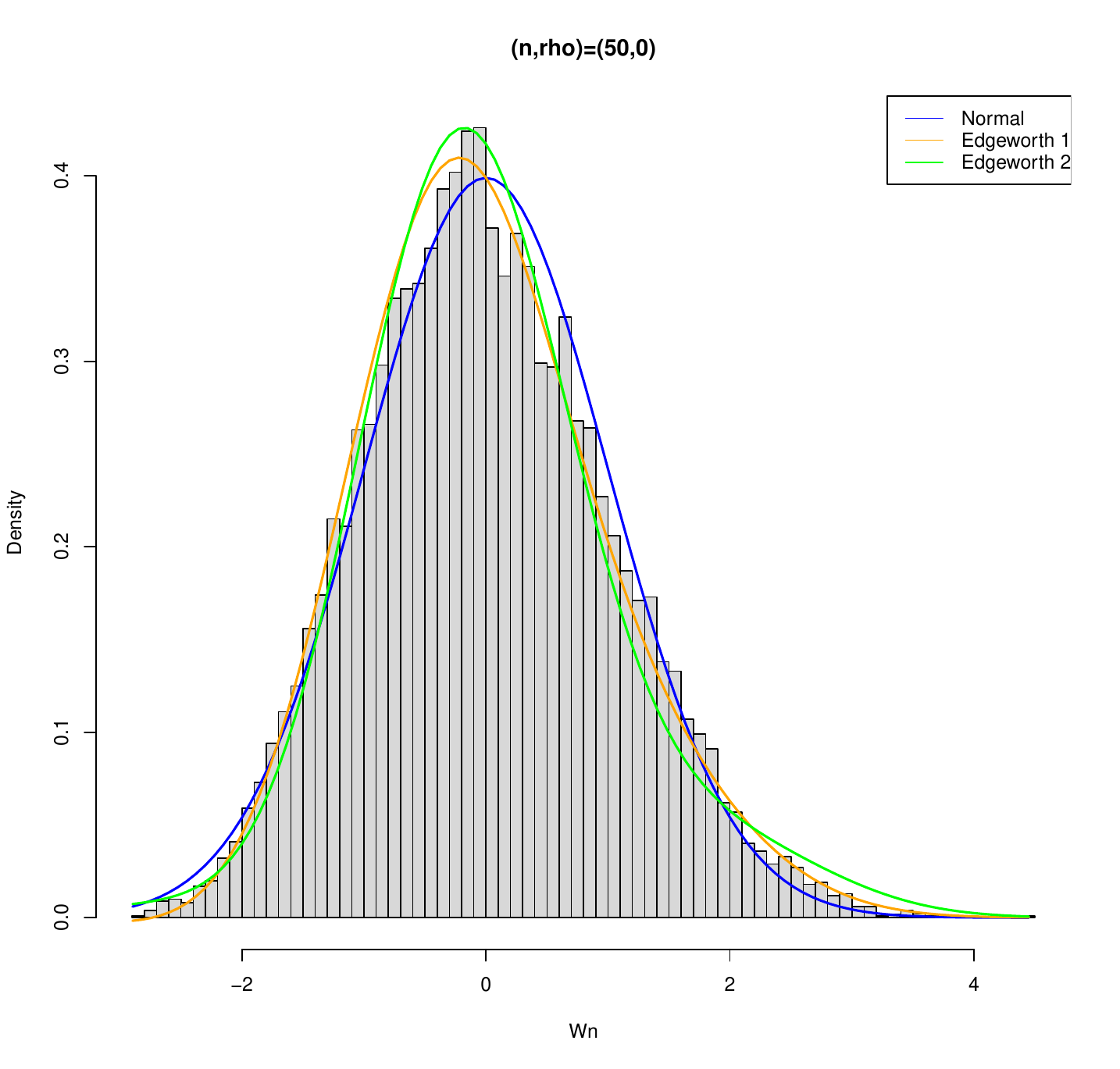}
			\caption{Poisson \& Chisq at $n=50$}
			\label{fig:image7}
		\end{subfigure}%
		\begin{subfigure}{0.5\textwidth}
			\centering
			\includegraphics[width=\linewidth]{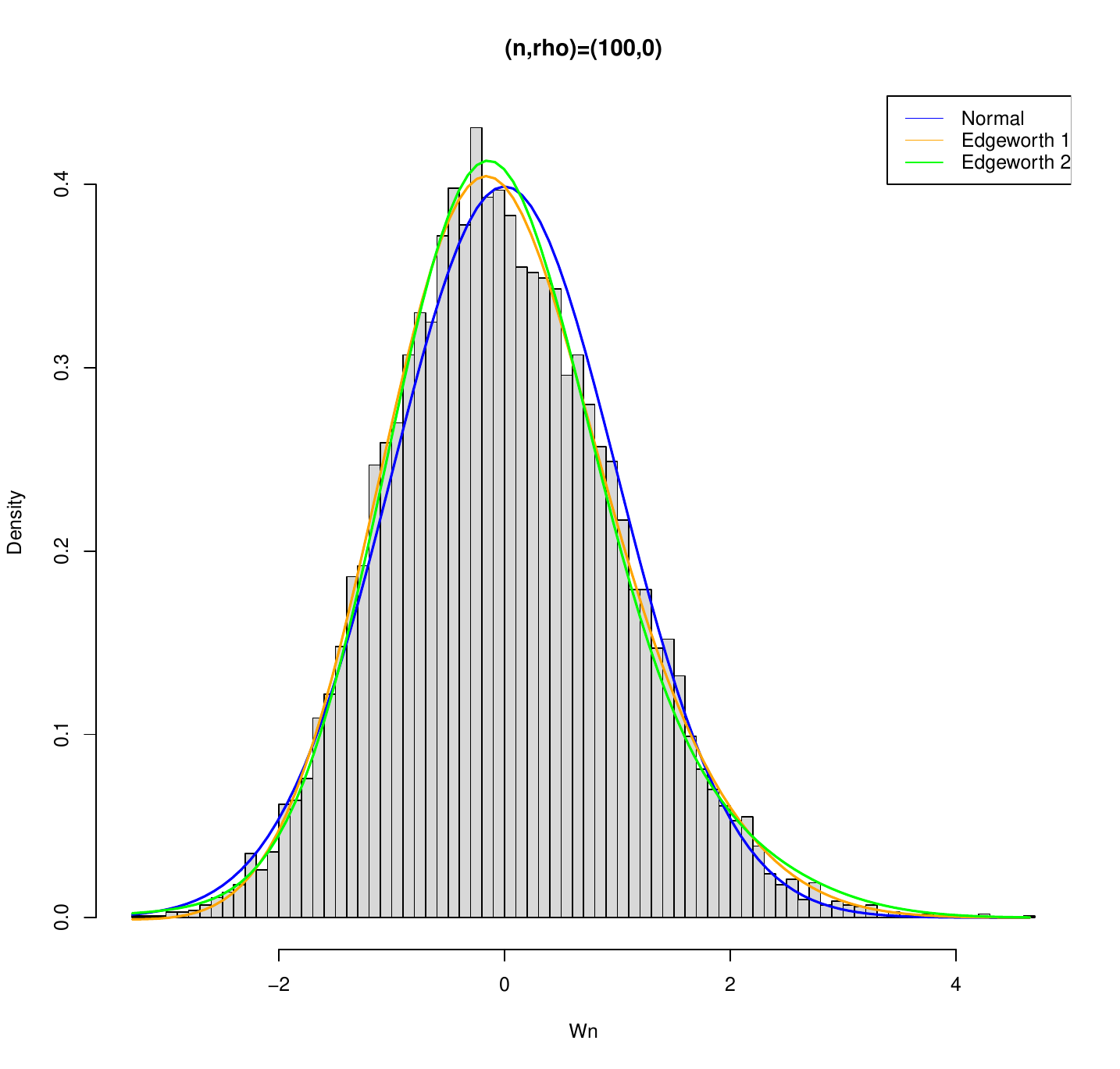}
			\caption{Poisson \& Chisq at $n=100$}
			\label{fig:image8}
		\end{subfigure}
		\caption{The Edgeworth expansions for continuous-discrete case at $n=50, 100$}
		\label{fig:both_images4}
	\end{figure}
	
	Next, we select the statistic $W_n=n^{1/2}(\hat\rho-\rho)$, and analyze the fit of different asymptotic distributions to it. The value of each coefficient is calculated using 10000 samples to obtain the expression of Edgeworth expansions. We then do a simulation experiment with a small sample. The parameters
	\[n\in\{50,100\}\]
	are chosen, so that $n$ is neither too small for asymptotics to be meaningful nor too large to distinguish $\Psi_{s,n}(x)$ and $\Phi(x)$. This is an ideal example for illustrating the performance of Edgeworth expansions, because the sample size is small, and the normal approximation is inaccurate.
	
	In Figure \ref{fig:both_images2}, the histograms depict the empirical distribution of $W_n$. The blue curve represents the density function of the standard normal distribution, while the orange curve denotes the probability density function of the first-order Edgeworth expansion. The apparent deviation between the orange and blue curves indicates that the first-order Edgeworth expansion fits the distribution of $W_n$ more accurately than the standard normal distribution.
	
	Similarly, the green curve represents the probability density function of the second-order Edgeworth expansion. While the discrepancy between the orange and green curves is minimal, the nuanced differences reveal that the fit of the second-order Edgeworth expansion to the distribution of $W_n$ is better than that of the first-order Edgeworth expansion. Figure \ref{fig:both_images4} shows the same results as Figure \ref{fig:both_images2}. 
	
	Therefore, based on our simulation results, we verify that the first-order Edgeworth expansion is more accurate than the normal distribution, and the second-order Edgeworth expansion is more accurate than the first-order Edgeworth expansion.

	\subsection{Simulation experiments for ratio of sample means}
	In this section, we consider the case $k=2$ to evaluate the performance of the Edgeworth expansion for the ratio of sample means through numerical experiments. We present the results of both the first-order and second-order Edgeworth expansions. For comparison, we also present the results of the normal approximation. 
	
	We begin by generating three independent and identically distributed random variables: one continuous and two discrete. Specifically, one follows a $\chi^2(1)$ distribution and the other two follow a $Poisson(1)$ distribution. Define 
	\[\Z=(X,X^2,Y_1,Y_2).\]
	Assume
	\begin{equation*}
		\setlength{\arraycolsep}{4pt}
		\mu=(1,3,1,1),\quad\Sigma=\begin{pmatrix}
			3&15&1&1\\
			15&105&3&3\\
			1&3&2&1\\
			1&3&1&2
		\end{pmatrix}.
	\end{equation*}
	Let $\Z_i=(X_{i},X_{i}^2, Y_{1i},Y_{2i})$. Based on our previous theorem, we find that $\Z_i$ satisfies GPCC. We then consider the statistic 
	\[W_n=n^{1/2}(H(\bar Z)-H(\mu)),\]
	and analyze the fit of different asymptotic distributions to it. The coefficients are calculated using 10000 samples to obtain the Edgeworth expansions. We then conduct a simulation experiment with a small sample. The parameters
	\[n\in\{100, 200, 300, 500\}\]
	are chosen, so that $n$ is neither too small for asymptotics to be meaningful nor too large to distinguish between $\Psi_{s,n}(x)$ and $\Phi(x)$. This provides an ideal example for illustrating the performance of Edgeworth expansions because the sample size is small and the normal approximation is inaccurate.

	\begin{figure}
		\begin{subfigure}{0.5\textwidth}
			\centering
			\includegraphics[width=\linewidth]{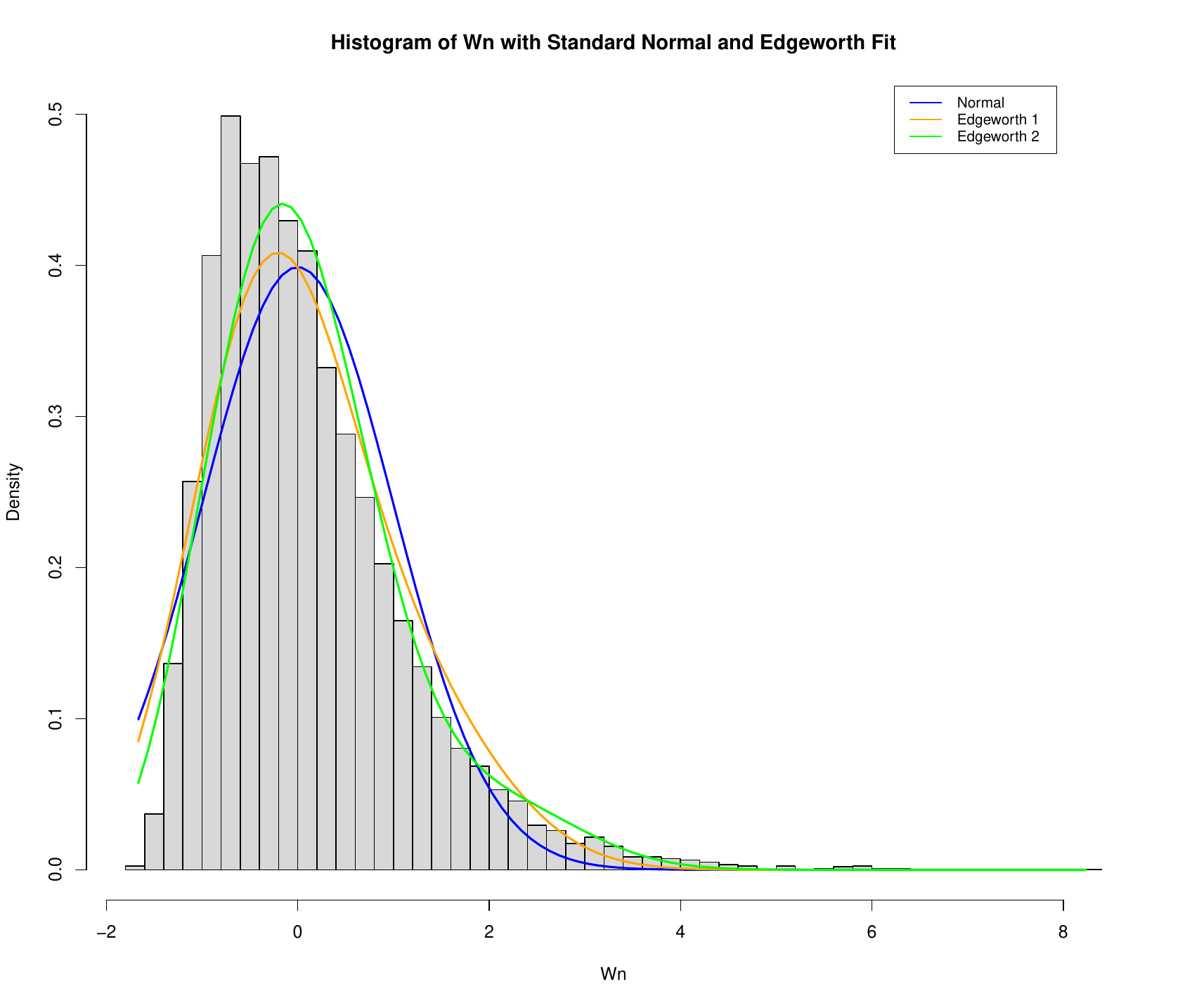}
			\caption{Poisson \& Chisq case at $n=100$}
			\label{fig:image5}
		\end{subfigure}%
		\begin{subfigure}{0.5\textwidth}
			\centering
			\includegraphics[width=\linewidth]{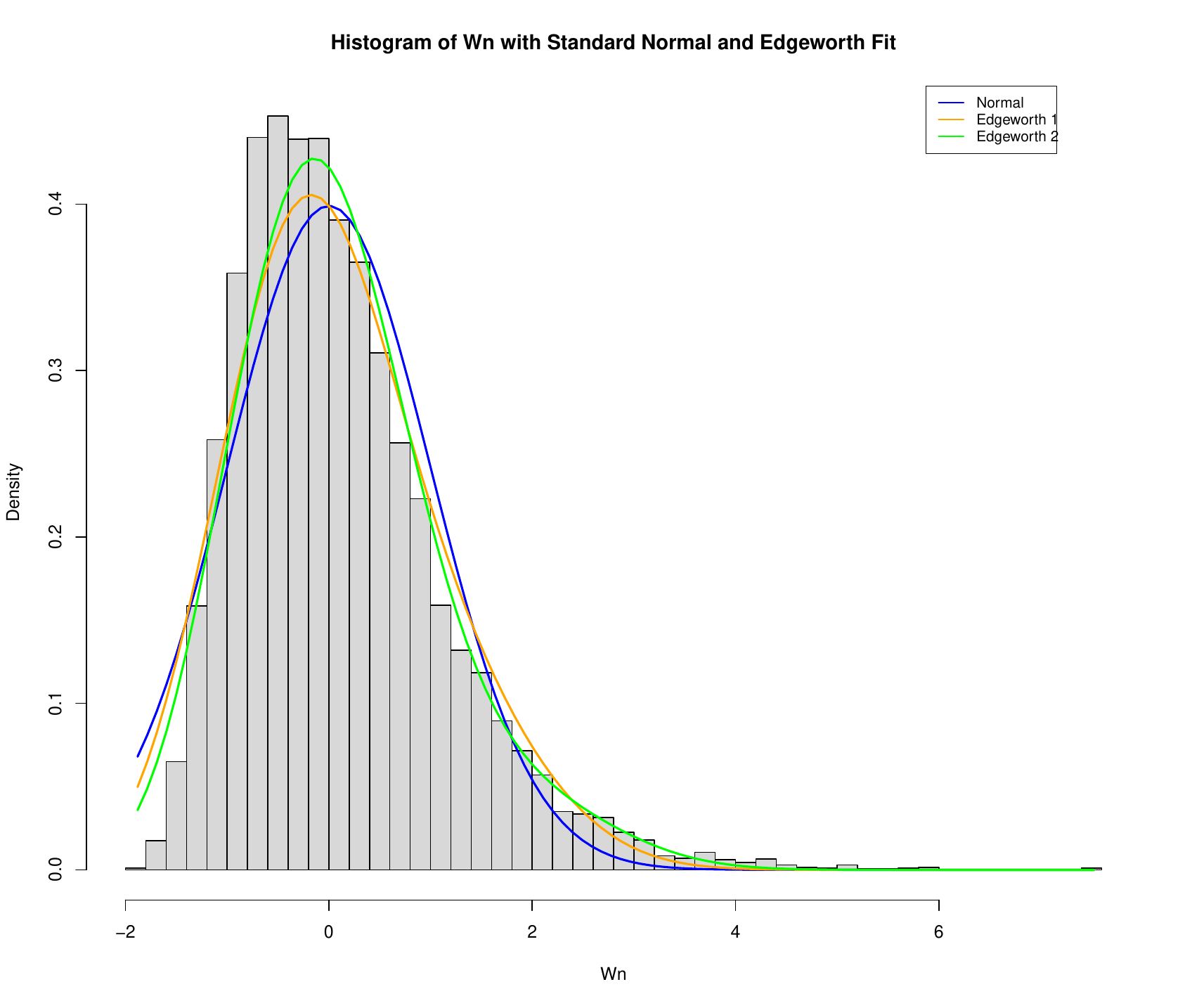}
			\caption{Poisson \& Chisq case at $n=200$}
			\label{fig:image6}
		\end{subfigure}
		
		\begin{subfigure}{0.5\textwidth}
			\centering
			\includegraphics[width=\linewidth]{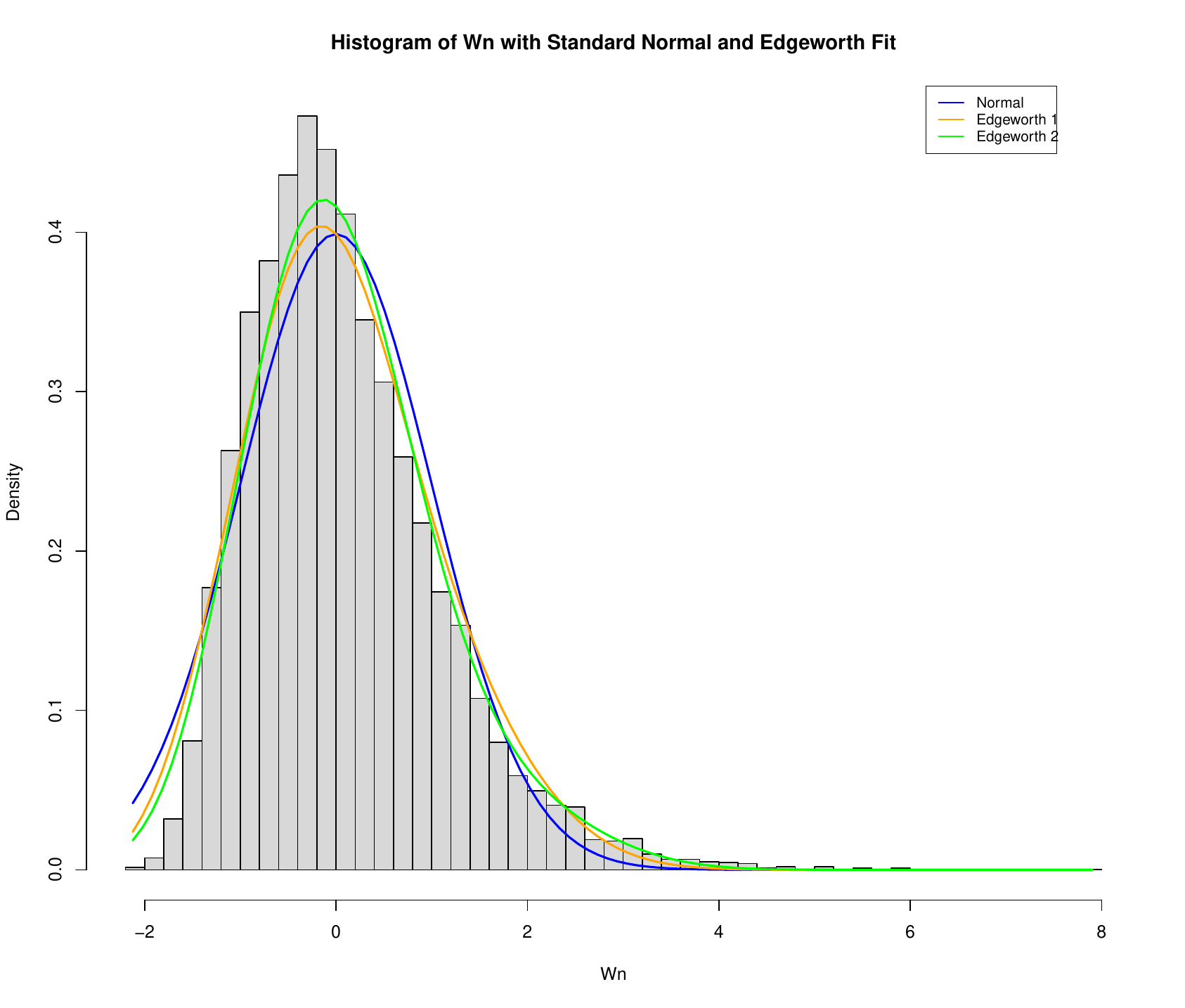}
			\caption{Poisson \& Chisq at $n=300$}
			\label{fig:image7}
		\end{subfigure}%
		\begin{subfigure}{0.5\textwidth}
			\centering
			\includegraphics[width=\linewidth]{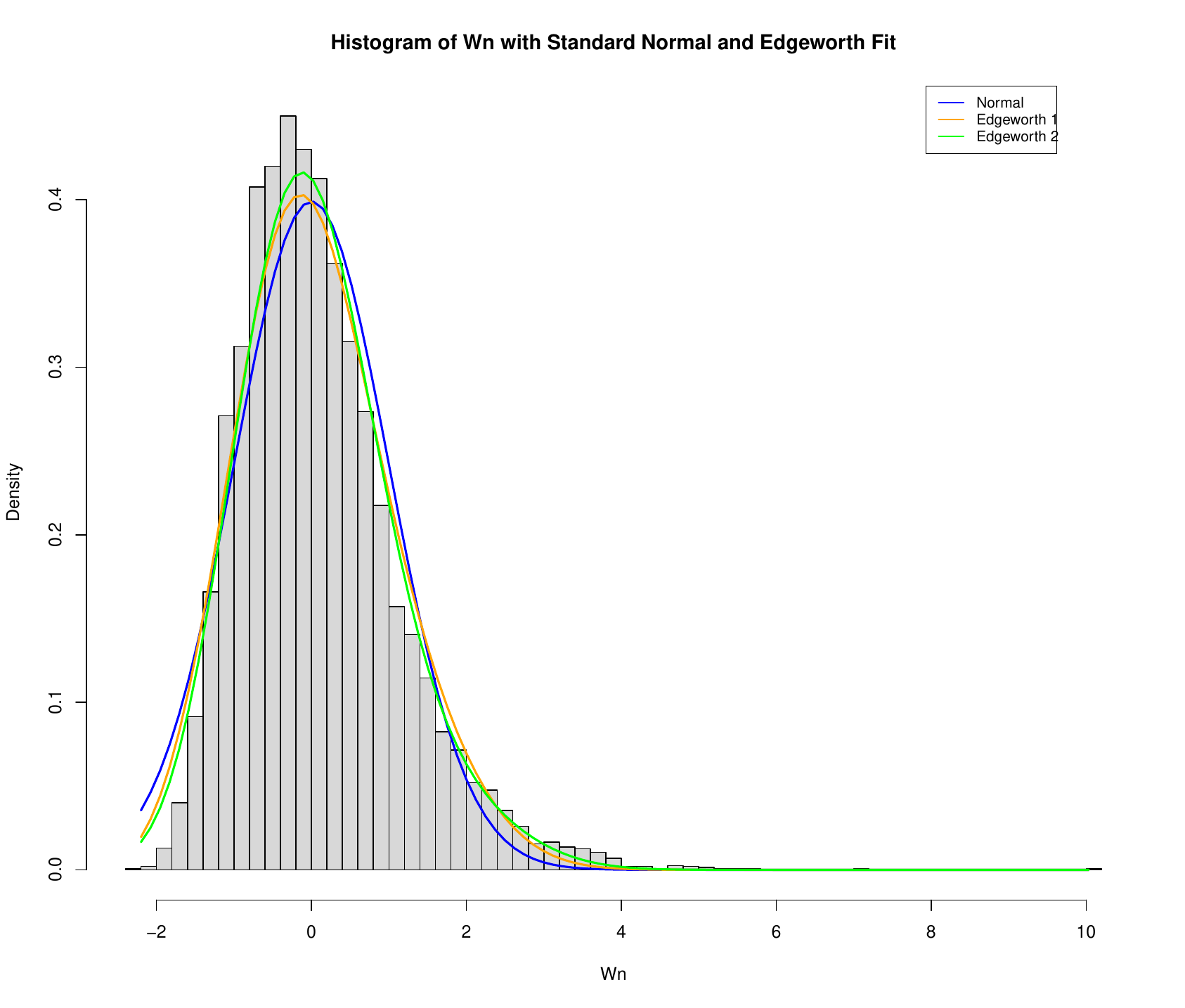}
			\caption{Poisson \& Chisq at $n=500$}
			\label{fig:image8}
		\end{subfigure}
		\caption{The Edgeworth expansions for the ratio of sample means for different sample sizes}
		\label{fig:both_images6}
	\end{figure}
	
	In Figure \ref{fig:both_images6}, the histograms depict the empirical distribution of $W_n$. The blue curves illustrate the density function of the standard normal distribution, while the orange curves denote the probability density function of the first-order Edgeworth expansion, and the green curves represent the probability density function of the second-order Edgeworth expansion. A notable deviation exists between the orange and blue curves, suggesting that the first-order Edgeworth expansion provides a more accurate fit for the distribution of $W_n$ compared to the standard normal distribution. Similarly, the difference between the orange and green curves indicates that the second-order Edgeworth expansion offers a superior fit for the distribution of $W_n$ relative to the first-order Edgeworth expansion. 
	
	\subsection{Simulation experiments for Z-score test statistic}
	In this subsection, we use numerical experiments to evaluate the performance of the Edgeworth expansion for the Z-score test statistic. We present the results of both the first-order and second-order Edgeworth expansions. For comparison, we also present the results of the normal approximation. 
	
	In this experiment, we generate two independent and identically distributed continuous random variables $X$ and $Y$, each following a $N(0,1)$ distribution, and generate data of size $n$. Suppose that 
	\begin{equation*}
		\setlength{\arraycolsep}{4pt}
		\mu=(0,1,0,1),\quad\Sigma=\begin{pmatrix}
			1&0& 0& 0\\
			0&3& 0& 1\\
			0& 0& 1& 0\\
			0&1& 0& 3
		\end{pmatrix}.
	\end{equation*}
	
	Let $\Z_i=(X_i, X_i^2, Y_i,Y_i^2)$. Based on our previous theorem, we find that $\Z_i$ satisfies GPCC. We then consider the statistic 
	\[W_n=n^{1/2}(H(\bar Z)-H(\mu)),\]
	and analyze the fit of different asymptotic distributions to it. The coefficients are calculated using 10000 samples to obtain the Edgeworth expansions. We then conduct a simulation experiment with a small sample. The parameters $n_1/n_2=1/4$ and
	\[n_2\in\{5, 10,15,20\}\]
	are chosen, so that $n_2$ is neither too small for asymptotics to be meaningful nor too large to distinguish between $\Psi_{s,n}(x)$ and $\Phi(x)$. This is an ideal example for illustrating the performance of Edgeworth expansions because the sample size is small, and the normal approximation is inaccurate.
	\begin{figure}
		\begin{subfigure}{0.5\textwidth}
			\centering
			\includegraphics[width=\linewidth]{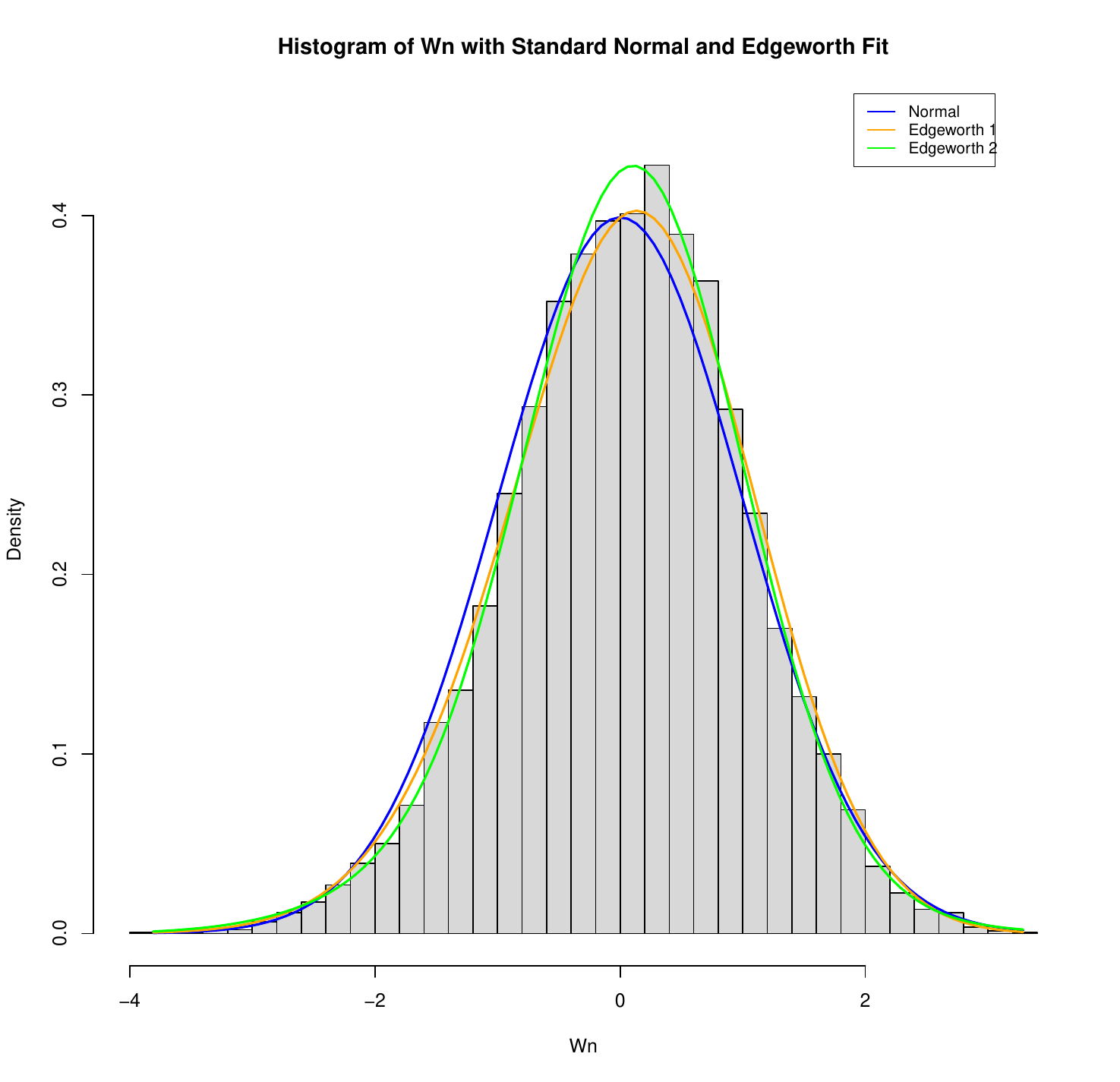}
			\caption{Normal \& Normal at $n=5$}
			\label{fig:image7}
		\end{subfigure}%
		\begin{subfigure}{0.5\textwidth}
			\centering
			\includegraphics[width=\linewidth]{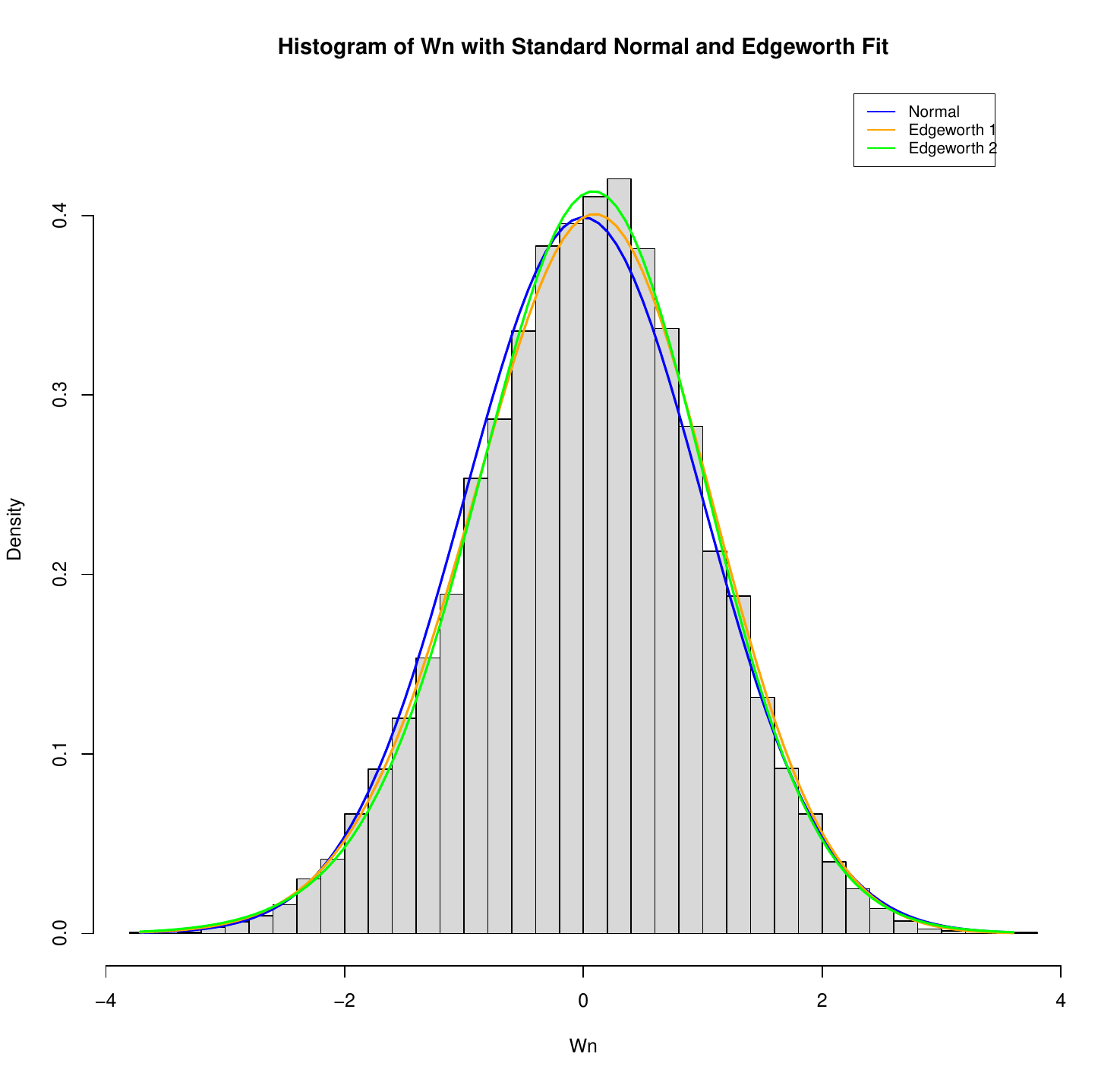}
			\caption{Normal \& Normal at $n=10$}
			\label{fig:image8}
		\end{subfigure}
		
		\begin{subfigure}{0.5\textwidth}
			\centering
			\includegraphics[width=\linewidth]{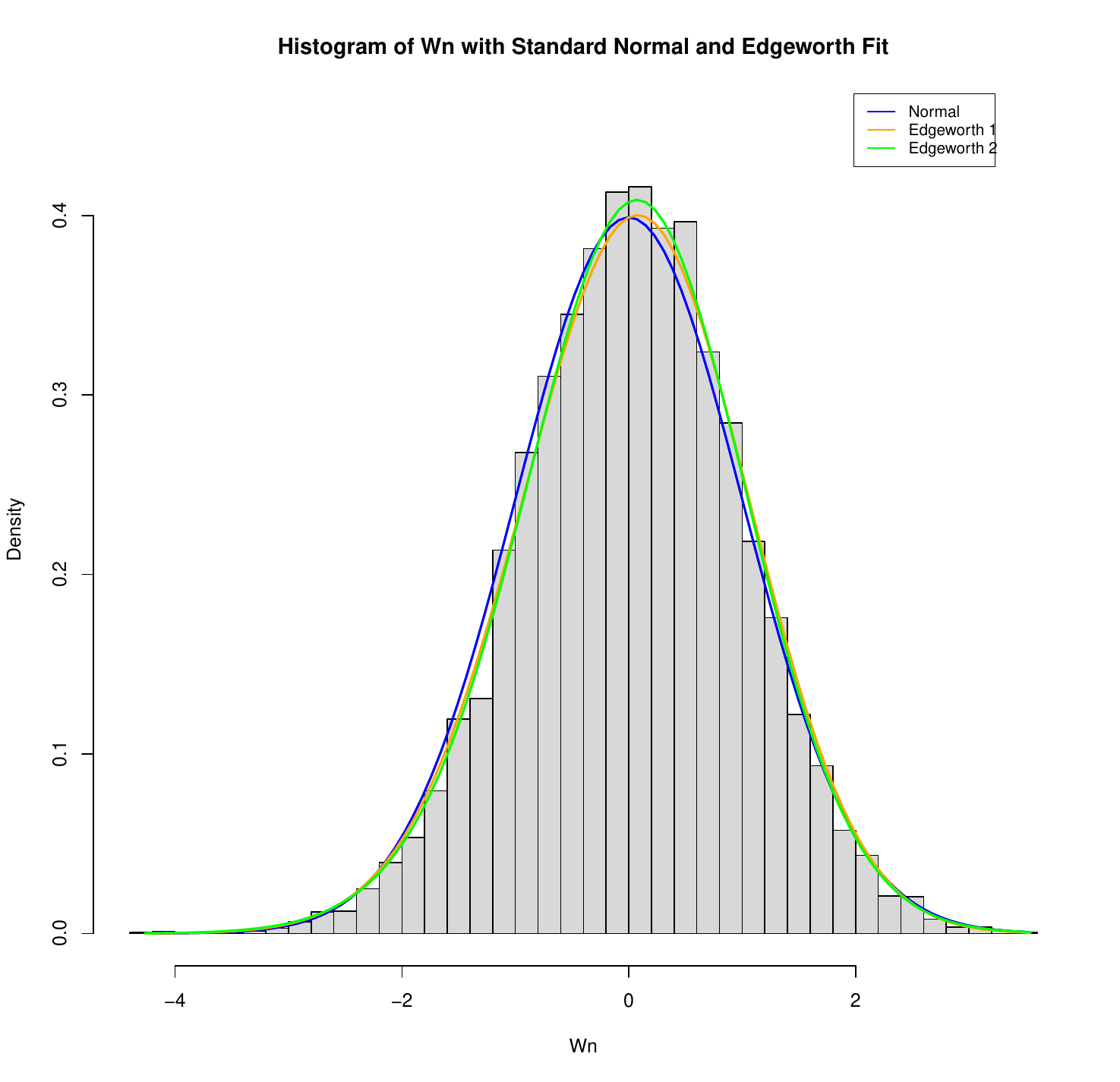}
			\caption{Normal \& Normal at $n=15$}
			\label{fig:image7}
		\end{subfigure}%
		\begin{subfigure}{0.5\textwidth}
			\centering
			\includegraphics[width=\linewidth]{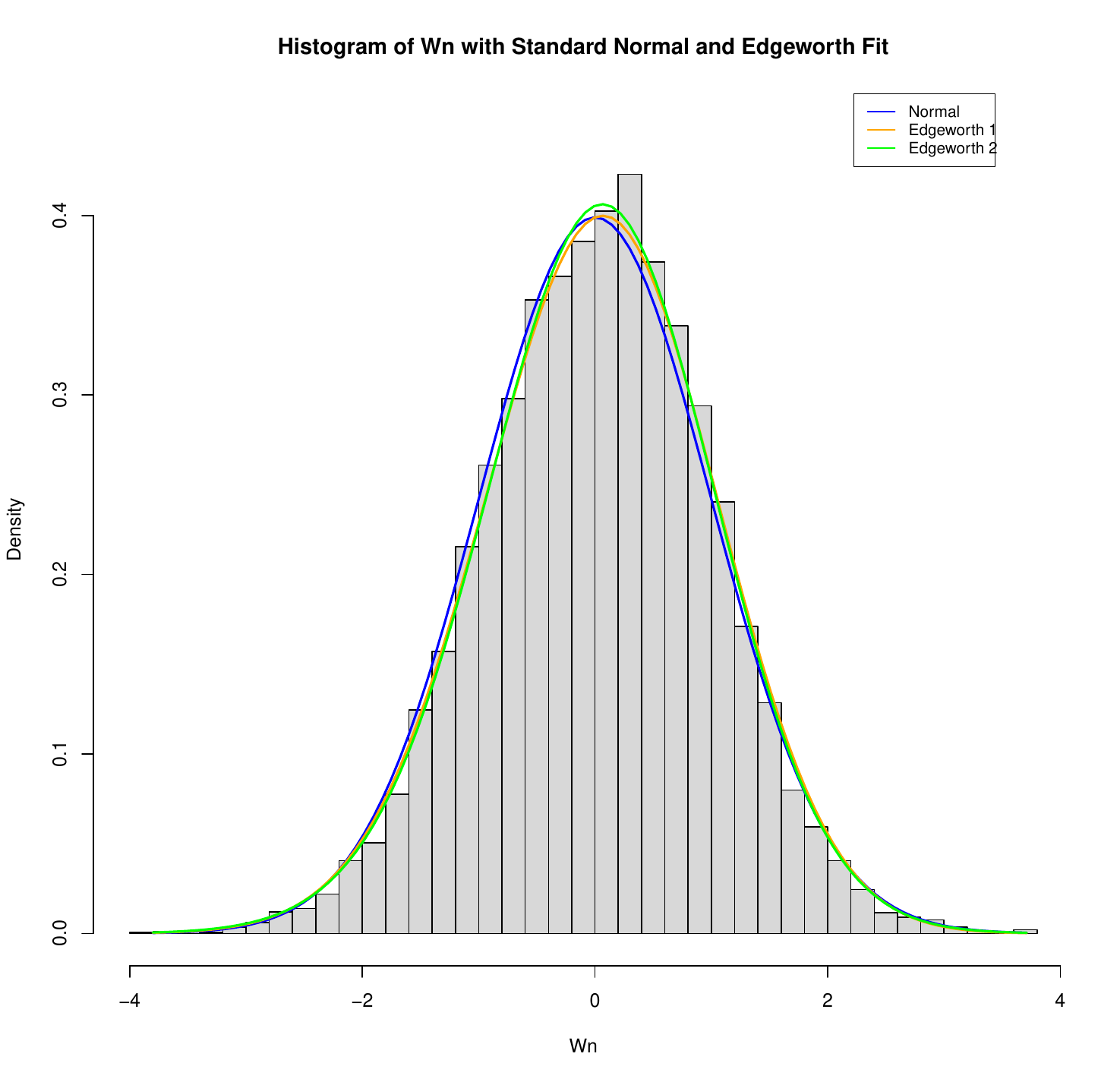}
			\caption{Normal \& Normal at $n=20$}
			\label{fig:image8}
		\end{subfigure}
		\caption{The Edgeworth expansions for the Z-score statistic for different sample sizes}
		\label{fig:both_images8}
	\end{figure}

	In Figure \ref{fig:both_images8}, the histograms depict the empirical distribution of $W_n$. The blue curves illustrate the density function of the standard normal distribution, while the orange curves denote the probability density function of the first-order Edgeworth expansion, and the green curves represent the probability density function of the second-order Edgeworth expansion. A notable deviation exists between the orange and blue curves, suggesting that the first-order Edgeworth expansion provides a more accurate fit for the distribution of $W_n$ compared to the standard normal distribution. Similarly, the difference between the orange and green curves indicates that the second-order Edgeworth expansion offers a superior fit for the distribution of $W_n$ relative to the first-order Edgeworth expansion.

	\section{Technical proof of Theorem \ref{distance}}\label{sec6}
	In this section, we present the core of the proof of Theorem \ref{distance}. Throughout the proofs, we use $C$ to denote an absolute constant that may vary with each occurrence.

	Assume $P^*$ is a conditional probability given $C_n$. Let $\bm{\mu}_{s}$ be the $s$-th moment of $\textbf{Z}_j$, $\bm\rho_{s}$ be the $s$-th  absolute moment of $\Z_j $ and $\bm\chi_{s}$ be the $s$-th cumulant of $\Z_j(1\le j\le n)$. Namely, we write
	\[\bm\rho_{s}=\mathbb{E}\Vert \Z_j\Vert^s, \quad \bm\mu_{s}=\mathbb{E}\Z_{j}^{s},\quad (1\le j\le n).\]
	Define truncated random vectors
	\begin{equation*}
		\hat \Z_{j}=\left\{
		\begin{aligned}
			\Z_j &, &\|\Z_j\|\le n^{1/2}\\
			0 &, &\|\Z_j\|>n^{1/2},
		\end{aligned}
		\right.
		\quad \tilde \Z_{j}=\hat \Z_{j}-\mathbb{E}\hat \Z_{j}\quad (1\le j\le n).
	\end{equation*}
	And then let $\hat\bmmu_{s,j}$ be the $s$-th moment of $\hat \Z_j$, $\hat\bmrho_{s,j}$ be the $s$-th  absolute moment of $\hat \Z_j $ and $\hat\bmchi_{s,j}$ be the $s$-th cumulant of $\hat \Z_j (1\le j\le n)$. Besides, let $\tilde\bmmu_{s,j}$ be the $s$-th moment of $\tilde \Z_j$, $\tilde\bmrho_{s,j}$ be the $s$-th  absolute moment of $\tilde \Z_j $ and $\tilde\bmchi_{s,j}$ be the $s$-th cumulant of $\tilde \Z_j (1\le j\le n)$. Namely, write
	\[\hat\bmmu_{s,j}=\mathbb{E}\hat \Z_{j}^{s},\quad\hat\bmrho_{s,j}=\mathbb{E}\Vert \hat\Z_j\Vert^s,\quad \tilde\bmmu_{s,j}=\mathbb{E}\tilde \Z_{j}^{s},\quad\tilde\bmrho_{s,j}=\mathbb{E}\Vert \tilde\Z_j\Vert^s.\]
	Also introduce
	\[\Delta_{n,j,s}=\int_{\{\Vert \Z_j\Vert>n^{1/2}\}}\Vert \Z_j\Vert^s,\quad \bar\Delta_{n,s}=n^{-1}\sum_{j=1}^n\Delta_{n,j,s},\]
	\[\bar \Delta_{n,s}(\epsilon)=n^{-1}\sum_{j=1}^n\int_{\{\|\Z_j\|>\epsilon n^{1/2}\}}\|\Z_j\|^s\quad (\epsilon>0).\]
	Finally, let $B_n$ be the common covariance matrix of $\hat \Z_{1}$ and $\tilde \Z_{1}$. The symbol $*$ denotes the convolution operation. And we define the norm of a $k\times k$ matrix $T$. Namely,
	\[B_n=Cov(\hat \Z_{1})=Cov(\tilde \Z_{1}),\quad \|T\|=\sup_{x\in R^k,~\|x\|\le 1}\|Tx\|.\]
	
	\begin{lemma}\label{P_ra}
		Assume $P_{r}$ is the formal Edgeworth expansion of the probability distribution of $W_{n}'$. Similarly, let $P_{ra}$ represent the formal Edgeworth expansion of the conditional probability distribution of $W_{n}'$, given the last $k-a$ components of $\Z_j$.  Then, we can obtain:
		\[\mathbb{E}P_{ra}=P_{r}+o(n^{-(s-2)/2}).\]
	\end{lemma}
	
	\begin{proof}Recall the definition of $W_{n}'$,
		\begin{align*}
			W_n'=n^{1/2}\Big(&\sum_{i=1}^k l_i(\bar Z_i-\mu_i)+\frac12\sum_{i,j=1}^k l_{i,j}(\bar Z_i-\mu_i)(\bar Z_j-\mu_j)+\cdots\\\nonumber
			&+\frac{1}{(s-1)!}\sum_{i_1,\dots,i_{s-1}=1}^kl_{i_1,\dots,i_{s-1}}(\bar Z_{i_1}-\mu_{i_1})\cdots(\bar Z_{i_{s-1}}-\mu_{i_{s-1}})\Big).
		\end{align*}
		Next, we define $W'_{na}$ as the Taylor expansion of $W_n$, given the last $k-a$ components of $\Z_j$,
		\begin{align} \label{Taylorexpansion2}
			W_{na}'&=n^{1/2}\Big(\sum_{i=1}^a m_i(\bar Z_i-\mu_i)+\frac12\sum_{i,j=1}^a m_{i,j}(\bar Z_i-\mu_i)(\bar Z_j-\mu_j)+\cdots\\\nonumber
			&\qquad +\frac{1}{(s-1)!}\sum_{i_1,\dots,i_{s-1}=1}^a m_{i_1,\dots,i_{s-1}}(\bar Z_{i_1}-\mu_{i_1})\cdots(\bar Z_{i_{s-1}}-\mu_{i_{s-1}})\Big).
		\end{align}
		Notice that
		\begin{align*}
			&\mathbb{E}^*(W_{na}')^j=\mathbb{E}^*(W_n)^j+o(n^{-(s-2)/2}),\quad \mathbb{E}[\mathbb{E}^*(W_n)^j]=\mathbb{E}(W_n)^j
		\end{align*}
		and
		\begin{align*}
			&\mathbb{E}(W_n')^j=\mathbb{E}(W_n)^j+o(n^{-(s-2)/2}).
		\end{align*}
		Hence, we obtain
		\begin{align*}
			\mathbb{E}[\mathbb{E}^*(W_{na}')^j]=\mathbb{E}(W_n')^j+o(n^{-(s-2)/2}).
		\end{align*}
		Let $\tilde\kappa^*_{j,n}$ be the $j$-th conditional cumulant of $W_{na}'$ and $\kappa^*_{j,n}$ be the $j$-th conditional cumulant of $W_{n}'$. Then, we obtain 
		\[\mathbb{E}\tilde\kappa^*_{j,n}=\mathbb{E} \kappa^*_{j,n}+o(n^{-(s-2)/2}),\]
		where $\mathbb{E}\tilde \kappa^*_{j,n}=\sum_{i=1}^{s-2}n^{-i/2}b_{j,i}+o(n^{-(s-2)/2})$ when $ j\neq 2$, while
		\[\mathbb{E}\tilde \kappa^*_{j,n}=\sigma^2+\sum_{i=1}^{s-2}n^{-i/2}b_{2,i}+o(n^{-(s-2)/2})\]
		when $ j=2$. Here, $b_{j,i}$ depend only on appropriate moments of $\Z_1$ and derivatives of $H$ at $\bmmu$ of orders $s-1$ and less. The expression
		\begin{equation*}
			\exp\Big(it\tilde\kappa^*_{1,n}+\frac{(it)^2}{2}\tilde\kappa^*_{2,n}+\sum_{j=3}^{s}\frac{(it)^j}{j!}\tilde\kappa^*_{j,n}\Big)
		\end{equation*}
		is an approximation of the conditional characteristic function of $W_{na}'$. Namely, we obtain an approximation of the conditional characteristic function of $W_n$ by appropriate conditional moments of $\Z_1$ and derivatives of $H$ at $\bmmu$ of orders $s-1$ and less.
		Notice that
		\begin{align*}
			&\mathbb{E}\exp\Big(it\tilde\kappa^*_{1,n}+\frac{(it)^2}{2}\tilde\kappa^*_{2,n}+\sum_{j=3}^{s}\frac{(it)^j}{j!}\tilde\kappa^*_{j,n}\Big)\\
			&=\exp\Big(it\tilde\kappa_{1,n}+\frac{(it)^2}{2}(\tilde\kappa_{2,n}-\sigma^2)+\sum_{j=3}^{s}\frac{(it)^j}{j!}\tilde\kappa_{j,n}\Big)\exp(-\sigma^2t^2/2)+o(n^{-(s-2)/2}).
		\end{align*}
		Then, repeating the process from \eqref{Taylor expansion} to \eqref{fourierexpansion}, we obtain the formal Edgeworth expansion expression $P_{ra}$ of the conditional distribution of the random vector $W_{na}'$. The expansion $P_{ra}$ is related to the conditional moments of $\Z_1$ of orders not greater than $j+2$. Furthermore, utilizing the relationship between moments and cumulants, we derive
		\begin{align*}
			\mathbb{E}P_{ra}=P_r+o(n^{-(s-2)/2}).
		\end{align*}
		Therefore, we obtain the result.
	\end{proof}

	Lemma \ref{P_ra} provides a direct link between the Edgeworth expansion of the conditional probability distribution of $W_n'$ and the Edgeworth expansion of its probability distribution. To establish the conclusion, it suffices to derive an upper bound for the following expression:
	\[\mathbb{E} \int  f_{a_n}\,d\Big(Q_n''-\sum_{r=0}^{s-2}n^{-r/2}P_{ra}(-\Phi_{0, B_n}:\{\tilde\bmchi_{\bfv}\})\Big).\]
	In other words, if we define
	\[H_n=Q_n''-\sum_{r=0}^{s+k-2}n^{-r/2}P_{ra}(-\Phi_{0, B_n}:\{\tilde\bmchi_{\bfv}\})\]
	as a new signed measure, we only need to estimate that the bound of $\mathbb{E}\int f_{a_n}\, dH_n$ is of order $o(n^{-(s-2)/2})$. 
	Next, by Lemma \ref{lemma:11.2}, we can obtain the upper bound of $E\int f_{a_n}\, dH_n$:
	\begin{align*}
		\Big| \mathbb{E}\int f_{a_n}\,dH_n\Big|\le &M_{s'}(f)\mathbb{E}\bigg(\int[1+(\|x\|+\epsilon+\|a_n\|)^{s'}]|H_n*K_{\epsilon}|\,dx\\
		&\qquad+\bar\omega_{f_{a_n}}\Big(2\epsilon:\Big|\sum_{r=0}^{s+k-2}n^{-r/2}P_r(-\Phi_{0,B_n}:\{\tilde\bmchi_{\bfv}\})\Big|\Big)\bigg)\quad (\epsilon>0),
	\end{align*} 
	where we choose the probability measure $K_{\epsilon}$ to satisfy
	\[K_{\epsilon}(\{x:\|x\|<\epsilon\})=1,\]
	\[D^\alpha\hat K_{\epsilon}(t)\le C\epsilon^{|\alpha|}\exp(-(\epsilon\|t\|)^{1/2})\quad (t\in R^a,\, \|\alpha\|\le s+a+1).\]
	And $\hat K_{\epsilon}$ is the Fourier-Stieltjes transform of $K_{\epsilon}$. This is possible by Corollary 10.4 of \cite{bhattacharya2010normal}.
	
	Next, the proof can be divided into two parts. The first part of the proof is to estimate the bound of the moments of $H(x)$ after smoothing it with $K_{\epsilon}$. The second part of the proof is to estimate the average modulus of oscillation $\bar\omega_{f_{a_n}}$.
	
	$\mathbf{Step~1.}$~ According to the result of Lemma \ref{lemma:11.6}, we derive the following bound:
	\begin{align}\label{1.1.1}
		\mathbb{E}\int[1+(\|\x\|+\epsilon+\|a_n\|)^{s'}]|H_n*K_{\epsilon}|\,d\x\le C\max_{0\le|\beta|\le a+s+1}\mathbb{E}\int|D^\beta(\hat H_n\hat K_{\epsilon})(\bm{t})|d\bm{t},
	\end{align}
	where $\hat H_n$ is the Fourier-Stieltjes transform of $H_n$. Additionally, according to Leibniz's rule for differentiation, if $\alpha\in N^a$ and $\beta\in N^a$, we rewrite: 
	\[D^\beta(\hat H_n\hat K_{\epsilon})=\sum_{0\le\alpha\le\beta}C(D^{\beta-\alpha}\hat H_n)(D^\alpha\hat K_{\epsilon}).\]

	Write $c_n=n^{1/2}/(16\rho_3)$, then we continue to calculate the remaining integral. According to Lemma \ref{lemma: 14.1}, Lemma \ref{lemma: 7.2} and Lemma \ref{An}, we obtain:
	\[A_n\ge\frac{C(s,a)n^{(1/2)(s-2)/(s+a-1)}}{\rho_s^{1/(s+a-1)}},\quad c_n\ge A_n.\]
	Specific proof can be found in Appendix \ref{appdx3}. Next, using Lemma \ref{lemma: 14.4} and the relationship $P_r(-\phi_{0,\V}:\{\bm\chi_\bfv\})=\tilde P_r(-\bfD:\{\bm\chi_\bfv\})\phi_{0,\V}$, we obtain:
	\begin{align*}
		\mathbb{E}\int_{\{\|\bm{t}\|>A_n\}}&|D^{\beta-\alpha}\hat H_n(\bm{t}) D^\alpha\hat K_{\epsilon}(\bm{t})|d\bm{t}\le I_1+I_2+I_3,
	\end{align*}
	where the last sum corresponds to the decomposition of the last integral over $\{\|\bm{t}\|>A_n\}$ into two parts: the integral for $\{\|\bm{t}\|>c_n\}$ and $\{A_n<\|\bm{t}\|\le c_n\}$. Additionally, we split $\hat H_n(t)$ into two parts, i.e.:
	\begin{align}
		I_1&\overset{\Delta}{=}E\int_{\{\|\bm{t}\|>c_n\}}|D^{\beta-\alpha}\hat Q'_n(\bm{t}) D^\alpha\hat K_{\epsilon}(\bm{t})|d\bm{t},\\
		\label{1.1.4}I_2&\overset{\Delta}{=}\int_{\{c_n\ge\|\bm{t}\|>A_n\}}C\Big(1+\|\bm{t}\|^{|\beta-\alpha|}\Big)\exp\Big(-\frac{5}{24}\|\bm{t}\|^2\Big)d\bm{t},\\
		\label{1.1.5} I_3&\overset{\Delta}{=}\int_{\{\|\bm{t}\|>A_n\}}\Big|D^{\beta-\alpha}\sum_{r=0}^{s+k-2}n^{-r/2}\tilde P_{ra}(i\bm{t}:\{\bm\chi_{v,n}\})\exp\Big(\frac 12 \langle \bm{t}, D_n\bm{t}\rangle\Big)\Big|d\bm{t}.
	\end{align}
	Due to the presence of the exponential term, the right-hand side of equations \eqref{1.1.4} and \eqref{1.1.5} approaches zero exponentially fast as $n$ goes to infinity. In other words, we prove that: 
	\[I_2=o(n^{-(s-2)/2}),\qquad I_3=o(n^{-(s-2)/2}).\]
	Therefore, we only need to estimate the bound of $I_1$, where we will apply the general partial Cram\'{e}r's condition (GPCC).

	By applying Leibniz's rule for the differentiation of the product of $n$ functions, we obtain
	\begin{equation}\label{differentiation}
		|D^{\beta-\alpha}\hat Q_n''(\bm{t})|\le n^{|\beta-\alpha|}\mathbb{E}^*\bigg\|\frac{\tilde \Z_{1}}{n^{1/2}}\bigg\|^{|\beta-\alpha|}|g_n(\bm{t})|^{n-|\beta-\alpha|},
	\end{equation}
	where
	\[g_n(\bm{t})=\mathbb{E}^*\Big(\exp\big[i\langle n^{-1/2}\bm{t}, \tilde \Z_{1}\rangle\big]\Big).\]
	
	We are now left to verify that if the conditional characteristic functions $v_Z^*(\bm{t})$ of $\Z_{1}$ satisfy the GPCC, then the conditional characteristic function $g_n(\bm{t})$ of the truncated and centered vectors $\tilde \Z_{1}$ also satisfy the GPCC. The next lemma demonstrates this.
	\begin{lemma}\label{condition}
		For all integers $n\ge 1$ and for all $\bm{t}\in R^k$, we have:
		\[|g_n(\bm{t})|\le|v_\Z^*(\bm{t})|+\frac{2\rho_s}{n^{s/2}}.\]
		In particular, under the hypothesis of Theorem \ref{distance}, there  exists $\eta>0$ such that we have the local general partial Cram\'{e}r's condition (GPCC):
		\[\limsup_{n\to\infty}\mathbb{E}|v_\Z^*(\bm{t})|\le 1-\eta.\]
	\end{lemma}
	\begin{proof}
		Observe that
		\begin{align*}
			\mathbb{E}|g_n(\bm{t})|&=\mathbb{E}|\mathbb{E}^*(\exp(it\tilde \Z_{1}))|=\mathbb{E}\Big|\mathbb{E}^*\Big[\exp\big(it\Z_{1}I_{\{\|\Z_1\|\le\sqrt n\}}\big)\Big]\Big|\\
			&=\mathbb{E}\Big|\mathbb{E}^*\Big[\exp(it\Z_{1})I_{\{\|\Z_{1}\|\le\sqrt n\}}\Big]+\mathbb{E}^*\Big[I_{\{\|\Z_{1}\|>\sqrt n\}}\Big]\Big|\\
			&=\mathbb{E}\Big|\mathbb{E}^*\Big[\exp(it\Z_{1})\Big]-\mathbb{E}^*\Big[(\exp(it\Z_{1})-1)I_{\{\|\Z_{1}\|>\sqrt n\}}\Big]\Big|.
		\end{align*}
		By definition $v_\Z^*(\bm{t})=\mathbb{E}^*(\exp(i\bm{t}\Z_{1}))$, we then have
		\begin{align*}
			\mathbb{E}|g_n(\bm{t})|-\mathbb{E}|v_\Z^*(\bm{t})|&\le \mathbb{E}\Big|\mathbb{E}^*\Big[(\exp(it\Z_{1})-1)I_{\{\|\Z_{1}\|>\sqrt n\}}\Big]\Big|\\[2mm]
			&\le \mathbb{E}\Big [2\mathbb{E}^*\Big(I_{\{\|\Z_{1}\|>\sqrt n\}}\Big)\Big]\\[2mm]
			&\le\mathbb{E} \frac{2\mathbb{E}^*[\|\Z_{1}\|^s]}{n^{s/2}}.
		\end{align*}
		Therefore, we obtain 
		\[\limsup_{n\to\infty}\mathbb{E}|g_n(\bm{t})|\le \limsup_{n\to\infty}\Big[\mathbb{E}|v_Z^*(\bm{t})|+\mathbb{E}\frac{2\mathbb{E}^*[\|\Z_{i}\|^s]}{n^{s/2}}\Big]\le 1-\eta.\]
		Therefore, the conclusion holds.
	\end{proof}
	
	Let us return to the proof of Theorem \ref{distance} and continue to evaluate the integral $I_1$. Therefore, by Lemma \ref{condition}, we obtain:
	\[\sup_{\|\bm{t}\|>c_n}\mathbb{E}|g_n(\bm{t})|<\theta<1\]
	for all sufficiently large $n$. Here $\theta$ is a number independent of $n$. Hence, by the above equation \eqref{differentiation}, we obtain a specific estimate of $I_1$ that we aim to control under GPCC,
	\begin{align*}
		I_1&=\mathbb{E}\int_{\{\|\bm{t}\|>c_n\}}\|D^{\beta-\alpha}\hat Q''_n(t)D^{\alpha}\hat K_{\epsilon}(\bm{t})\|d\bm{t}\\
		&\le C \epsilon^{|\alpha|}n^{|\beta-\alpha|}\theta^{n-|\beta-\alpha|}\int_{\{\|t\|>n^{1/2}/16\rho_3\}}\exp\big(-(\epsilon\|\bm{t}\|)^{1/2}\big)d\bm{t}\\
		&\le C n^{|\beta-\alpha|}\theta^{n-|\beta-\alpha|}\epsilon^{|\alpha|-k}\int\exp\big(-\|\bm{t}\|^{1/2}\big)d\bm{t}\\
		&\le Cn^{s+k+1}\theta^n\epsilon^{-k}
	\end{align*}
	for all large $n$. Then, we can choose $\epsilon=e^{-dn}$ and $d$ is any positive number satisfying $d<-\frac{1}{k}\log\theta$, so that we can provide an upper bound for the integral term $I_1$, i.e.:
	\[I_1=o(n^{-(s-2)/2})\qquad (n\to\infty).\]
	
	Therefore, we have demonstrated:
	\begin{align}\label{1.1.12}
		\mathbb{E}\int_{\{\|\bm{t}\|>A_n\}}&|D^{\beta-\alpha}\hat H_n(\bm{t}) D^\alpha\hat K_{\epsilon}(\bm{t})|d\bm{t}\le o(n^{-(s-2)/2}) \qquad (n\to\infty).
	\end{align}
	
	The remainder of the proof is provided in the Section \ref{A.1}.
	
	\section{Proofs of main results}\label{appdx}
	
	\subsection{Proof of Theorem \ref{distance}}\label{A.1}
	Assume $P^*$ is a conditional probability given $C_n$. Let $\bm{\mu}_{s}$ be the $s$-th moment of $\textbf{Z}_j$, $\bm\rho_{s}$ be the $s$-th  absolute moment of $\Z_j $ and $\bm\chi_{s}$ be the $s$-th cumulant of $\Z_j(1\le j\le n)$. Namely, we write
	\[\bm\rho_{s}=\mathbb{E}\Vert \Z_j\Vert^s, \quad \bm\mu_{s}=\mathbb{E}\Z_{j}^{s},\quad (1\le j\le n).\]
	Define truncated random vectors
	\begin{equation*}
		\hat \Z_{j}=\left\{
		\begin{aligned}
			\Z_j &, &\|\Z_j\|\le n^{1/2}\\
			0 &, &\|\Z_j\|>n^{1/2},
		\end{aligned}
		\right.
		\quad \tilde \Z_{j}=\hat \Z_{j}-\mathbb{E}\hat \Z_{j}\quad (1\le j\le n).
	\end{equation*}
	And then let $\hat\bmmu_{s,j}$ be the $s$-th moment of $\hat \Z_j$, $\hat\bmrho_{s,j}$ be the $s$-th  absolute moment of $\hat \Z_j $ and $\hat\bmchi_{s,j}$ be the $s$-th cumulant of $\hat \Z_j (1\le j\le n)$. Besides, let $\tilde\bmmu_{s,j}$ be the $s$-th moment of $\tilde \Z_j$, $\tilde\bmrho_{s,j}$ be the $s$-th  absolute moment of $\tilde \Z_j $ and $\tilde\bmchi_{s,j}$ be the $s$-th cumulant of $\tilde \Z_j (1\le j\le n)$. Namely, write
	\[\hat\bmmu_{s,j}=\mathbb{E}\hat \Z_{j}^{s},\quad\hat\bmrho_{s,j}=\mathbb{E}\Vert \hat\Z_j\Vert^s,\quad \tilde\bmmu_{s,j}=\mathbb{E}\tilde \Z_{j}^{s},\quad\tilde\bmrho_{s,j}=\mathbb{E}\Vert \tilde\Z_j\Vert^s.\]
	Also introduce
	\[\Delta_{n,j,s}=\int_{\{\Vert \Z_j\Vert>n^{1/2}\}}\Vert \Z_j\Vert^s,\quad \bar\Delta_{n,s}=n^{-1}\sum_{j=1}^n\Delta_{n,j,s},\]
	\[\bar \Delta_{n,s}(\epsilon)=n^{-1}\sum_{j=1}^n\int_{\{\|\Z_j\|>\epsilon n^{1/2}\}}\|\Z_j\|^s\quad (\epsilon>0).\]
	Finally, let $B_n$ be the common covariance matrix of $\hat \Z_{1}$ and $\tilde \Z_{1}$. The symbol $*$ denotes convolution operation. Furthermore, we define the norm of a $k\times k$ matrix $T$, specifically,
	\[B_n=Cov(\hat \Z_{1})=Cov(\tilde \Z_{1}),\quad \|T\|=\sup_{x\in R^k,~\|x\|\le 1}\|Tx\|.\]
	
	Before giving the proof of Theorem \ref{distance}, let us state and prove three auxiliary lemmas. 
	
	\begin{lemma}\label{L1.1}
		Let $V=I$. Assume $Q_n^*$ is the conditional distribution of $n^{-1/2}(\Z_1+\dots+\Z_n)$ given $C_n$ and $Q_n'$ is the conditional distribution of $n^{-1/2}(\hat \Z_1+\dots+\hat \Z_n)$ given $C_n$. If $\bmrho_s<\infty$ for some $s>0$, then there exists a positive constant $c_1(s,k)$ such that
		\begin{equation}\label{1.1}
			\mathbb{E}\Vert Q_n^*-Q_n^{\prime}\Vert\le c_1(s,k)\bar\Delta_{n,s}n^{-(s-2)/2}.
		\end{equation}
		Also, there exist two positive constants $c_2(s,k), c_3(s,k)$ such that whenever
		\[\bar\Delta_{n,s}\Big(\frac 23\Big)\le c_2(s,k)n^{(s-2)/2}\]
		for some integer $s\ge 2$,
		\begin{equation}\label{1.2}
			\mathbb{E}\int \Vert x\Vert^r\vert Q_n^*-Q_n'\vert(dx)\le c_3(s,k)\bar\Delta_{n,s}n^{-(s-2)/2}
		\end{equation}
		for all $r\in (0,s]$.
	\end{lemma}
	\begin{proof}
		Let $G_j$ be the conditional distribution of $n^{-1/2}\Z_j$ given $C_n$ and $G_j'$ be the conditional distribution of $n^{-1/2}\hat \Z_j$ given $C_n$, $1\le j \le n$. Then
		\[Q_n^*=G_1*G_2*\dots* G_n, \qquad Q_n'=G_1'*G_2'*\dots * G_n',\]
		and
		\begin{align*}
			\mathbb{E}\|Q_n^*-Q_n'\|&=\mathbb{E}\Big\|\sum_{j=1}^n G_1*\dots*G_{j-1}*(G_j-G_j')*G_{j+1}'*\dots*G_n'\Big\|\\
			&\le \mathbb{E}\sum_{j=1}^n\|G_j-G_j'\|=2\mathbb{E}\sum_{j=1}^nP^*(\|\Z_j\|>n^{1/2})\\
			&\le 2\sum_{j=1}^nn^{-s/2}\mathbb{E}\int_{\{\|\Z_j\|>n^{1/2}\}}\|\Z_j\|^s\,dP^*=2\bar \Delta_{n,s}n^{-(s-2)/2}.
		\end{align*}
		Therefore, we complete the proof for \eqref{1.1}. Next, we shall prove the bound \eqref{1.2}. Assume that $s$ is an integer and $s\ge 2$. Since $\|\x\|^r\le 1+\|\x\|^s$ for $1\le r\le s$, it is enough to prove bound \eqref{1.2} for the case of $r=s$. 
		Therefore, we only need to give the proof for $r=s$. 
		
		First, by utilizing the definition of convolution and the properties of probability distributions, we establish an upper bound for $\mathbb{E}\int \|\x\|^s|Q_n-Q'_n|(dx)$ :
		\begin{align}\label{1.3}
			&\qquad\mathbb{E}\int \|\x\|^s|Q_n-Q'_n|(d\x)\\\nonumber
			&\le\sum_{j=1}^n \mathbb{E}\int\|\x\|^s\big|G_1*\dots*G_{j-1}*(G_j-G'_j)*G'_{j+1}*\dots*G'_n\big|(d\x)\\\nonumber
			&\le\sum_{j=1}^n \mathbb{E}\int\bigg(\int\|\bm{u}+\bm{v}\|^s|G_j-G'_j|(d\bm{v})\bigg)G_1*\dots*G_{j-1}*G'_{j+1}*\dots*G'_n(d\bm{u})\\\nonumber
			&\le2^{s-1}\sum_{j=1}^n \mathbb{E}\bigg(\|G_j-G'_j\|\int\|\bm{u}\|^sG_1*\dots*G_{j-1}*G'_{j+1}*\dots*G'_n(d\bm{u})\\\nonumber
			&\qquad\qquad\qquad\quad+\int\|\bm{v}\|^s|G_j-G'_j|(d\bm{v})\bigg).
		\end{align}
		
		Observe that the bound for the second half of inequality \eqref{1.3} can be obtained directly as
		\begin{equation}\label{1.4}
			\mathbb{E}\int\|\bm{v}\|^s|G_j-G'_j|(d\bm{v})=\mathbb{E}\int_{\{\|\Z_j\|>n^{1/2}\}}\|n^{-1/2}\Z_j\|^sdP^*=n^{-s/2}\Delta_{n,j,s}.
		\end{equation}
		Therefore, it suffices to estimate the bound of $E\int\|u\|^sG_1*\dots*G_{j-1}*G'_{j+1}*\dots*G'_n(du)$. By applying the double expectation theorem, we obtain
		\begin{align}\label{1.5}
			&\qquad\mathbb{E}\int\|\bm{u}\|^sG_1*\dots*G_{j-1}*G'_{j+1}*\dots*G'_n(d\bm{u})\\\nonumber
			&=\mathbb{E}\mathbb{E}^*\|n^{-1/2}(\Z_1+\dots+\Z_{j-1}+\hat \Z_{j+1}+\dots+\hat \Z_n)\|^s\\\nonumber
			&=\mathbb{E}\|n^{-1/2}(\Z_1+\dots+\Z_{j-1}+\hat \Z_{j+1}+\dots+\hat \Z_n)\|^s\\\nonumber
			&\le 2^{s-1}\bigg(\mathbb{E}\|n^{-1/2}(\Z_1+\dots+\Z_j+\hat \Z_{j+1}+\dots+\hat \Z_n)\|^s+\mathbb{E}\|n^{-1/2}\Z_j\|^s\bigg)\\\nonumber
			&\le 2^{2(s-1)}\bigg(\mathbb{E}\|n^{-1/2}(\Z_1+\dots+\Z_j+\tilde \Z_{j+1}+\dots+\tilde \Z_n)\|^s\\\nonumber
			&\qquad\qquad\qquad+\|n^{-1/2}(\mathbb{E}\hat \Z_{j+1}+\dots+\mathbb{E}\hat \Z_n)\|^s\bigg)+2^{s-1}\mathbb{E}\|n^{-1/2}\Z_j\|^s.
		\end{align}
		Thus, relying on Lemma \ref{lemma: 14.1} and the definition of $\hat \Z_j$, we arrive at the conclusion of the bound:
		\[\mathbb{E}\|n^{-1/2}\Z_j\|^s\le n^{-s/2}(n^{s/2}+\Delta_{n,j,s})\le1+n^{-(s-2)/2}\bar \Delta_{n,s},\]
		\begin{align*}
			\|n^{-1/2}(\mathbb{E}\hat \Z_{j+1}+\dots+\mathbb{E}\hat \Z_n)\|^s\le\Big(a^{1/2}n^{-(s-2)/2}\bar \Delta_{n,s}\Big)^s.
		\end{align*}
		Following the same methodology as presented in  \cite{bhattacharya2010normal}, we derive the bound:
		\[\mathbb{E}\|n^{-1/2}(\Z_1+\dots+\Z_j+\tilde \Z_{j+1}+\dots+\tilde \Z_n)\|^s\le c(s,k).\]
		Therefore, by utilizing the estimates in \eqref{1.3}, \eqref{1.4}, and \eqref{1.5}, we derive
		\[\mathbb{E}\int \Vert x\Vert^r\vert Q_n-Q_n'\vert(dx)\le c_3(s,k)\bar\Delta_{n,s}n^{-(s-2)/2}.\]
		
		Hence, the conclusion holds.
	\end{proof}

	\begin{lemma}\label{L1.2}
		Assume $\Z_1,\dots, \Z_n$ are n independent random vectors with values in $R^k$ having zero means. Define truncated random vectors
		\begin{equation*}
			\hat \Z_{j}=\left\{
			\begin{aligned}
				\Z_j &, &\|\Z_j\|\le n^{\frac 12},\\
				0 &, &\|\Z_j\|>n^{\frac12},
			\end{aligned}
			\right.
			\quad \tilde \Z_{j}=\hat \Z_{j}-\mathbb{E}\hat \Z_{j}\quad (1\le j\le n).
		\end{equation*}
		Then one has
		\[\mathbb{E}\|\tilde \Z_{1}\|^{s+k+1}=o(n^{(k+1)/2}).\]
	\end{lemma}
	\begin{proof}
		First, based on the definition of $\hat \Z_{1}$, we proceed to calculate: 
		\begin{align*}
			\mathbb{E}\|\hat \Z_{1}\|^{s+k+1}&=\int_{\{0\le\|\hat \Z_{1}\|\le n^{\frac 14}\}}\|\hat \Z_{1}\|^{s+k+1}+\int_{\{n^{\frac 14}\le\|\hat \Z_{1}\|\le n^{\frac 12}\}}\|\hat \Z_{1}\|^{s+k+1}\\[3mm]
			&\le n^{\frac{k+1}{4}}\int_{\{0\le\|\hat \Z_{1}\|\le n^{\frac 14}\}}\|\hat \Z_{1}\|^{s}+n^{\frac{k+1}{2}}\int_{\{n^{\frac 14}\le\|\hat \Z_{1}\|\le n^{\frac 12}\}}\|\hat \Z_{1}\|^s\\[3mm]
			&=o(n^{(k+1)/2}).
		\end{align*}
		Then, by Lemma \ref{lemma: 14.1}, we conclude that
		\[\mathbb{E}\|\tilde \Z_{1}\|^{s+k+1}\le 2^{s+k+1}\mathbb{E}\|\hat \Z_{1}\|^{s+k+1}\le o(n^{(k+1)/2}).\]
		Hence, the conclusion holds.
	\end{proof}
	
	\begin{lemma}\label{L1.3}
		Let $Q_n'$ be the conditional distribution of $n^{-1/2}(\hat \Z_{1}+\dots+\hat \Z_{n})$ given $C_n$, while $Q_n''$ represents the conditional distribution of $n^{-1/2}(\tilde \Z_{1}+\dots+\tilde \Z_{n})$ given $C_n$. Additionally, we define $a_n=n^{1/2}\mathbb{E}\hat \Z_{1}$. Recall that the translate $f_{y}(x)$ of $f(x)$ by $y\in R^k$ is defined by $f_y(x)=f(x+y)$, $x\in R^k$. Then we have
		\[\mathbb{E}\int f\, dQ_n'=\mathbb{E}\int f_{a_n}\,dQ_n''.\]
	\end{lemma}
	\begin{proof}
		According to the definition of $a_n$ and $f_y(x)$, we derive
		\begin{align*}
			\mathbb{E}\int f_{a_n}\,dQ_n''&=\mathbb{E}\int f(x+a_n)dQ_n''(x)=\mathbb{E}\int f(x+n^{1/2}\mathbb{E}\hat \Z_{1})dQ_n''(x)\\
			&=\mathbb{E}\int f(x)dQ_n''(x-n^{1/2}\mathbb{E}\hat \Z_{1}).
		\end{align*}
		And from the definition of $Q'_n$ and $Q''_n$, we observe
		\[\mathbb{E}\int f(x)dQ_n''(x-n^{1/2}\mathbb{E}\hat \Z_{1})=\mathbb{E}\int f(x)\, dQ_n'(x).\]
		Therefore, we obtain
		\[\mathbb{E}\int f\, dQ_n'=\mathbb{E}\int f_{a_n}\,dQ_n''.\]
		Then we complete the proof of Lemma \ref{L1.3}.
	\end{proof}

	We are now in a position to prove Theorem \ref{distance}. Assume that $V=I$, without loss of generality. Let $Q_n'$ be the conditional distribution of $n^{-1/2}(\hat \Z_{1}+\dots+\hat \Z_{n})$ given $C_n$, and $Q_n''$ be the conditional distribution of $n^{-1/2}(\tilde \Z_{1}+\dots+\tilde \Z_{n})$ given $C_n$. By Lemma \ref{L1.1}, for sufficiently large $n$, we have
	\begin{align*}
		\bigg\vert \mathbb{E}\int f \,d (Q^*_n-Q_n')\bigg\vert&\le M_{s'}(f)\mathbb{E}\int (1+\Vert x\Vert^{s'})\vert Q^*_n-Q_n'\vert\,(dx)\\
		&\le CM_{s'}(f)n^{-(s-2)/2}\bar\Delta_{n,s},
	\end{align*}
	where
	\[\bar\Delta_{n,s}=\int_{\{\Vert \Z_1\Vert>n^{1/2}\}}\Vert \Z_1\Vert^s=o(1) \qquad (n\to\infty).\]
	
	By writing $a_n=n^{1/2}\mathbb{E}\hat \Z_{1}$ and applying Lemma \ref{lemma: 14.1} along with the definition of $\bar \Delta_{n,s}$, we determine its bound:
	\[\Vert a_n\Vert= n^{1/2}\Vert \mathbb{E}\hat \Z_{1}\Vert\le k^{1/2}\bar\Delta_{n,s}n^{-(s-2)/2}=o(n^{-(s-2)/2}) \quad (n\to\infty).\]
	Furthermore, by using the definitions of $Q_n$ and $f_{a_n}$, we prove that: 
	\begin{equation}\label{truncation}
		\mathbb{E}\int f\, dQ_n'=\mathbb{E}\int f_{a_n}\,dQ_n''.
	\end{equation}
	Further details about  equation \eqref{truncation} are provided in Lemma \ref{L1.3}. According to the final inequality in Lemma \ref{lemma: 14.6}, we obtain that: 
	\begin{align*}
		&\quad\bigg| \int(f_{a_n}-f)\, d\Big(\sum_{r=0}^{s-2}n^{-r/2}P_r(-\Phi:\{\bmchi_\bfv\})\Big)\bigg|\\
		&=\bigg|\int f(x)\sum_{r=0}^{s-2}n^{-r/2}\Big(P_r(-\phi:\{\bmchi_\bfv\})(x-a_n)-P_r(-\phi:\{\bmchi_\bfv\})(x)\Big)\,dx\bigg|\\
		&\le CM_{s'}(f)n^{-(s-2)/2}\bar\Delta_{n,s},
	\end{align*}
	where $P_r(-\Phi_{o,V}:\{\bm\chi_v\})$ is the finite signed measure on $R^k$. Next, using the first inequality in Lemma \ref{lemma: 14.6}, we obtain:
	\begin{align*}
		&\quad\bigg|\int f_{a_n}\, d\Big(\sum_{r=0}^{s-2}n^{-r/2}P_r(-\Phi:\{\bmchi_\bfv\})-\sum_{r=0}^{s-2}n^{-r/2}P_r(-\Phi_{0, B_n}:\{\tilde\bmchi_{\bfv}\})\Big)\bigg|\\
		&\le CM_{s'}(f)n^{-(s-2)/2}\bar\Delta_{n,s},
	\end{align*}
	where $B_n=Cov(\tilde \Z_{1})$ and $\tilde\bmchi_{\bfv}$ denotes the $v$-th cumulant of $\tilde \Z_{1}$.
	
	Through the above analysis, calculations, and Lemma \ref{P_ra}, we found that  to establish the final conclusion, it is sufficient to derive an upper bound for the following formula:
	\[\mathbb{E} \int  f_{a_n}\,d\Big(Q_n''-\sum_{r=0}^{s-2}n^{-r/2}P_{ra}(-\Phi_{0, B_n}:\{\tilde\bmchi_{\bfv}\})\Big).\]
	
	$\mathbf{Step~1.}$~ According to the result of Lemma \ref{lemma:11.6}, we derive the following bound:
	\begin{align}\label{1.1.1}
		\mathbb{E}\int[1+(\|\x\|+\epsilon+\|a_n\|)^{s'}]|H_n*K_{\epsilon}|\,d\x\le C\max_{0\le|\beta|\le a+s+1}\mathbb{E}\int|D^\beta(\hat H_n\hat K_{\epsilon})(\bm{t})|d\bm{t},
	\end{align}
	where $\hat H_n$ is the Fourier-Stieltjes transform of $H_n$. Additionally, according to Leibniz's rule for differentiation, if $\alpha\in N^a$ and $\beta\in N^a$, we rewrite: 
	\[D^\beta(\hat H_n\hat K_{\epsilon})=\sum_{0\le\alpha\le\beta}C(D^{\beta-\alpha}\hat H_n)(D^\alpha\hat K_{\epsilon}).\]
	Then, by applying Lemma \ref{lemma: 9.10}, we obtain:
	\begin{align}\label{1.1.2}
		\mathbb{E}\int_{\{\|\bm{t}\|\le A_n\}}|D^{\beta-\alpha}\hat H_n(\bm{t}) D^\alpha\hat K_{\epsilon}(\bm{t})|dt&\le \mathbb{E}\int_{\{\|\bm{t}\|\le A_n\}}C*|D^{\beta-\alpha}\hat H_n(\bm{t)}|d\bm{t}\\\nonumber
		&\le \mathbb{E}\big( Cn^{-(s+a-1)/2}\eta_{s+a+1}\big),
	\end{align}
	where
	\[\eta_{s+k+1}=\int\|T_n \x\|^{s+a+1}Q'_n(d\x)=\mathbb{E}^*\|T_n\tilde \Z_{1}\|^{s+a+1},\]
	and
	\[A_n=\frac{Cn^{1/2}}{\big(\mathbb{E}^*\|T_n\tilde \Z_{1}\|^{s+a+1}\big)^{1/(s+a-1)}}.\]
	Here $T_n$ is the symmetric and positive-definite matrix satisfying $T_n^2=B_n^{-1}$ for all $n\ge n_0$. Recall that the matrix $B_n$ is defined as $B_n=n^{-1}\sum_{j=1}^nCov(\tilde \Z_{j})$.
	
	According to Corollary 14.2 of \cite{bhattacharya2010normal}, there exists an integer $n_0$ such that $(\|T_n\|: n\ge n_0)$ is bounded. Additionally, according to Lemma \ref{lemma: 14.1}, the term $\mathbb{E}\mathbb{E}^*\|\tilde \Z_{1}\|^{s+a+1}$ is bounded by:
	\begin{equation}\label{1.1.3}
		\mathbb{E}\mathbb{E}^*\|\tilde \Z_{1}\|^{s+a+1}\le \mathbb{E}(2^{s+a+1}\mathbb{E}^*\|\hat \Z_{1}\|^{s+a+1})=o(n^{(a+1)/2}),\quad (n\to\infty).
	\end{equation}
	Therefore, from the above estimate \eqref{1.1.3}, we find that the right side of equation \eqref{1.1.2} is bounded by:
	\[\mathbb{E}\big( Cn^{-(s+a-1)/2}\eta_{a+s+1}\big)\le Cn^{-(s+a-1)/2}\mathbb{E}\mathbb{E}^*\|\tilde \Z_{1}\|^{s+a+1}=o(n^{-s/2}).\]
	
	Therefore, we can obtain the bound of \eqref{1.1.1} on the set $\{\|\bm{t}\|\le A_n\}$. Specifically,
	\begin{equation}\label{bound10}
		\mathbb{E}\int_{\{\|\bm{t}\|\le A_n\}}|D^{\beta-\alpha}\hat H_n(\bm{t}) D^\alpha\hat K_{\epsilon}(\bm{t})|d\bm{t}=o(n^{-(s-2)/2}),\qquad (n\to\infty).\end{equation}
	
	Combining equations \eqref{1.1.12} and \eqref{bound10}, we derive the following:
	\begin{align}\label{A.11}
		\Big|\mathbb{E}\int f_{a_n}\,dH_n\Big|\le ~& \mathbb{E}~\bar\omega_{f_{a_n}}\bigg(2e^{-dn}:\bigg|\sum_{r=0}^{s+k-2}n^{-r/2}P_{ra}(-\Phi_{0,D_n}:\{\tilde\bmchi_{\bfv}\})\bigg|\bigg)\\\nonumber
		&\quad+M_{s'}(f)o(n^{-(s-2)/2}) \qquad (n\to\infty).
	\end{align}

	$\mathbf{Step~2.}$~~Now, we only need to calculate the more specific form of \[\bar\omega_{f_{a_n}}(2e^{-dn}:|\sum_{r=0}^{s+k-2}n^{-r/2}P_{ra}(-\Phi_{0,B_n}:\{\tilde\bmchi_{\bfv}\})|).\] The details are as follows:
	\begin{align*}
		\mathbb{E}~\bar\omega_{f_{a_n}}\bigg(2e^{-dn}:\Big|&\sum_{r=0}^{s+k-2}n^{-r/2}P_{ra}(-\Phi_{0,B_n}:\{\tilde\bmchi_{\bfv}\})\Big|\bigg)\\
		&\le \mathbb{E}\sum_{r=0}^{s+k-2}n^{-r/2}\bar\omega_{f_{a_n}}\bigg(2e^{-dn}:\Big|P_{ra}(-\Phi_{0,B_n}:\{\tilde\bmchi_{\bfv}\})\Big|\bigg).
	\end{align*}
	
	Next, we split it into two cases and calculate each separately. Note that: 
	\begin{align}\label{1.1.13}
		P_{ra}(-\phi:\{\bm\chi_\bfv\})=\sum_{m=1}^r\frac{1}{m!}\bigg[~~\sideset{}{^*}\sum_{j_1,\dots,j_m}\bigg(~~~\sideset{}{^{**}}\sum_{j_1,\dots,j_m}\frac{\chi^*_{v_1}\dots\chi^*_{v_m}}{v_1!\dots v_m!}(-1)^{r+2m}D^{v_1+\dots+v_m}\phi\bigg)\bigg],
	\end{align}
	where, $\Sigma^*$ denotes summation over all $m$-tuples of positive integers $(j_1,\dots, j_m)$ satisfying $j_1+\dots+j_m=r$, and $\Sigma^{**}$ denotes summation over all $m$-tuples of nonnegative integral  vectors $(v_1,\dots, v_m)$ satisfying $|v_i|=j_i+2$ for fixed $(j_1,\dots, j_m)$. Furthermore, $\chi^*_{v_1}\dots\chi^*_{v_m}$ denote the conditional cumulants. Additionally, defining $\rho^*_{j_1+2}\dots\rho^*_{j_m+2}$ as the conditional moments, from Lemma \ref{lemma: 7.2} and Lemma \ref{lemma: 6.3}, we obtain:
	\begin{align}\label{1.1.14}
		\mathbb{E}|\chi^*_{v_1}\dots\chi^*_{v_m}|&\le  C\mathbb{E}(\rho^*_{j_1+2}\dots\rho^*_{j_m+2})=C\bigg(\frac{\rho_{j_1+2}}{\rho_2^{(j_1+2)/2}}\bigg)\dots\bigg(\frac{\rho_{j_m+2}}{\rho_2^{(j_m+2)/2}}\bigg)\rho_2^{(r/2+m)}\\\nonumber
		&\le C\bigg(\frac{\rho_{r+2}}{\rho_2^{(r+2)/2}}\bigg)^{j_1/r}\dots\bigg(\frac{\rho_{r+2}}{\rho_2^{(r+2)/2}}\bigg)^{j_m/r}\rho_2^{(r/2+m)}\\\nonumber
		&=C\rho_2^{(r/2+m)}\bigg(\frac{\rho_{r+2}}{\rho_2^{(r+2)/2}}\bigg)=C\rho_2^{m-1}\rho_{r+2}.
	\end{align}
	Through some calculations, we obtain:
	\begin{align}\label{1.1.15}
		|D^{v_1+\dots+v_m}\phi|\le C(1+\|\bm{t}\|^{|v_1+\dots+v_m|})\phi,
	\end{align}
	where
	\[|v_1+\dots+v_m|\le(j_1+2)+\dots+(j_m+2)\le 3r.\]
	
	Therefore, for $0\le r \le s-2$, by combining equations \eqref{1.1.13}, \eqref{1.1.14} and \eqref{1.1.15}, we obtain: 
	\begin{align*}
		\mathbb{E}\bar\omega_{f_{a_n}}&\bigg(2e^{-dn}: n^{-r/2}\Big|P_{ra}(-\Phi_{0,B_n}:\{\tilde\bmchi_{\bfv}\})\Big|\bigg)\\
		&\le n^{-r/2}\int C\omega_f(\x:2e^{-dn})E\|\tilde \X_{1}\|^{r+2}(1+\|\x\|^{3r})\phi_{a_n,B_n}(\x)d\x\\
		&\le C\rho_s\bigg[\int_{\{\|\x\|\le n^{1/6}\}}\omega_f(\x:2e^{-dn})|\phi_{a_n,B_n}(\x)-\phi(\x)|d\x\\
		&\qquad\qquad+\int_{\{\|\x\|\le n^{1/6}\}}\omega_f(\x:2e^{-dn})\phi(\x)d\x\bigg]\\
		&\quad~+Cn^{-r/2}\rho_s\int_{\{\|\x\|>n^{1/6}\}}\omega_f(\x:2e^{-dn})(1+\|\x\|^{3r})\phi_{a_n,B_n}(\x)d\x\\
		&\le C\rho_s\bigg[ M_{s'}(f)\int_{\{\|\x\|\le n^{1/6}\}}(1+\|\x\|^{s'})|\phi_{a_n,B_n}(\x)-\phi(\x)|d\x+\bar\omega_f(2e^{-dn}:\Phi)\bigg]\\
		&\quad~ +Cn^{-r/2}\rho_s\int_{\{\|\x\|>n^{1/6}\}}(1+\|\x\|^{3r+s'})\phi_{a_n,B_n}(\x)d\x\\
		&\le M_{s'}(f)o(n^{-(s-2)/2})+C\rho_s\bar\omega_f(2e^{-dn}:\Phi).
	\end{align*}
	Additionally, these inequalities rely on the definition of $P_{ra}(-\Phi_{0,B_n}:\{\tilde\bmchi_{\bfv}\})$, Lemma \ref{lemma: 14.6}, Lemma \ref{lemma: 7.2} and:
	\[\omega_f(\x:\epsilon)\le2M_{s'}(f)(1+(\|\x\|+\epsilon)^{s}).\] 
	
	On the other hand, when $s-1\le r\le s+k-2$:
	\begin{align}\label{1.1.16}
		\mathbb{E}\bar\omega_{f_{a_n}}\bigg(2e&^{-dn}: n^{-r/2}\Big|P_{ra}(-\Phi_{0,B_n}:\{\tilde\bmchi_{\bfv}\})\Big|\bigg)\\\nonumber
		&\le C n^{-r/2} \mathbb{E}\|\tilde \Z_{1}\|^{r+2}M_{s'}(f)\int(1+\|\x\|^{3r+s'})\phi_{a_n,B_n}(\x)d\x\\\nonumber
		&=M_{s'}(f)o(n^{-(s-2)/2})\qquad (n\to\infty),
	\end{align}
	These inequalities are based on Lemma \ref{L1.2}. In other words, by using Lemma \ref{L1.2} we can obtain:
	\[n^{-r/2}\mathbb{E}\|\tilde \Z_{1}\|^{r+2}=o(n^{-(s-2)/2}).\]
	
	Therefore, it can be demonstrated that:
	\begin{equation}\label{1.1.17}
		\Big|\mathbb{E}\int_{\{\|\bm{t}\|>A_n\}} f_{a_n}\,dH_n\Big|\le o(n^{-(s-2)/2})M_{s'}(f)+C\bar\omega_f(2e^{-dn}:\Phi) \quad (n\to\infty).
	\end{equation}
	Noting that
	\begin{align*}
		\bigg|\int f_{a_n}d \bigg(\sum_{r=s-1}^{s+k-2}n^{-r/2}P_{ra}(-\Phi_{0,D_n}:\{\bm\chi_{v,n}\})\bigg)\bigg|=M_{s'}(f)o(n^{-(s-2)/2}).
	\end{align*}
	and combining equations \eqref{A.11}, \eqref{1.1.16} and \eqref{1.1.17}, we arrive at the conclusion that:
	\begin{align}\label{1.1.18}
		\bigg|\mathbb{E}\int_{\{\|\bm{t}\|>A_n\}} f_{a_n}\,d(Q''_n-\sum_{r=0}^{s-2}n^{-r/2}&P_{ra}(-\Phi_{0,B_n}))\bigg|\\\nonumber
		&\le M_{s'}(f)o(n^{-(s-2)/2})+C\bar\omega_f(2e^{-dn}:\Phi).
	\end{align}
	
	Therefore, we obtain the result.

	\subsection{Proof of Theorem \ref{main1}}\label{s2}
	We now give the proof of Theorem \ref{main1} stated in Section \ref{sec3}, which asserts that there is a valid Edgeworth expansion for the function of sample means of vector variables under the GPCC. When $a=1$, the GPCC degenerates into the partial Cram\'{e}r's condition. The proof of Theorem \ref{main1} in this case is obtained by \cite{bai1991edgeworth}. When $a=k$, the GPCC degenerates into the Cram\'{e}r's condition. The proof of Theorem \ref{main1} in this case is obtained by \cite{bhattacharya1978validity}. 
	
	Therefore, we only need to prove the case when $1<a<k$. Define the functions:
	\begin{align*}
		&h_n(\textbf{z})=n^{1/2}[H(\bm\mu+n^{-1/2}\textbf{z})-H(\bm\mu)],\qquad \textbf{z}=(z^{(1)},\dots,z^{(k)})\in R^k,\\
		&f_{s-1}(\textbf{z})=\sum l_i z^{(i)}+\frac {1}{2} n^{-1/2}\sum l_{i,j}z^{(i)}z^{(j)}+\frac {1}{3!} n^{-1}\sum l_{i_1,i_2,i_3}z^{(i_1)}z^{(i_2)}z^{(i_3)}+\dots\\
		&\qquad\qquad+\frac{1}{(s-1)!}n^{-(s-2)/2}\sum l_{i_1,\dots,i_{s-1}}z^{(i_1)}\cdots z^{(i_{s-1})} .
	\end{align*}
	Then we can rewrite $W_n$  and $W_n'$ as
	\[W_n=h_n(n^{1/2}(\bar \Z-\bm\mu)),\quad  W_n'=f_{s-1}(n^{ 1/2}(\bar \Z-\bm\mu)).\]
	
	Let $D_i$ denote differentiation with respect to the $i$th coordinate. Write $D=(D_1,\dots, D_k)$. Then $p_r(-D)$ is a differential operator. Write 
	\[\phi_V(\textbf{u})=(2\pi)^{-k/2}(det V)^{-1/2}\exp\Big(-\frac 12 \langle \textbf{u}, V^{-1}\textbf{u}\rangle\Big),\]
	\[\xi_{s,n}(\textbf{u})=\Big[1+\sum_{r=1}^{s-2}n^{-r/2}p_r(-D)\Big]\phi_V(\textbf{u}),\quad \textbf{u}\in R^k.\]
	
	Let $Q_n$ denote the distribution of $n^{1/2}(\bar \Z-\bm\mu)$ and $Q_n^*$ denote that given $E_n$. $\Phi_V$ is the $k$-variate normal distribution with mean zero and covariance matrix $V$.  Let a class  $ \mathscr{B}$ of  Borel sets satisfy
		\begin{equation}\label{special}
			\sup_{B\in\mathscr{B}}\int_{(\partial B)^\epsilon}\phi_{\sigma^2}(x)dx=O(\epsilon)\quad (\epsilon\to 0).
		\end{equation}
		For any $B\in\mathscr{B}$,
		define by
		\[A=\{\textbf{u}\in R^k: h_n(\textbf{u})\in B\}.\]
	For the continuity of $h_n(\textbf{u})$, we can obtain
	\[\partial A\subset\{\textbf{u}\in R^k: h_n(\textbf{u})\in \partial B\}.\] 
	Now, assume that $\textbf{u}\in (\partial A)^\epsilon$. Then, there exists a $\bf{u}'$ such that $h_n(\textbf{u}')\in\partial B $ and $|\textbf{u}'-\textbf{u}|<\epsilon$. Let $M_n=\{|\textbf{u}|<((s-1)\Lambda \log n)^{1/2}\}$, where $\Lambda$ is the largest eigenvalue of $V$. Given this, if $\textbf{u}\in M_n$, then $|h_n(\textbf{u}')-h_n(\textbf{u})|\le d'\epsilon$, where $d'$ is an upper bound of $|\rm{grad}$\,$h_n|$ on $M_n^\epsilon$. Here $\partial B$ is the boundary of $B$, and $ M_n^\epsilon$ is the set of all points within a distance $\epsilon$ from $ M_n$. Since the $\Phi_V$-probability of the complement of $M_n$ is $o(n^{-(s-2)/2})$, we derive
	\begin{equation}\label{2.1}
		\Phi_{V}((\partial A)^{\epsilon})\le\Phi_{V}(\{ h_n(\textbf{u})\in (\partial B)^{d'\epsilon}\})+o(n^{-(s-2)/2}).
	\end{equation}
	And according to Lemma \ref{lemma: 2.1}, we get
	\begin{align}\label{2.2}
		\Phi_{V}(\{ h_n(\textbf{u})\in (\partial B)^{d'\epsilon}\})&=\int_{\{h_n(\textbf{u})\in (\partial B)^{d'\epsilon}\}}\xi_{s,n}(\textbf{u})d\textbf{u}+o(n^{-(s-2)/2})\\\nonumber
		&=\int_{ (\partial B)^{d'\epsilon}}\phi_{\sigma^2}(v)dv+o(n^{-(s-2)/2})\\\nonumber
		&=O(\epsilon)+o(n^{-(s-2)/2}).
	\end{align}
	Therefore, combining the equation \eqref{2.1}, and \eqref{2.2} and from Corollary \ref{co1.1} on asymptotic expansion under GPCC, we obtain
	\begin{equation}\label{2.3}
		\sup_{A\in\mathscr{B}^k}\Big\vert \mathbb{E} Q^*_n(A)-\sum_{r=0}^{s-2}n^{-r/2}P_r(-\Phi_{V}:\{\chi_\bfv\})(A)\Big\vert=o(n^{-(s-2)/2}),
	\end{equation}
	where $A$ satisfies the boundary condition \eqref{boundary}. According to the relationship of signed measure $P_r$ and $p_r$, we get
	\begin{align*}
		\sum_{r=0}^{s-2}n^{-r/2}P_r(-\Phi_{V}:\{\bm\chi_\bfv\})(A)&=\sum_{r=0}^{s-2}n^{-r/2}p_r(-D:\{\bm\chi_\bfv\})\Phi_{V}(A)\\
		&=\int_{A}\Big(1+\sum_{r=1}^{s-2}n^{-r/2}p_r(-D)\Big )\phi_V(\textbf{u})\,d \textbf{u}\\
		&=\int_{A}\xi_{s,n}(\textbf{u})\,d \textbf{u}.
	\end{align*}
	Therefore, we can rewrite equation \eqref{2.3} as
	\begin{equation}\label{2.4}
		\sup_{A\in\mathscr{B}^k}\Big\vert \mathbb{E} Q^*_n(A)-\int_{A}\xi_{s,n}(\textbf{u})\,d \textbf{u}\Big\vert=o(n^{-(s-2)/2}),
	\end{equation}
	And then from the Lemma \ref{lemma: 2.1}, we can translate the integral of the multivariate Edgeworth expansion over the region to the integral of the univariate one. Namely, we calculate
	\begin{equation}\label{2.5}
		\int_{\{h_n(\textbf{u})\in B\}}\xi_{s,n}(\textbf{u}) \,d \textbf{u}=\int_B d F_n(u)+o(n^{-(s-2)/2}).
	\end{equation}
	where 
	\[F_n(u)=\int_{-\infty}^u\Big[1+\sum_{r=1}^{s-2}n^{-r/2}q_r(v)\Big]\phi_{\sigma^2}(v)\,dv.\]
	And $q_r$ is polynomials whose coefficients do not depend on $n$.
	
	Therefore, by combining equations \eqref{2.4} and \eqref{2.5}, we derive that for all Borel set $B$:
	\begin{equation}
		\sup_{B\in\mathscr{B}}\Big\vert \mathbb{E}Q^*_n(A)-\int_B \,d F_n(u)\Big\vert=o(n^{-(s-2)/2}).
	\end{equation}
	
	Next, by utilizing the definition of conditional expectation and our statistic $W_n$, we  obtain:
	\[\mathbb{E}Q^*_n(A)=\mathbb{EE}^*[I_{\{h_n(\textbf{u})\in B\}}]=P(W_n\in B).\]
	Therefore, by considering $B$ as the specific Borel set $(-\infty, x)$, which satisfies \eqref{special}, we obtain:
	\begin{equation}\label{2.6}
		\sup_{x}\Big\vert \mathcal{Q}_n(x)-\int_B \,d F_n(u)\Big\vert=o(n^{-(s-2)/2}),
	\end{equation}
	where
	\[F_n(u)=\int_{-\infty}^u\Big[1+\sum_{r=1}^{s-2}n^{-r/2}q_r(v)\Big]\phi_{\sigma^2}(v)\,dv.\]
	
	Next, we shall identify $F_n$ and $\Psi$. We will show this in two cases. In the first case, we assume that $\Z_1$ is bounded. Assume that the distribution function of $W_n^{\prime}$ is $P(x)$. On the one hand, note that $f_{s-1}$ is a Taylor expansion of $h_n$ and $W_n'$ is a polynomial in $n^{1/2}(\bar \Z-\bmmu)$. Therefore, the moment of $W_n'$ can be approximated by the moments of $W_n$. And according to the equation \eqref{2.6}, we write
	\[\mathbb{E}W_n^{\prime j}=\int_{R^k}f_{s-1}^j\xi_{s,n}(\z)\,d\z+o(n^{-(s-2)/2}).\]
	And then according to Lemma \ref{lemma: 2.1}, we get
	\begin{equation}\label{2.7}
		\mathbb{E}W_n^{\prime j}=\int_{-\infty}^{\infty}u^jdF_n(u)+o(n^{-(s-2)/2}).
	\end{equation}
	On the other hand, evidence from \cite{bhattacharya1978validity} suggests that
	\[\kappa_{j,n}=O(n^{-(j-2)/2}), ~~j\ge 3;\quad \tilde\kappa_{j,n}=\kappa_{j,n}+o(n^{-(s-2)/2}), ~~j\ge 1; \quad \tilde\kappa_{j,n}=0, ~~j>s.\]
	That is, the difference between $\tilde\kappa_{j,n}$ and $\kappa_{j,n}$ is $o(n^{-(s-2)/2})$. Therefore, based on the approximation of the characteristic function of $W_n^{\prime}$ i.e. the equation \eqref{expansion}, it can be deduced that 
	\begin{equation}\label{2.8}
		\sup_{|t|\le 1}\Big|\hat\psi_{s,n}(t)-\mathbb{E}(\exp(itW_n^{\prime}))\Big|=o(n^{-(s-2)/2}).
	\end{equation}
	Next, from the equation \eqref{2.8} and derivatives of $\hat\psi_{s,n}$ at zero differ from those of $\mathbb{E}(\exp(itW_n^{\prime}))$ by $o(n^{-(s-2)/2})$, we find 
	\begin{equation}\label{2.9}
		\mathbb{E}W_n^{\prime j}=\int_{-\infty}^{\infty}u^jd\Psi_{s,n}(u)+o(n^{-(s-2)/2}).
	\end{equation}
	
	Hence, by applying equations \eqref{2.7} and \eqref{2.9}, and noting that neither $F_n$ nor $\Psi_{s,n}$ include terms of $o(n^{-(s-2)/2})$, we can conclude that:
	\begin{equation}\label{2.10}
		\int_{-\infty}^{\infty}u^j\,dF_n(u)=\int_{-\infty}^{\infty}u^j\,d\Psi_{s,n}(u).
	\end{equation}
	From equation \eqref{2.10}, we observe that the values and derivatives of all orders of the Fourier-Stieltjes transforms of $F_n$ and $\Psi_{s,n}$ coincide at the origin. Hence, $F_n$ and $\Psi_{s,n}$ have the same distribution. In other words, $F_n=\Psi_{s,n}$.
	
	In the other situation, when $\Z_1$ is in the general case. we define a new random vector $\Z_{1,c}$ as follows:
	\begin{equation*}
		\Z_{1,c}=\left\{
		\begin{aligned}
			\Z_1 &, &|\Z_1|\le c,\\
			0 &, &|\Z_1|>c.
		\end{aligned}
		\right.
	\end{equation*}
	Additionally,  we can choose $c$ to be sufficiently large such that the characteristic function of $\Z_{1c}$ satisfies GPCC. Specifically, the expectation of the bound of the conditional characteristic function, given $E_n$, is bounded away from one at infinity. Furthermore, we define the coefficient polynomials of $n^{-r/2}$ in $\psi_{s,n}$ as $\bar q_r$, 
	\[\bar q_r(v)=\Big[\pi_r\Big(-\frac{d}{dv}\Big)\phi_{\sigma^2}(v)\Big]/\phi_{\sigma^2}(v).\]
	Let $\bm\gamma_{s}$ be the vector of all cumulants of $\Z_{1}$ of order $s$ and less, and let $\bm\gamma_{s,c}$ be the vector of all cumulants of $\Z_{1,c}$ of order $s$ and less. Since $\Z_{1,c}$ is a bounded random vector, from our previous results, we obtain $q_r(\bm\gamma_{s,c})=\bar q_r(\bm\gamma_{s,c})$. Because of 
	\[\bm\gamma_{s,c}\to\bm\gamma_s,\quad (c\to\infty)\]
	and the continuity of $\bar q_r$ and $q_r$, we can conclude that
	\[q_r(\bm\gamma_{s})=\bar q_r(\bm\gamma_{s}).\] 
	Thus, the proof of Theorem \ref{main1} is complete.

	\subsection{Proof of Theorem \ref{main3}}\label{s4}

	We now provide the proof of Theorem \ref{main3} as stated in Section 4 which asserts that there is a valid Edgeworth expansion for the function of sample means of special vector variables under GPCC \eqref{Order 1}. 
	Let
	\begin{equation*}
		\left\{
		\begin{aligned}
			V_1&=w_1+\dots+w_k,\\
			V_2&=K_1(w_1)+\dots+K_1(w_k),\\
			&~~\vdots\\
			V_k&=K_{k-1}(w_1)+\dots+K_{k-1}(w_k),
		\end{aligned}
		\right. \quad J=\left | \begin{matrix}
			\frac{\partial V_1}{\partial w_1}  \quad&\frac{\partial V_1}{\partial w_2} \quad  & \dots & \quad\frac{\partial V_1}{\partial w_k}\\[3mm]
			\frac{\partial V_2}{\partial w_1} \quad&\frac{\partial V_2}{\partial w_2}\quad & \dots &\quad\frac{\partial V_2}{\partial w_k} \\[3mm]
			\vdots \quad& \vdots \quad& \cdots & \quad\vdots \\[3mm]
			\frac{\partial V_k}{\partial w_1}\quad &\frac{\partial V_k}{\partial w_2}\quad & \dots &\quad\frac{\partial V_k}{\partial w_k} \\
		\end{matrix} \right |.
	\end{equation*}
	
	Let $\mathcal{P}$ be the distribution function. According to the Lebesgue decomposition theorem, the distribution function of $Z_{j1}$ can be uniquely decomposed into three components:
	\[P(Z_{j1}\le x)=c_1 F_{j1}(x)+c_2 F_{j2}(x)+c_3 F_{j3}(x),\]
	where $c_1>0$, $c_k\ge 0$ for $k=2,3$, and $c_1+c_2+c_3=1$. Here, $F_{j1}(x), F_{j2}(x)$, and $F_{j3}(x)$ are the absolutely continuous, discrete, and singular distribution functions, respectively.
	
	Hence, the distribution of $Z_{j1}+Z_{j2}+\dots +Z_{jn}$ has an absolutely continuous component:
	\[c_1^k F_{11}*F_{21}*\dots*F_{n1}.\]
	
	Next, we shall establish the existence of the density function of $(V_1, V_2, \dots, V_k)$ in the absolutely continuous component using the variable transformation method. Since the Jacobi determinant $J$ is not equal to 0, we can obtain:
	\[\mathcal{P}(V_1, V_2, \dots, V_k)=\prod_{i=1}^k\mathcal{P}_{w_i} |J|.\]
	Thus, we establish that $(V_1, V_2, \dots, V_k)$ has an absolutely continuous component. Therefore, the conditional distribution of $V_1$ has an absolutely continuous component given $V_2, \dots, V_k$. 
	Specificallly, the conditional distribution function of $V_1$ given $V_2, \dots, V_k$ can be written as: 
	\[F=\delta G+(1-\delta)H,\]
	where $\delta>0$ and $G$ is absolutely continuous with density $g$. Then 
	\[v_1(t)\le \delta\bigg|\int_{-\infty}^{\infty}\exp\Big(i\sum_{j=1}^kt_jx_j\Big)g(x)dx\bigg|+1-\delta,\]
	and so it is suffices to prove that
	\begin{equation}\label{R-L}
		\lim_{\|t\|\to\infty}\bigg|\int_{-\infty}^{\infty}\exp\Big(i\sum_{j=1}^kt_jx_j\Big)g(x)dx\bigg|=0.\end{equation}
	
	According to the Riemann-Lebesgue lemma, \eqref{R-L} is evident. Using the notations from Remark \ref{CPC}, and letting $\T_1=\sum_{j=1}^k \Z_j$, we next provide the proof which asserts that the validity of the Edgeworth expansion for the function of sample means of vector variables when $\T_1$ satisfies GPCC \eqref{Order 1}. 
	
	Based on the definition of $W_n$ and $\mathcal{Q}_n$, we can obtain the following relationship:
	\begin{equation}
		\mathcal{Q}_n(x)=P(W_n\le x)=P(\sqrt n(H(\bar \Z)-H(\bm\mu))\le x).
	\end{equation}
	Note that
	\[W_n=\sqrt n(H(\bar \Z)-H(\bm\mu))=\sqrt n(H(\bar \T)-H(\bm\mu)),\]
	where
	\[\bar \Z=\frac 1n \sum_{i=1}^n \Z_i=\frac 1d\sum_{i=1}^d \T_i=\bar \T,\qquad d=n/k, \qquad \bm\mu=E\Z_n.\]
	Therefore, we obtain
	\begin{equation}
		\mathcal{Q}_n(x)=P(W_n\le x)=P(\sqrt n(H(\bar \T)-H(\bm\mu))\le x).
	\end{equation}
	Therefore, it suffices to prove that $\mathbb{E}|\T_1|^s<\infty$. We will demonstrate this in the following lemma.
	
	\begin{lemma}\label{moment}
		Assume that $\Z_{1}$ has finite $s$-th absolute moment for $j=1, 2,\dots, k$, where $m\ge 3$ is a known integer, then $\T_{1}$ has finite $s$-th absolute moment for $j=1, 2,\dots, k$.
	\end{lemma}
	\begin{proof}
		Assume $\rho_s=\mathbb{E}|\Z_1|^s<\infty$, then by the relationship of $\T_1$ and $\Z_1$, we can derive
		\begin{align*}
			\mathbb{E}|\T_1|^s&=\mathbb{E}\Big|\frac 1k \big(\Z_{1}+\dots+\Z_{k}\big)\Big|^s=\frac{1}{k^s}\mathbb{E}|\Z_{1}+\dots+\Z_{k}|^s\\
			&\le\frac {1}{k^s}* k^{s-1} (\mathbb{E}|\Z_1|^s+\dots+\mathbb{E}|\Z_k|^s)\\
			&=\frac 1k(\mathbb{E}|\Z_1|^s+\dots+\mathbb{E}|\Z_k|^s)\\
			&<\infty.
		\end{align*}
		Therefore, the lemma is completed.
	\end{proof}
	Using Lemma \ref{moment}, we can verify all the conditions necessary for Theorem \ref{main1}. Therefore, we can obtain the conclusion of Theorem \ref{main3}.

	\subsection{Proofs of examples}
	In this subsection,we will demonstrate the validity of the examples mentioned earlier in the article.
	
	\begin{proof}[Proof of Example \ref{3-dimension}] Observing that $X$ and $Y$ are independent, with $X$ following a continuous distribution and $Y$ following a discrete distribution, we can derive the following result:
		\begin{align*}
			\limsup _{\|\mathbf{ t}_2\|\to\infty}\mathbb{E}\Big|\mathbb{E}\Big(\exp[i (t_{1}X+t_{2}X^2)]\Big| Y\Big)\Big|=\limsup _{\|\mathbf{ t}_2\|\to\infty}\Big|\mathbb{E}\Big(\exp[i (t_{1}X+t_{2}X^2)]\Big)\Big|<1.
		\end{align*}
		The last inequality holds due to Theorem \ref{main3} mentioned above. 
	\end{proof}
	
	\begin{proof}[Proof of Example \ref{5-dimension}] Observing that $X$ and $Y$ are independent, with $X$ following a continuous distribution and $Y$ following a discrete distribution, we can derive the following result:
		\begin{align*}
			&\quad\limsup _{\|\mathbf{ t}_3\|\to\infty}\mathbb{E}\Big|\mathbb{E}\Big(\exp[i (t_{1}X+t_{2}X^2+t_{3}XY)]\Big| Y,Y^2\Big)\Big|\\
			&=\limsup _{\|\mathbf{ t}_3\|\to\infty}\Big|\mathbb{E}\Big(\exp[i (t_{1}X+t_{2}X^2+t_{3}y_0X)]\Big)\Big|<1.
		\end{align*}
		The last inequality holds due to Theorem \ref{main3} mentioned above. 
	\end{proof}
	
	\begin{proof}[Proof of Examples \ref{square}, \ref{Ex1} and \ref{logdata}]
		Since $x^2$ and $\log x$ are the first-order differentiable functions, the condition in Theorem \ref{main3} holds. Therefore, we can infer that $(w,w^2)$ and $(w,\log w)$ satisfy the GPCC.
	\end{proof}

	\section{Calculations}\label{appdx2}

	\subsection{Second order Edgeworth expansion for $W_n$}
	In this section, we outline a comprehensive calculation for the second-order correction term. Additionally, we delve into detailed computations for each coefficient referenced in Remark \ref{second-term}. First, we begin by calculating several important moments:
	\begin{align*}
		&\mathbb{E}(Z_{i_1}Z_{i_2})=\mu_{i_1i_2},\quad \mathbb{E}(Z_{i_1}Z_{i_2}Z_{i_3})=n^{-1/2}\mu_{i_1i_2i_3}, \\
		&\mathbb{E}(Z_{i_1}Z_{i_2}Z_{i_3}Z_{i_4})=n^{-1}\mu_{i_1i_2i_3i_4}+U_1,\\
		&\mathbb{E}(Z_{i_1}Z_{i_2}Z_{i_3}Z_{i_4}Z_{i_5})=n^{-1/2}U_2+O(n^{-3/2}),\\
		&\mathbb{E}(Z_{i_1}Z_{i_2}Z_{i_3}Z_{i_4}Z_{i_5}Z_{i_6})=n^{-1}U_4+n^{-1}U_5+U_3+O(n^{-3/2}),
	\end{align*}
	where $U_1, U_2, U_3$ retain the same meanings as mentioned above. Additionally, $U_i( i=4,5)$ in the above expressions refer to specific operations on population moments, which are defined as follows.
	\begin{align*}
		U_4&=\mu_{i_1i_2}\mu_{i_3i_4i_5i_6}+\mu_{i_1i_3}\mu_{i_2i_4i_5i_6}+\mu_{i_1i_4}\mu_{i_2i_3i_5i_6}+\mu_{i_1i_5}\mu_{i_2i_3i_4i_6}+\mu_{i_1i_6}\mu_{i_2i_3i_4i_5}\\
		&\quad+\mu_{i_2i_3}\mu_{i_1i_4i_5i_6}+\mu_{i_2i_4}\mu_{i_1i_3i_5i_6}+\mu_{i_2i_5}\mu_{i_1i_3i_4i_6}+\mu_{i_2i_6}\mu_{i_1i_3i_4i_5}+\mu_{i_3i_4}\mu_{i_1i_2i_5i_6}\\
		&\quad+\mu_{i_3i_5}\mu_{i_1i_2i_4i_6}+\mu_{i_3i_6}\mu_{i_1i_2i_4i_5}+\mu_{i_4i_5}\mu_{i_1i_2i_3i_6}+\mu_{i_4i_6}\mu_{i_1i_2i_3i_5}+\mu_{i_5i_6}\mu_{i_1i_2i_3i_4}\\
		U_5&=\mu_{i_1i_2i_3}\mu_{i_4i_5i_6}+\mu_{i_1i_2i_4}\mu_{i_3i_5i_6}+\mu_{i_1i_2i_5}\mu_{i_3i_4i_6}+\mu_{i_1i_2i_6}\mu_{i_3i_4i_5}+\mu_{i_1i_3i_4}\mu_{i_2i_5i_6}\\
		&\quad+\mu_{i_1i_3i_5}\mu_{i_2i_4i_6}+\mu_{i_1i_3i_6}\mu_{i_2i_4i_5}+\mu_{i_1i_4i_5}\mu_{i_2i_3i_6}+\mu_{i_1i_4i_6}\mu_{i_2i_3i_5}+\mu_{i_1i_5i_6}\mu_{i_2i_3i_4}
	\end{align*}

	Recall the Taylor expansion of $W_n$ and obtain its form when $s=4$. Specifically, $W_n$ can be expressed as follows:
	\[W_n=\sum_{1}l_iZ_i+n^{-1/2}\frac 12\sum_{2}l _{ij}Z_iZ_j+n^{-1}\frac{1}{3!}\sum_{3}l_{i_1i_2i_3}Z_{i_1}Z_{i_2}Z_{i_3}+O_p(n^{-3/2}).\]
	
	Next, we can obtain the first fouth-order moments of $W_n$ through a series of calculations. The specific form is shown below.
	\begin{align*}
		\mathbb{E}(W_n)&=n^{-1/2}\frac 12\sum_{2}l_{i_1i_2}\mu_{i_1i_2}+O(n^{-3/2}),\\
		\mathbb{E}(W_n^2)&=\sum_{2} l_{i_1}l_{i_2}\mu_{i_1i_2}+n^{-1}\sum_{3}l_{i_1}l_{i_2i_3}\mu_{i_1i_2i_3}+\frac 14n^{-1}\sum_{4} l_{i_1i_2}l_{i_3i_4}U_1\\
		&\quad+\frac 13n^{-1}\sum_{4}l_{i_1}l_{i_2i_3i_4} U_1+O(n^{-3/2}),\\
		\mathbb{E}(W_n^3)&=n^{-1/2}\sum_{3}l_{i_1}l_{i_2}l_{i_3}\mu_{i_1i_2i_3}+\frac 32n^{-1/2}\sum_{4}l_{i_1}l_{i_2}l_{i_3i_4}U_1+O(n^{-3/2}),\\
		\mathbb{E}(W_n^4)&=n^{-1}\sum_{4}l_{i_1}l_{i_2}l_{i_3}l_{i_4}\mu_{i_1i_2i_3i_4}+\sum_{4}l_{i_1}l_{i_2}l_{i_3}l_{i_4}U_1+2n^{-1}\sum_{5}l_{i_1}l_{i_2}l_{i_3}l_{i_4i_5}U_2\\
		&\quad+\frac 23 n^{-1}\sum_{6}l_{i_1}l_{i_2}l_{i_3}l_{i_4i_5i_6}U_3+\frac 32 n^{-1}\sum_{6}l_{i_1}l_{i_2}l_{i_3i_4}l_{i_5i_6}U_3+O(n^{-3/2}).
	\end{align*}
	
	Building on the relationship between moments and cumulants, we deduce the first four cumulants of $W_n$. Specifically, $B_1$ is the coefficient of the first-order cumulant with respect to $n^{-1/2}$, $B_2$ is the coefficient of the second-order cumulant with respect to $n^{-1}$, $B_3$ is the coefficient of the third-order cumulant with respect to $n^{-1/2}$, and $B_4$ is the coefficient of the fourth-order cumulant with respect to $n^{-1}$. Subsequently, we derive detailed expressions for each coefficient.
	
	\subsection{Some useful expressions}\label{B.2} In this section, we provide specific expressions for the symbols that have been utilized in Remark \ref{second-term} and Corollary \ref{correlation}.
	\begin{align*}
		A_1&=\frac 12 \sum_{i,j=1}^k l_{ij}\mu_{ij},~ B_1=\frac 12\sum_{2}l_{ij}\mu_{ij}, ~A_2=\sum_{i,j,k=1}^k l_i l_j l_m \mu_{ijm}+3\sum_{i,j,k,l=1}^k l_i l_j l_{ml} \mu_{im} \mu_{jl},\\
		B_2&=\sum_{3}l_{i_1}l_{i_2i_3}\mu_{i_1i_2i_3}+\sum_{4}\Big(\frac 14l_{i_1i_2}l_{i_3i_4}U_1+\frac 13 l_{i_1}l_{i_2i_3i_4}U_1-\frac 14 l_{i_1i_2}l_{i_3i_4}\mu_{i_1i_2}\mu_{i_3i_4}\Big),\\
		B_3&=\sum_{3}l_{i_1}l_{i_2}l_{i_3}\mu_{i_1i_2i_3}+\sum_{4}\Big(\frac 32l_{i_1}l_{i_2}l_{i_3i_4}U_1-\frac 32 l_{i_1}l_{i_2}l_{i_3i_4}\mu_{i_1i_2}\mu_{i_3i_4}\Big),\\
		B_4&=\sum_{4}l_{i_1}l_{i_2}l_{i_3}l_{i_4}\mu_{i_1i_2i_3i_4}+2\sum_{5}\Big(l_{i_1}l_{i_2}l_{i_3}l_{i_4i_5}U_2-l_{i_1}l_{i_2}l_{i_3}l_{i_4i_5} \mu_{i_1i_2i_3}\mu_{i_4i_5}\Big)+\sum_{6}\Big(\frac23 l_{i_1}l_{i_2}\\
		&\quad l_{i_3}l_{i_4i_5i_6}U_3+\frac32 l_{i_1}l_{i_2}l_{i_3i_4}l_{i_5i_6}U_3-3l_{i_1}l_{i_2}l_{i_3i_4}l_{i_5i_6}\mu_{i_5i_6}U_1+3 l_{i_1i_2}l_{i_3i_4}l_{i_5}l_{i_6}\mu_{i_1i_2}\mu_{i_3i_4}\mu_{i_5i_6}\Big).
	\end{align*}
	The summation symbols $\sum_i(i=1,\dots,6)$ in the above expressions represent the summation over $i$ subscripts, with each subscript ranging from $1$ to $k$. Additionally, $U_i( i=1,2,3)$ in the above expressions refer to specific operations on population moments, which are defined as follows.
	\begin{align*}
		U_1&=\mu_{i_1i_2}\mu_{i_3i_4}+\mu_{i_1i_3}\mu_{i_2i_4}+\mu_{i_2i_3}
		\mu_{i_1i_4},\\
		U_2&=\mu_{i_1i_2}\mu_{i_3i_4i_5}+\mu_{i_1i_3}\mu_{i_2i_4i_5}+\mu_{i_1i_4}\mu_{i_2i_3i_5}+\mu_{i_1i_5}\mu_{i_2i_3i_4}+\mu_{i_2i_3}\mu_{i_1i_4i_5}+\mu_{i_2i_4}\mu_{i_1i_3i_5}\\
		&\quad+\mu_{i_2i_5}\mu_{i_1i_3i_4}+\mu_{i_3i_4}\mu_{i_1i_2i_5}+\mu_{i_3i_5}\mu_{i_1i_2i_4}+\mu_{i_4i_5}\mu_{i_1i_2i_3},\\
		U_3&=\mu_{i_1i_2}(\mu_{i_3i_4}\mu_{i_5i_6}+\mu_{i_3i_5}\mu_{i_4i_6}+\mu_{i_3i_6}\mu_{i_4i_5})+\mu_{i_1i_3}(\mu_{i_2i_4}\mu_{i_5i_6}+\mu_{i_2i_5}\mu_{i_4i_6}+\mu_{i_2i_6}\mu_{i_4i_5})\\
		&\quad+\mu_{i_1i_4}(\mu_{i_2i_3}\mu_{i_5i_6}+\mu_{i_2i_5}\mu_{i_3i_6}+\mu_{i_2i_6}\mu_{i_3i_5})+\mu_{i_1i_5}(\mu_{i_2i_3}\mu_{i_4i_6}+\mu_{i_2i_4}\mu_{i_3i_6}\\
		&\quad+\mu_{i_2i_6}\mu_{i_3i_4})+\mu_{i_1i_6}(\mu_{i_2i_3}\mu_{i_4i_5}+\mu_{i_2i_4}\mu_{i_3i_5}+\mu_{i_2i_5}\mu_{i_3i_4}).
	\end{align*}
	
	The above definitions are utilized in Remark \ref{second-term}. Additionally, the following definitions are used in  Corollary \ref{correlation}.
	\begin{align*}
		&A_3=\frac 12 \sum_{i,j=1}^k l_{ij}\tilde\mu_{ij},\quad A_4=\sum_{i,j,k=1}^k l_i l_j l_m \tilde\mu_{ijm}+3\sum_{i,j,k,l=1}^k l_i l_j l_{ml} \tilde\mu_{im} \tilde\mu_{jl}\\
		&\tilde\mu_{11}=\mathbb{E}(Y_{11}-\mathbb{E}Y_{11})(Y_{11}-\mathbb{E}Y_{11}),\quad \tilde\mu_{12}=\mathbb{E}(Y_{12}-\mathbb{E}Y_{12})(Y_{12}-\mathbb{E}Y_{12}),\\
		&\tilde\mu_{13}=\mathbb{E}(Y^2_{11}-\mathbb{E}Y^2_{11})(Y^2_{11}-\mathbb{E}Y^2_{11}),\quad\tilde\mu_{14}=\mathbb{E}(Y^2_{12}-\mathbb{E}Y^2_{12})(Y^2_{12}-\mathbb{E}Y^2_{12}),\\
		&\tilde\mu_{15}=\mathbb{E}(Y_{11}Y_{12}-\mathbb{E}Y_{11}Y_{12})(Y_{11}Y_{12}-\mathbb{E}Y_{11}Y_{12})\dots
	\end{align*}

		\section{Auxiliary lemmas}\label{appdx3}
	This appendix collects several auxiliary results that were used in the preceding arguments.

	\begin{lemma}[Lemma 14.1 of Bhattacharya and Rao (2010)]\label{lemma: 14.1}
		Assume that $\rho_s<\infty$ for some $s\ge 2$. Define truncated random vectors
		\begin{equation*}
			Y_{j}=\left\{
			\begin{aligned}
				X_j &, &\|X_j\|\le n^{\frac 12},\\
				0 &, &\|X_j\|>n^{\frac12},
			\end{aligned}
			\right.
			\quad  Z_{j}=Y_{j}-\mathbb{E}Z_{j}\quad (1\le j\le n).
		\end{equation*}
		\begin{itemize}
			\item[(i)] One has \[\rho_{sj}=\mathbb{E}\Vert Y_j\Vert^s+\Delta_{n,j,s},\quad \bar\Delta_{n,s}\le\rho_s.\]
			\item[(ii)] If $\alpha$ is a nonnegative integral vector satisfying $1\le \vert\alpha\vert\le s$, then 
			\[\vert \mathbb{E}X_j^{\alpha}-\mathbb{E}Y_j^{\alpha}\vert\le n^{-(s-\vert\alpha\vert)/2}\Delta_{n,j,s},\]
			\[\vert \mathbb{E}Y_j^{\alpha}-\mathbb{E}Z_j^{\alpha}\vert \le \vert\alpha\vert(2^{\vert\alpha\vert}+1)n^{-(s-\vert\alpha\vert)/2}\Delta_{n,j,s}.\]
			\item[(iii)] One has
			\[\vert d_{il}-v_{il}\vert\le 2n^{-(s-2)/2}\bar\Delta_{n,s}\quad (1\le i, ~ l\le k).\]
			\item[(iv)] If $2\le s'\le s$, then
			\[\mathbb{E}\Vert Y_j\Vert^{s'}\le \rho_{s',j},\quad \rho'_{s',j}=\mathbb{E}\Vert Z_{j}\Vert^{s'}\le 2^{s'}\rho_{s',j},\quad \rho'_{s'}\le 2^{s'}\rho_{s'}.\]
			\item[(v)] If $s'>s$, then
			\begin{align*}
				\mathbb{E}\Vert Y_j\Vert^{s'}\le& (\epsilon n^{1/2})^{s'-s}\int_{\{\Vert X_j\Vert\le \epsilon^{1/2}\}}\Vert X_j\Vert^s\\
				&+n^{(s'-s)/2}\int_{\{\epsilon n^{1/2}<\Vert X_j\Vert\le n^{1/2}\}}\Vert X_j\Vert^s\le n^{(s'-s)/2}\rho_{s,j}\qquad (0\le\epsilon\le 1),
			\end{align*}
			\[\rho'_{s',j}=\mathbb{E}\Vert Z_j\Vert^{s'}\le 2^{s'}\mathbb{E}\Vert Y_j\Vert^{s'}.\]
		\end{itemize} 
	\end{lemma}

	\begin{lemma}[Lemma 14.6 of Bhattacharya and Rao (2010)]\label{lemma: 14.6}
		Assume that $V=I$, and $\bar \Delta_{n,s}\le \frac{n^{-(s-2)/2}}{8k}$ holds for some $s\ge 3$. Then for every integer r, $0\le r\le s-2$, one has
		\begin{align*}
			n^{-r/2}|P_r(-&\phi:\{\chi_\bfv\})(x)-P_r(-\phi_{0, D}:\{\chi'_v\})(x)|\\
			&\le c \bar\Delta_{n,s}n^{-(s-2)/2}(1+\|x\|^{3r+2})\exp\Big(-\frac{\|x\|^2}{6}+\|x\|\Big),\quad (x\in R^k),
		\end{align*}
		and
		\begin{align*}
			n^{-r/2}|P_r(-&\phi:\{\chi_\bfv\})(x+a_n)-P_r(-\phi:\{\chi_\bfv\})(x)|\\
			&\le c \bar\Delta_{n,s}n^{-(s-2)/2}(1+\|x\|^{3r+1})\exp\Big(-\frac{\|x\|^2}{2}+\frac{\|x\|}{8k^{1/2}}\Big),\quad (x\in R^k),
		\end{align*}
		where
		\[a_n=n^{-1/2}\sum_{j=1}^n \mathbb{E}Y_j.\]
	\end{lemma}

	\begin{lemma}[Lemma 11.6 of Bhattacharya and Rao (2010)]\label{lemma:11.6}
		Let $g$ be a real-valued function in $L^1(R^k)$ satisfying
		\[\int\|x\|^{k+1}|g(x)|dx<\infty.\]
		Then there exists a positive constant $c(k)$ depending only on $k$ (and not on $g$) such that
		\[\|g\|_1\le c(k)\max_{|\beta|=0,k+1}\int|D^{\beta}\hat g(t)|dt.\]
	\end{lemma}

	\begin{lemma}[Lemma 6.2 of Bhattacharya and Rao (2010)] \label{lemma: 7.2}
		Let $X$ be a random vector in $R^k$ having a finite $s$-th absolute moment $\rho_s$ for some positive $s$. If $X$ is not degenerate at $0$,
		\begin{itemize}
			\item[(i)] $r\to\log\rho_r$ is a convex function on $[0,s]$.
			\item[(ii)] $r\to\rho_r^{1/r}$ is nondecreasing on $[0,s]$.
			\item[(iii)] $r\to(\rho_r/\rho_2^{r/2})^{1/(r-2)}$ is nondecreasing on $(2,s]$ if $s>2$.
		\end{itemize}
	\end{lemma}

	\begin{lemma}[Corollary 14.4 of Bhattacharya and Rao (2010)] \label{lemma: 14.4}
		Suppose $\rho_s<\infty$ for some $s\ge 3$, and that $V=I$. Let $g_j$ denote the characteristic function of $Z_j (1\le j\le n)$. Then if
		\[\|t\|\le \frac{n^{1/2}}{16\rho_3},\qquad \bar\Delta_{n,s}\le\frac{n^{(s-2)/2}}{8k},\]
		then,
		\[\bigg|\bigg(D^{\alpha}\prod_{j=1}^ng_j\bigg)\bigg(\frac{t}{n^{1/2}}\bigg)\bigg|\le c_1(\alpha, k)(1+\|t\|^{|\alpha|})\exp\Big(-\frac{5}{24}\|t\|^2\Big).\]
	\end{lemma}
	
	\begin{lemma}[Lemma 6.3 of Bhattacharya and Rao (2010)] \label{lemma: 6.3}
		Let $X$ be a random vector in $R^k$ having a finite $s$-th absolute moment $\rho_s$ for some positive integer $s$. Then for nonnegative integral vectors $v$ satisfying $|v|<s$,
		\[|\mu_v|\le\mathbb{E}|X^v|\le\rho_{|v|},\]
		and there exists a constant $c_1(v)$ depending only on $v$ such that
		\[|\chi_v|\le c_1(v)\rho_{|v|}.\]
	\end{lemma}
	
	\begin{lemma}[Corollary 11.2 of Bhattacharya and Rao (2010)]\label{lemma:11.2}
		Let $\mu$ be a finite measure and $\nu$ a finite signed measure on $R^k$. Let $\epsilon$ be a positive number and $K_{\epsilon}$ a probability measure on $R^k$ satisfying
		\[K_{\epsilon}(B(0:\epsilon))=1.\]
		Then for every real-valued, Borel-measurable function f on $R^k$ that is bounded on compacts,
		\[\Big|\int f d(\mu-\nu)\Big|\le \gamma(f:\epsilon)+\int\omega_f(\cdot : 2\epsilon)d\nu^+,\]
		where
		\[\gamma(f:\epsilon)=max\Big\{\int M_f(\cdot:\epsilon)d(\mu-\nu)*K_{\epsilon}, \int m_f(\cdot:\epsilon)d(\mu-\nu)*K_{\epsilon}\Big\}\]
		provided that $|M_f(\cdot : 2\epsilon)|$ and $|m_f(\cdot : 2\epsilon)|$ are integrable with respect to $\mu$ and $|\nu|$. If, in addition, f is bounded and 
		\[\mu(R^k)=\nu(R^k),\]
		then
		\[\Big|\int f d(\mu-\nu)\Big| \le\frac12\omega_f(R^k)\|(\mu-\nu)*K_{\epsilon}\|+\int  \omega_f(\cdot:2\epsilon)d\nu^+.\]
	\end{lemma}
	
	\begin{lemma}[Lemma 11.6 of Bhattacharya and Rao (2010)]\label{lemma:11.6}
		Let $g$ ba a real-valued function in $L^1(R^k)$ satisfying
		\[\int\|x\|^{k+1}|g(x)|dx<\infty.\]
		Then there exists a positive constant $c(k)$ depending only on $k$ (and not on $g$) such that
		\[\|g\|_1\le c(k)\max_{|\beta|=0,k+1}\int|D^{\beta}\hat g(t)|dt.\]
	\end{lemma}

	\begin{lemma}[Theorem 9.10 of Bhattacharya and Rao (2010)]\label{lemma: 9.10}
		Let $G$ be a probability measure on $R^k$ with zero mean, positive-definite covariance matrix $V$, and finite $s$-th absolute moment for some integer $s$ not smaller than $3$. Then there exist two positive constants $c_1, c_2$ such that for all $t$ in $R^k$ satisfying
		\[\|t\|\le c_1n^{1/2}\eta_s^{-1/(s-2)}\]
		one has, for all nonnegative integral vectors $\alpha$, $0\le|\alpha|\le s$,
		\begin{align*}
			\Big|D^{\alpha}\Big[\hat G^n(\frac{Bt}{n^{1/2}})-&\exp\{-\frac 12\|t\|^2\}\sum_{r=0}^{s-3}n^{-r/2}\tilde P_r(iBt:\{\chi_\bfv\})\Big]\Big|\\
			&\le\frac{c_2\eta_s}{n^{-(s-2)/2}}[\|t\|^{s-|\alpha|}+\|t\|^{3(s-2)+|\alpha|}]\exp\{-\frac{\|t\|^2}{4}\}.
		\end{align*}
		Where $\chi_\bfv$ is the vth cumulant of G, and $\eta_s=\int\|Bx\|^sG(dx)$. Here $B$ is the symmetric positive-definite matrix satisfying $B^2=V^{-1}$.
	\end{lemma}

	\begin{lemma}[Lemma 2.1 of Bhattacharya and Ghosh (1978)]\label{lemma: 2.1}
		Assume $\rho_s=\mathbb{E}|Z_1|^s<\infty$ and that all derivatives of $H$ of orders s and less are continuous in a neighborhood of $\mu=EZ_1$, for some $s\ge 3$. Then there exist polynomials $q_r$ (in one variable), whose coefficients do not depend on n, such that uniformly over all Borel subsets B of $R^1$ one has
		\[\int_{\{g_n(z)\in B\}}\xi_{s,n}(z) dz=\int_B dF_n(u)+o(n^{-(s-2)/2}),\]
		where
		\[F_n(u)=\int_{-\infty}^u\Big[1+\sum_{r=1}^{s-2}n^{-r/2}q_r(v)\Big]\phi_{\sigma^2}(v)dv\quad u\in R^1.\]
		Also, for all nonnegative integers j
		\[\int_{M_n}g_n^j(z)\xi_{s,n}(z)dz=\int_{-\infty}^{\infty}u^jdF_n(u)+o(n^{-(s-2)/2}),\]
		\[\int_{R_k}h_{s-1}^j(z)\xi_{s,n}(z)dz=\int_{-\infty}^{\infty}u^jdF_n(u)+o(n^{-(s-2)/2}).\]
	\end{lemma}  
	
	\begin{lemma}\label{An}
		Let $Z_{1,n}$, $A_n$ and $c_n$ be defined as in Section 
		\ref{sec6}. Additionally, let $C(s,k)$ be an absolute constant. Then for sufficiently large $n$, we have
		\[c_n\ge A_n.\]
	\end{lemma}
	\begin{proof}
		Based on the definition of $T_n$, we can derive:
		\[A_n=\frac{C(s,k)n^{1/2}}{(E\|T_nZ_{1n}\|^{s+k+1})^{1/(s+k-1)}}=\frac{C(s,k)n^{1/2}}{(E\|Z_{1n}\|^{s+k+1})^{1/(s+k-1)}}.\]
		When $n$ is sufficiently large,
		\[\frac{A_n}{c_n}=\frac{o(n^{(s-2)/2(s+k-1)})}{O(n^{1/2})}<1.\]
		Therefore, we obtain the result.
	\end{proof}

%
%
%
%
%
%
%
%
%
%
%
%
%
%

\end{document}